\documentclass[11pt,reqno]{amsart}

\usepackage{amssymb,amsthm,amsmath,amscd,caption,float,extarrows}
\usepackage[backref=page, colorlinks=true, linkcolor=magenta, citecolor=cyan]{hyperref}
\usepackage{mathtools}
\usepackage{soul}
\usepackage{times}

\RequirePackage[dvipsnames,usenames]{color}

\input{kmacros3.sty}
\usepackage{quiver}

\usepackage{tikz}
\usepackage{tikz-cd}
\usetikzlibrary{cd}
\usepackage{graphicx}
\usepackage[all,cmtip]{xy}

\usepackage{mabliautoref}
\numberwithin{equation}{theorem}

\usepackage{bm}
\usepackage{pifont}
\usepackage{upgreek}

\usepackage{eucal}
\usepackage{ulem}
\usepackage{stmaryrd}

\numberwithin{equation}{theorem}

\usepackage[left=1in,top=1in,right=1in,bottom=1in]{geometry}

\usepackage{setspace}
\spacing{1.05}

\usepackage{enumitem}
\usepackage{calc}

\usepackage{verbatim}
\usepackage{alltt}

\newcommand{\p}{{\mathfrak p}}
\newcommand{\fm}{{\mathfrak m}}

\DeclareMathOperator{\ev}{eval}

\DeclareMathOperator{\Pure}{Pure}
\DeclareMathOperator{\CPure}{CPure}
\DeclareMathOperator{\rad}{rad}
\newtheoremstyle{cited}{.5\baselineskip\@plus.2\baselineskip\@minus.2\baselineskip}{.5\baselineskip\@plus.2\baselineskip\@minus.2\baselineskip}{\itshape}{}{\bfseries}{\bfseries .}{5pt plus 1pt minus 1pt}{\thmname{#1}\thmnumber{ #2}\thmnote{ \normalfont#3}}

\theoremstyle{cited}
\newtheorem{citedthm}[theorem]{Theorem}

\newtheoremstyle{citeddef}{.5\baselineskip\@plus.2\baselineskip\@minus.2\baselineskip}{.5\baselineskip\@plus.2\baselineskip\@minus.2\baselineskip}{}{}{\bfseries}{\bfseries .}{5pt plus 1pt minus 1pt}{\thmname{#1}\thmnumber{ #2}\thmnote{ \normalfont#3}}
\theoremstyle{citeddef}

\newcommand{\xrightarrowdbl}[2][]{%
  \xrightarrow[#1]{#2}\mathrel{\mkern-14mu}\rightarrow
}

\makeatletter
\def\@tocline#1#2#3#4#5#6#7{\relax
  \ifnum #1>\c@tocdepth 
  \else
    \par \addpenalty\@secpenalty\addvspace{#2}%
    \begingroup \hyphenpenalty\@M
    \@ifempty{#4}{%
      \@tempdima\csname r@tocindent\number#1\endcsname\relax
    }{%
      \@tempdima#4\relax
    }%
    \parindent\z@ \leftskip#3\relax \advance\leftskip\@tempdima\relax
    \rightskip\@pnumwidth plus4em \parfillskip-\@pnumwidth
    #5\leavevmode\hskip-\@tempdima
      \ifcase #1
       \or\or \hskip 1em \or \hskip 2em \else \hskip 3em \fi%
      #6\nobreak\relax
    \hfill\hbox to\@pnumwidth{\@tocpagenum{#7}}\par
    \nobreak
    \endgroup
  \fi}

  \newenvironment{dedication}
  {
   \thispagestyle{empty}
   \itshape             
   \center          
  }
  {\par 
   \vspace{\stretch{2}} 
  }
\makeatother

\begin{document}

\title{Mittag-Leffler modules and variants on intersection flatness}
\author{Rankeya Datta}
\address{Department of Mathematics, University of Missouri, Columbia, MO 65212 USA}
\email{rankeya.datta@missouri.edu}
\thanks{Datta was partially supported by an AMS-Simons travel grant and a grant from the Simons Foundation MP-TSM-00002400}
\author{Neil Epstein}
\address{Department of Mathematical Sciences, George Mason University, Fairfax, VA 22030, USA}
\email{nepstei2@gmu.edu}
\author{Kevin Tucker}
\address{Department of Mathematics, University of Illinois at Chicago, Chicago, IL 60607, USA}
\email{kftucker@uic.edu}
\thanks{Tucker was supported in part by NSF grant DMS-2200716}

\begin{abstract}
    We systematically study the intersection flatness and Ohm-Rush properties for modules over a commutative ring, drawing inspiration from the work of Ohm and Rush and of Hochster and Jeffries.  We establish new structural results for modules that are intersection flat/Ohm-Rush by exhibiting intimate connections between these notions and the seminal work of Raynaud and Gruson on Mittag-Leffler modules. In particular, we develop a theory of Ohm-Rush modules that is parallel to the theory of Mittag-Leffler modules. 
    We also obtain descent and local-to-global results for intersection flat/Ohm-Rush modules. Our investigations reveal a particularly pleasing picture for flat modules over a complete local ring, in which case many otherwise distinct properties coincide. 
\end{abstract}

\maketitle

\begin{dedication}
    Dedicated to the memory of David Rush.
\end{dedication}

\tableofcontents

\vspace{-4ex}
\section{Introduction}

Let $R$ be a commutative ring and $M$ be a \emph{flat} $R$-module. For ideals $I, J$ of $R$, the exact sequence
\[
0 \to R/I\cap J \to R/I \oplus R/J      
\]
gives rise to an exact sequence $0 \to M/(I \cap J)M \to M/IM \oplus M/JM$. Taking the kernel, we get $(I \cap J)M = IM \cap JM$. In other words, expansions of ideals to a module commutes with finite intersections for flat modules. Replacing a finite collection of ideals by a collection $\{I_\alpha \colon \alpha \in A\}$ indexed by an arbitrary set $A$, one can naturally ask when 
\begin{equation}
    \label{eq:expansion-commutes-contraction}
\bigcap_{\alpha \in A} I_\alpha M = \left(\bigcap_{\alpha \in A} I_\alpha\right)M?    
\end{equation}

This question was first studied systematically by Ohm and Rush \cite{OhmRu-content}. They observed that for any $R$-module $M$, one can always define a function
\[
\textrm{$c_M \colon M \to \{$ideals of $R\}$},   
\]
called the \emph{content of $M$}, where for $x \in M$, $c_M(x)$ is the intersection of all ideals $I$ of $R$ such that $x \in IM$. The reason for using the term `content' for this function is that when $M = R[x]$ is the polynomial ring, then for $f \in R[x]$, $c_{R[x]}(f)$ is the usual content of $f$, namely the ideal generated by the nonzero coefficients of $f$. Ohm and Rush observed that \autoref{eq:expansion-commutes-contraction} always holds precisely when for all $x \in M$, $x \in c_M(x)M$. That is, \autoref{eq:expansion-commutes-contraction} holds if and only if for any element of $x \in M$, there is a \emph{unique smallest ideal} $I$ of $R$ such that $x \in IM$. In honor of Ohm and Rush, and following \cite{EpsteinShapiroOhmRushI,EpsteinShapiroOhmRushII,EpsteinShapiroOhmRushIII}, we will call modules that satisfy \autoref{eq:expansion-commutes-contraction} \emph{Ohm-Rush} in this paper (\autoref{def:Ohm-Rush}), although they were called \emph{content modules} in \cite{OhmRu-content}. 

There is also a natural generalization of \autoref{eq:expansion-commutes-contraction} to submodules of modules. For an $R$-module $L$ and a submodule $U$ of $L$, let $UM$ denote the image of the canonical map $U \otimes_R M \to L \otimes_R M$. Note that when $L = R$ and $U$ is an ideal of $R$, then $UM$ is naturally identified with the expansion of $U$ to $M$. If $L$ is finitely generated and $\{U_\alpha \colon \alpha \in A\}$ is a family of submodules of $L$, one can study the equality
\begin{equation}
    \label{eq:intersection-flat}
    \bigcap_{\alpha \in A} U_\alpha M = \left(\bigcap_{\alpha \in A} U_\alpha\right)M
\end{equation}
in the module $L \otimes_R M$. We call $M$ \emph{intersection flat} (\autoref{def:IF}) following \cite{
HochsterHunekeFRegularityTestElementsBaseChange,HochsterJeffriesintflatness} if \autoref{eq:intersection-flat} holds for all families of submodules $\{U_\alpha \colon \alpha \in A\}$ of all finitely generated $R$-modules $L$. It is clear that modules that are intersection flat are Ohm-Rush, but the converse fails in general (see \autoref{eg:OR-not-preserved-base-change-polynomial}). 

Our interest in the intersection flatness and Ohm-Rush properties originated from the study of singularities via the Frobenius map. Recall Kunz's celebrated theorem \cite{KunzCharacterizationsOfRegularLocalRings} which states that a Noetherian ring $R$ of prime characteristic $p > 0$ is regular precisely when the Frobenius $F \colon R \to R$ or $p$-th power map $r \mapsto r^p$
is flat. Denoting the target copy of $R$ with module structure induced by restriction of scalars along Frobenius as $F_*R$, Hochster and Huneke studied when $F_*R$ is an intersection flat $R$-module in \cite{HochsterHunekeFRegularityTestElementsBaseChange}. 
They encountered this condition in their quest to show the existence of certain uniform multipliers in tight closure known as \emph{test elements}. The \emph{test ideal} generated by the test elements has found many applications and is mysteriously related to multiplier ideals from complex algebraic geometry \cite{SmithTestMultiplierIdeal, HaraTestMultiplierIdeal}. Nevertheless, it remains a long-standing open conjecture that prime characteristic excellent reduced rings admit test elements. 

Let $S$ be a regular ring of prime characteristic $p > 0$. It was observed by Sharp \cite{SharpBigTestElements} that the intersection flatness of $F_*S$ implies the existence of test elements for reduced homomorphic images of $S$. The following are the known results about intersection flatness and Ohm-Rush properties of $F_*S$.
\begin{enumerate}
    \item[(1)] $F_*S$ was to shown to be an intersection flat $S$-module when $S$ is complete regular local in \cite{KatzmanParameterTestIdealOfCMRings};
    \item[(2)] $F_*S$ was shown to be an Ohm-Rush $S$-module when $S$ is an excellent regular local ring in \cite{KatzmanLyubeznikZhangOnDiscretenessAndRationality};
    \item[(3)] $F_*S$ is an intersection flat $S$-module when  $S$ is $F$-finite \cite{BlickleMustataSmithDiscretenessAndRationalityOfFThresholds,KatzmanParameterTestIdealOfCMRings,SharpBigTestElements};
\end{enumerate}
Our main goal in this paper is to systematically study intersection flatness, Ohm-Rush and related properties. 
In forthcoming work \cite{DESTtate} the techniques of this article will then be used to then exhibit new cases of the existence of test elements.

One property that is related to Ohm-Rush and intersection flatness relies on \emph{trace ideals}, which have been studied in numerous algebraic contexts. Let $M$ be an $R$-module and suppose $\varphi \in M^* \coloneqq \Hom_R(M,R)$. Then for any $x \in M$ and for any ideal $I$ of $R$ such that $x \in IM$, the linearity of $\varphi$ implies that $\varphi(x) \in I$. Thus, the ideal 
$
    \Tr_M(x) \coloneqq \im(M^* \xrightarrow{\ev @ x} R),
$
called the \emph{trace of $x$}, is contained in $c_M(x)$. Consequently, if $x \in \Tr_M(x)M$, we get $\Tr_M(x) = c_M(x)$ because $c_M(x)$ is the intersection of all ideals $I$ of $R$ such that $x \in IM$. We will call an $R$-module $M$ \emph{Ohm-Rush trace} (\autoref{def:Ohm-Rush-trace}) if for all $x \in M$, one has 
\begin{equation}
    \label{eq:ORT}
x \in \Tr_M(x)M.
\end{equation} 
We will say a ring map $R \to S$ is \emph{Ohm-Rush trace} if $S$ is an Ohm-Rush trace $R$-module. We note that an Ohm-Rush trace $R$-module $M$ is automatically flat (see \autoref{rem:ORT-modules}(g)).

The Ohm-Rush trace condition was also studied in \cite{OhmRu-content}, where modules satisfying \autoref{eq:ORT} were called \emph{trace modules}. However, we feel it is less confusing to use our more verbose terminology because the term `trace' is used abundantly in mathematics. Projective modules are prototypes of Ohm-Rush trace modules (see \autoref{lem:projective-ORT}), and more examples of Ohm-Rush trace modules can be constructed from them because this notion satisfies a number of stability properties; see \autoref{lem:composing-ORT} and \autoref{prop:ORT-indeterminate}. 
One can think of the Ohm-Rush trace property as codifying not just that $M^*$ is non-trivial, but that there are `sufficiently many' functionals on $M$. For instance, for an Ohm-Rush trace module $M$, the canonical map $M \to M^{**}$ is not just injective but also cyclically pure (\autoref{lem:ORT-cyclic-purity}). Recall that a map of $R$-modules $M \to N$ is \emph{(cyclically) pure} if for all (cyclic) modules $P$, the induced map $M \otimes_R P \to N \otimes_R P$ is injective. Pure maps are also called \emph{universally injective} in the literature \cite[\href{https://stacks.math.columbia.edu/tag/058H}{Tag 058H}]{stacks-project}.

The three properties of modules we have introduced so far are related as follows:
\[\begin{tikzcd}[cramped]
	{\textrm{Ohm-Rush trace}} & {\textrm{Ohm-Rush}} & {\textrm{Intersection flat}}
	\arrow[Rightarrow, from=1-1, to=1-2]
	\arrow[Rightarrow, from=1-3, to=1-2]
\end{tikzcd}\]
We would like to add more implications in the diagram. The way forward is via the theory of Mittag-Leffler and strictly Mittag-Leffler modules developed by Raynaud and Gruson in \cite{rg71}. They utilized these latter notions in their proof of the faithfully flat descent of projectivity. 
A module $M$ over a commutative ring $R$ is \emph{Mittag-Leffler} if for all $R$-linear maps $g \colon P \to M$ from a finitely presented $R$-module $P$, there exists an $R$-linear map $f \colon P \to Q$ such that $Q$ is finitely presented
\[\begin{tikzcd}
	& P \\
	Q && M
	\arrow["{\exists f}"', dashed, from=1-2, to=2-1]
	\arrow["g", from=1-2, to=2-3]
\end{tikzcd}\]
and such that for all $R$-modules $N$, $\ker(f \otimes_R \id_N) = \ker(g \otimes_R \id_N)$ (see \autoref{def:Mittag-Leffler}). Such an $f$ is called a \emph{stabilizer} of $g$. One should think of the condition on kernels as saying that $g$, which maps to a possibly highly non-finitely presented module $M$, behaves as if it were a map between finitely presented modules. 

The Mittag-Leffler property generalizes the notion of a finitely presented module, and the condition on equality of kernels is intimately related to the notion of purity. Indeed, one can show (\autoref{lem:domination} (3)) that $\ker(f \otimes_R \id_N) = \ker(g \otimes_R \id_N)$ holds for all $R$-modules $N$ precisely when $f'$ and $g'$ are pure maps of modules in the pushout

\[
    \xymatrix{P \ar[r]_ g \ar[d]_ f &  M \ar[d]^{f'} \\  Q \ar[r]^{g'} &  T .}
\]

Note that for any commutative diagram of linear maps 
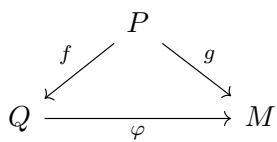
\[\begin{tikzcd}
	& P \\
	Q && M
	\arrow["f"', from=1-2, to=2-1]
	\arrow["g", from=1-2, to=2-3]
	\arrow["\varphi"', from=2-1, to=2-3]
\end{tikzcd}\]
one always has 
$
\ker(f \otimes_R \id_N) \subseteq \ker((\varphi \circ f) \otimes_R \id_N) = \ker(g \otimes_R \id_N).
$
Thus, if you have two maps $f \colon P \to Q$ and $g \colon P \to M$ that both factor each other, then $\ker(f \otimes_R \id_N) = \ker(g \otimes_R \id_N)$ for all $R$-modules $N$. This observation leads to a strengthening of the notion of a  Mittag-Leffler module. 
We say and $R$-module $M$ is \emph{strictly Mittag-Leffler} if for all linear maps $g \colon P \to M$ where $P$ is finitely presented, there exists a linear map $f \colon P \to Q$ such that $Q$ is finitely presented, and $f, g$ factor through each other (\autoref{def:strict-Mittag-Leffler}). The above discussion shows that a strictly Mittag-Leffler module is always Mittag-Leffler. Moreover, projective modules are strictly Mittag-Leffler (\autoref{rem:Mittag-Leffler} (h)) and over Noetherian complete local rings, Matlis duality (\autoref{lem:Auslander-Warfield-lemma}) can be used to show that the Mittag-Leffler and strict Mittag-Leffler properties coincide (\autoref{prop:ML-SML}).

In general, the strict Mittag-Leffler property is significantly stronger than the Mittag-Leffler property. For instance, if $M$ is strictly Mittag-Leffler, then any pure map $P \to M$ from a finitely presented module $P$ is automatically split (\autoref{rem:Mittag-Leffler} (g)). 
Using this observation one can show that over a local ring $(R,\fm)$, a flat strictly Mittag-Leffler module $M$ can be expressed as a filtered union of finite free submodules that are direct summands of $M$, whereas for a flat Mittag-Leffler module $M$ these finite free submodules need only be pure in $M$ in general; see \cite{rg71} or \autoref{prop:ML-SML-modules-colimit-pure-split-free}. In fact, we will show that further removing the restriction that the union must be filtered characterizes flat Ohm-Rush modules (\autoref{prop:analog-RG219}). This explains key differences between flat strictly Mittag-Leffler, flat Mittag-Leffler, and flat Ohm-Rush modules in the local setting.

From the structure of a strictly Mittag-Leffler module $M$ over a local ring $R$ described above, it follows that $M$ typically $M$ typically has many non-trivial linear functionals. This is reminiscent of the Ohm-Rush trace property, and indeed, Raynaud and Gruson showed:

\begin{theorem*}\cite[Part II, Prop.\ 2.3.4]{rg71}
    \label{thm:ORT-SML}
Let $R$ be a commutative ring and $M$ be an $R$-module. Then $M$ is Ohm-Rush trace if and only if $M$ is a flat strictly Mittag-Leffler $R$-module.
\end{theorem*}

\noindent
A feature result of this article is an analog of 
the above result for the Mittag-Leffler and intersection flatness properties.

\begin{theorem*}[\autoref{thm:MittagORIF}]
    \label{thm:ML-IF}
    Let $M$ be a flat module over a commutative ring $R$. Then the following are equivalent:
    \begin{enumerate}
        \item[$(1)$] $M$ is Mittag-Leffler.
        \item[$(2)$] $M$ is intersection flat.
        \item[$(3)$] $M \otimes_R S$ is an intersection flat $S$-module for all $R$-algebras $S$.  
    \end{enumerate}
\end{theorem*}

This theorem reveals that the intersection flatness property is substantially more restrictive than the Ohm-Rush property because while intersection flatness is preserved under arbitrary base change, the Ohm-Rush property is not. In fact, the Ohm-Rush property fails to be preserved even under even smooth base change (see \autoref{eg:OR-not-preserved-base-change-polynomial}). The above two theorems allow us to use the substantial arsenal developed in \cite{rg71} for (strictly) Mittag-Leffler modules to study the Ohm-Rush trace property and intersection flatness.

In summary, we have the following implications for \emph{flat} modules over an arbitrary commutative ring (the implications that hold without flatness assumptions are also indicated):
\begin{figure}[H]
    \begin{tikzcd}
	& {\textrm{Ohm-Rush}} \\
	{\textrm{Ohm-Rush trace}} && {\textrm{Intersection flat}} \\
	\\
	{\textrm{strictly Mittag-Leffler}} && {\textrm{Mittag-Leffler}}
	\arrow["\textrm{flatness not needed}", curve={height=-12pt}, Rightarrow, from=2-1, to=1-2]
	\arrow[curve={height=12pt}, Rightarrow, 2tail reversed, from=2-1, to=4-1]
	\arrow[curve={height=-12pt}, Rightarrow, 2tail reversed, from=2-3, to=4-3]
	\arrow["\textrm{flatness not needed}"', shift right=1, curve={height=18pt}, Rightarrow, from=4-1, to=4-3]
	\arrow["\textrm{flatness not needed}"', curve={height=12pt}, Rightarrow, from=2-3, to=1-2]
\end{tikzcd}
\caption{Diagram of implications for flat modules} \label{fig:implications}
\end{figure}

\par
The implications in the above diagram that are not equivalences are all strict. In forthcoming work \cite{DESTtate}, we will show that the Frobenius maps of appropriately chosen prime characteristic DVRs can be used to illustrate this. Surprisingly, our analysis together with results from \cite{rg71} and \cite{HochsterJeffriesintflatness} show that all five notions of modules from \autoref{fig:implications} are equivalent in the complete local case. We obtain several new characterizations as well.

      \begin{theorem*}[\autoref{thm:ML-SML-ORT-intersection-flat}]
        \label{thm:flat-all-notion-equivalent-complete}
        Let $(R,\fm,\kappa)$ be a Noetherian local ring that is $\fm$-adically complete. Let $M$ be a flat $R$-module and let $\widehat{M}$ denote its $\fm$-adic completion. Then the following are equivalent:
        \begin{enumerate}
            \item[$(1)$] $M$ is intersection flat.
            
            \item[$(2)$] $M$ is Mittag-Leffler.
            
            \item[$(3)$] $M$ is strictly Mittag-Leffler.
            
            \item[$(4)$] $M$ is an Ohm-Rush trace $R$-module.
        
            \item[$(5)$] $M$ is an Ohm-Rush $R$-module.
            
            \item[$(6)$] The 
            canonical map $M \to \widehat{M}$ is a pure map of $R$-modules.
            
            \item[$(7)$] The canonical map $M \to \Hom_R(\Hom_R(M,R),R)$ is a pure map of $R$-modules.
            
            \item[$(8)$] The canonical map $M \to \Hom_R(\Hom_R(M,R),R)$ is cyclically pure as a
            map of $R$-modules.
            
            \item[$(9)$] For all finitely generated $R$-modules $L$,
            $L \otimes_R M$ is $\fm$-adically separated.
        
            \item[$(10)$] For all cyclic $R$-modules $L$, $L \otimes_R M$ is $\fm$-adically separated.
            
            \item[$(11)$] For any finitely generated submodule $P$ of $M$, there exists a finitely generated submodule $L$ of $M$ containing $P$ such that $L$ is free and is a direct  summand of $M$.
        
            \item[$(12)$] For any cyclic submodule $P$ of $M$, there exists a finitely generated submodule $L$ of $M$  containing $P$ such that $L$ is free and is a direct summand of $M$.
        \end{enumerate}
        Thus, if $M$ is flat and $\fm$-adically complete, then $M$ satisfies the equivalent conditions $(1)-(12)$.
      \end{theorem*}

For flat modules over an arbitrary commutative ring, the Ohm-Rush property is the most general. Improving upon Sharp's result \cite{SharpBigTestElements} referenced earlier, in forthcoming work \cite{DESTPhantom} we will show the existence of test elements and exhibit a robust test ideal theory for reduced quotients of an excellent regular ring $S$ for which $F_*S$ has the Ohm-Rush property.
Thus, it is natural to desire a better understanding of the structure of flat Ohm-Rush modules, along the lines of the structure theory of (strictly) Mittag-Leffler modules developed in \cite{rg71}. We develop this theory in \autoref{sec:Stabilizers and cyclic stabilizers}.

Our first observation is that while the Mittag-Leffler property captures purity, the Ohm-Rush condition captures cyclic purity. Let $M$ be a module over a commutative ring $R$ and let $g \colon P \to M$ be a linear map from a finitely presented $R$-module $P$. We will say that a linear map $f \colon P \to Q$ is a \emph{cyclic stabilizer of $g$} if $Q$ is finitely presented and for all cyclic $R$-modules $N$ we have $\ker(f \otimes_R \id_N) = \ker(g \otimes_R \id_N)$. Thus, the notion of a cyclic stabilizer is a weakening of the notion of a stabilizer because for the latter one requires the equality of kernels to hold for all $R$-modules and not just cyclic ones.

Just as stabilizers are intimately related to purity, cyclic stabilizers are related to cyclic purity (see \autoref{lem:cyclic-stab-pushout}). 
We first show that cyclic stabilizers capture the notion of a flat Ohm-Rush module. Note that flatness is a mild assumption for Ohm-Rush modules. For instance, if $R$ is a domain, then an Ohm-Rush $R$-module is flat if and only if it is torsion-free (\autoref{prop:OR-flat-torsionfree}).

\begin{theorem*}[\autoref{thm:OR-equivalences}]
    \label{thm:cyclic-stab-OR}
Let $R$ be a commutative ring and $M$ be a flat $R$-module. Then the following are equivalent:
\begin{enumerate}
    \item[$(1)$] $M$ is an Ohm-Rush $R$-module.
    \item[$(2)$] Every linear map $v\colon R \to M$ admits a cyclic stabilizer $u \colon R \to F$ where
    $F$ is free of finite rank and such that $u$ factors $v$.
    \item[$(3)$] Every linear map $v \colon R \to M$ admits a cyclic stabilizer that factors $v$.
    \item[$(4)$] Every linear map $v\colon R \to M$ admits a cyclic stabilizer $u \colon R \to P$ such 
    that $u(1) \in c_P(u(1))P$.
\end{enumerate}
\end{theorem*}

\noindent
 
\noindent Surprisingly, we also show the notion of a flat Ohm-Rush module is characterized by the existence of stabilizers (and not just cyclic stabilizers) for maps $R \to M$. This result is summarized below along with an analogous result for Mittag-Leffler modules established in \cite{rg71}.

    \begin{theorem*}
        \label{thm:flat-OR-vs-ML}
        Let $R$ be a commutative ring and $M$ be a flat $R$-module. Then we have the following:
        \begin{enumerate}
            \item[$(1)$] \cite{rg71} $M$ is Mittag-Leffler if and only if for all positive integers $n$, every $R$-linear map $R^{\oplus n} \to M$ admits a stabilizer.
            \item[$(2)$] (\autoref{cor:OR-equivalences-addition}) $M$ is Ohm-Rush if and only if every $R$-linear map $R \to M$ admits a stabilizer.  
        \end{enumerate}
    \end{theorem*}

Another application of (cyclic) stabilizers is to deduce descent statements for Ohm-Rush and intersection flatness. Descent for stabilizers was established in \cite{rg71}, while we prove descent for cyclic stabilizers in \autoref{cor:descent-cyclic-stabilizers}. As a consequence, one then has the following result.
\begin{theorem*}[\autoref{thm:descentOhm-Rush}, \autoref{cor:pure-descent-IF}]
    \label{thm:descent-OR-IF-ML}
    Let $R \to S$ be a homomorphism of commutative rings and $M$ an $R$-module. Then we have the following:
    \begin{enumerate}
        \item[$(1)$] If $R \to S$ is pure and $S \otimes_R M$ is an intersection flat $S$-module, then $M$ is an intersection flat $R$-module.
        \item[$(2)$] If $R \to S$ is cyclically pure, $M$ is flat and $S \otimes_R M$ is an Ohm-Rush $S$-module, then $M$ is an Ohm-Rush $R$-module.  
    \end{enumerate}
\end{theorem*}
\noindent Note that we do not know if the Ohm-Rush property descends along cyclically pure ring maps without the assumption that $M$ is flat over the base ring $R$. One application (\autoref{cor:IF-universally-OR}) of descent of intersection flatness and \autoref{thm:ML-SML-ORT-intersection-flat} is that over a Noetherian local ring, the intersection flat modules are precisely those flat Ohm-Rush modules such that the Ohm-Rush property is preserved under arbitrary base change. 

Yet another application of (cyclic) stabilizers is to deduce openness of certain pure loci. Suppose $M$ is a module over a commutative ring $R$ and $g \colon P \to M$ is a linear map from a finitely presented $R$-module $P$ that admits a stabilizer $f \colon P \to Q$. Then for all prime ideals $\p$ of $R$, $f_\p$ is also a stabilizer of $g_\p$ as maps of $R_\p$-modules, and as $\coker(f)$ is a finitely presented $R$-module we have
\[
\textrm{$\{\p \in \Spec(R) \colon g_\p$ is $R_\p$-pure$\} 
= \{\p \in \Spec(R) \colon f_\p$ splits in $\Mod_{R_\p}\}$}.
\]
by \cite[Cor.\ 2.4]{Lazard}.
The latter locus is open because splittings of maps of finitely presented modules spread. Thus, if $g \colon P \to M$ admits a stabilizer, then the pure locus of $g$ is open in $\Spec(R)$. As a consequence, we have:

\begin{prop*}[\autoref{cor:ML-open-pure-loci}, \autoref{lem:purity-open-locus}]
    \label{prop:Intro-openness-pure-loci}
    Let $R$ be a commutative ring, $M$ a flat $R$-module, $P$ a finitely presented $R$-module and $g \colon P \to M$ a linear map. Let $\Pure(g)$ (resp. $\CPure(g)$) be the locus of primes $\p \in \Spec(R)$ such that $g_\p$ is $R_\p$-pure (resp. $R_\p$-cyclically pure). Then 
    \[\Pure(g) = \CPure(g).\] Furthermore,
    \begin{enumerate}
        \item[(1)] If $M$ is intersection flat, then $\Pure(g)$ is open in $\Spec(R)$.
        \item[(2)] If $M$ is Ohm-Rush and $P = R$, then $\Pure(g) = \Spec(R) \setminus \mathbf{V}(c_M(g(1)))$. 
    \end{enumerate}
\end{prop*}

Openness of pure loci is not only implied by the intersection flat and Ohm-Rush properties, but also implies these properties when the modules are known to be locally intersection flat and Ohm-Rush.
The main local-to-global result in this direction, relying also on the descent results above, is the following:

\begin{theorem*}[\autoref{thm:local-to-global-ML-OR}, \autoref{prop:local-to-global-Prufer}]
    \label{thm:Intro-local-global-IF-OR}
    Let $R$ be a commutative ring and $M$ a flat $R$-module.
    \begin{enumerate}
        \item[$(1)$] Suppose $M_\p$ is an intersection flat $R_\p$-module for all $\p \in \Spec(R)$. Then $M$ is intersection flat if and only if for all integers $n \geq 0$ and linear maps $g \colon R^{\oplus n} \to M$, the cyclically pure locus of $g$ is open in $\Spec(R)$. 
        \item[$(2)$] Suppose $M_\p$ is an Ohm-Rush $R_\p$-module for all prime ideals $\p$ of $R$. 
        \begin{enumerate}
            \item[$(a)$] $M$ is Ohm-Rush if for all integers $n \geq 0$ and linear maps $g \colon R^{\oplus n} \to M$, the cyclically pure locus of $g$ is open in $\Spec(R)$. 
            \item[$(b)$]   If $R$ is a Pr{\"u}fer domain (e.g., a Dedekind domain), then $M$ is Ohm-Rush if and only if for all linear maps $g \colon R \to M$, the cyclically pure locus of $g$ is open in $\Spec(R)$.
        \end{enumerate}  
    \end{enumerate}
\end{theorem*}

\noindent Note that the case $n = 1$ above is already quite interesting. For instance, if $M$ is locally Ohm-Rush and one has the openness of pure loci for maps $R \to M$, then one can recover the content function of $M$ up to radical (\autoref{prop:close-to-Ohm-Rush}); that is, for all $x \in M$, there exists a smallest ideal $I = \sqrt{I}$ such that $x \in IM$.

\section{Preliminaries}

\subsection{Conventions and abbreviations}

All rings in this paper are \emph{commutative}. We will frequently work over non-Noetherian rings because \cite{rg71,HochsterJeffriesintflatness} develop the theories of Mittag-Leffler modules and intersection flat modules over arbitrary commutative rings, and one of our goals in this paper is to link Raynaud and Gruson's classic work with notions that commutative algebraists have studied in relation to Hochster and Huneke's theory of tight closure \cite{HochsterHunekeTCandBrianconSkoda,HochsterHunekeFRegularityTestElementsBaseChange}. In addition, even in the study of singularities of Noetherian rings, which is the main application we have in mind, one has to frequently study non-Noetherian rings such as perfections, absolute integral closures and big Cohen-Macaulay algebras.  We will specify whenever Noetherian hypotheses are needed in the statements of results. When we use the term `regular ring' it is implicit that such rings are Noetherian. Additionally, for us a `local ring' is not necessarily Noetherian. However, when we say a local ring is `complete', we mean it is Noetherian and complete with respect to the ideal adic topology induced by the maximal ideal.  

The following abbreviations are used freely in the text.
\begin{enumerate}
    \item ML for Mittag-Leffler (\autoref{def:Mittag-Leffler}),
    \item SML for strictly Mittag-Leffler (\autoref{def:strict-Mittag-Leffler}),
    \item ORT for Ohm-Rush trace (\autoref{def:Ohm-Rush-trace}).
\end{enumerate}

\subsection{Pure maps} Recall  that given a ring $R$, a map of $R$-modules $M \to N$ is \emph{pure} if for all $R$-modules $P$, the induced map $M \otimes_R P \to N \otimes_R P$ is injective. Similarly, a map of $R$-modules $M \to N$ is \emph{cyclically pure} provided $M \otimes_R P \to N \otimes_R P$ is injective for all cyclic $R$-modules $P$. A ring homomorphism $R \to S$ is \emph{pure} (or \emph{cyclically pure}) if it is  pure (respectivley cyclically pure) as map of $R$-modules. If $M$ is a submodule of an $R$-module $N$ such that the inclusion $M \hookrightarrow N$ is pure, then we will often say that $M$ is \emph{pure in $N$}. 

Pure maps of modules are also called \emph{universally injective} maps of modules in the literature. While the latter terminology is more descriptive, we will primarily use the former terminology for its brevity.

The following fact about pure maps is presumably well-known, but for lack of a good reference we include a proof here. Recall that an $R$-algebra $S$ is a \emph{first order nilpotent thickening} of $R$ if $S \cong R/I$, where $I$ is an ideal of $R$ such that $\mathcal{I}^2 = 0$ \cite[\href{https://stacks.math.columbia.edu/tag/04EX}{Tag 04EX}]{stacks-project}.

\begin{lemma}
    \label{lem:pure-is-pure-arbitrary-base-change}
    Let $R$ be a ring and $\varphi \colon M \to N$ a map of $R$-modules. Then $\varphi$ is pure if and only if for any $R$-algebra $S$ that is also a first order nilpotent thickening of $R$, $\varphi \otimes_R \id_S$ is an injective $S$-linear map.
\end{lemma}

\begin{proof}
    The non-trivial implication is to show that if $\varphi \otimes_R \id_S$ is an injective $S$-linear map for all $R$-algebras $S$ that are a first order nilpotent thickening of $R$, then $\varphi$ is a pure $R$-linear map. Let $P$ be an $R$-module and let $R \ltimes P$ be the $R$-algebra obtained by Nagata's principle of idealization \cite{NagataLocalRings}. That is, the underlying additive group of $R \ltimes P$ is $R \oplus P$ with multiplication given by $(r_1,p_1)\cdot (r_2,p_2) = (r_1r_2, r_1p_2 + r_2p_1)$. The $R$-algebra structure is induced by the ring homomorphism
    \begin{align*}
        R &\to R \ltimes P.\\
        r & \mapsto (r,0)
    \end{align*}
    and coincides with the $R$-module structure on $R \oplus P$.
    Note that $R \ltimes P$ is a first order nilpotent thickening of $R$ because the ideal $I = \{(0,p) \colon p \in P\}$ satisfies $I^2 = 0$ and we have $(R \ltimes P)/I \cong R$.
    The canonical injection $P \hookrightarrow R \ltimes P$ given by mapping $p \to (0,p)$ is a split $R$-linear map (the usual projection $R \ltimes P \twoheadrightarrow P$ is a left-inverse), and hence, it is $R$-pure. Therefore, in the commutative diagram
    \[
    \xymatrix@C+1pc{M \otimes_R P \ar[r]_{\varphi \otimes_R \id_P} \ar[d] &  N \otimes_R P \ar[d] \\  M \otimes_R (R \ltimes P) \ar[r]^{\varphi \otimes_R \id_{R \ltimes P}} & N \otimes_R (R \ltimes P),}
    \]
    the vertical maps are injective. Since the bottom horizontal map is injective by assumption so is the top horizontal map.
\end{proof}

There is a related deformation-theoretic criterion for purity.

\begin{lemma}
    \label{lem:purity-deformation-theoretic}
    Let $R$ be a ring and $I$ be a nilpotent ideal of $R$ (i.e. $I^n = 0$ for some integer $n > 0$). Let $\varphi \colon M \to N$ be an $R$-linear map where $N$ is $R$-flat. If $\varphi \otimes_R \id_{R/I} \colon M/IM \to N/IM$ is $R/I$-pure, then $\varphi$ is $R$-pure.
\end{lemma}

\begin{proof}
    Since purity is preserved under restriction of scalars, the $R/I$-purity of $\varphi \otimes_R \id_{R/I}$ implies the $R$-purity of $\varphi \otimes_R \id_{R/I}$. The result now follows by \cite[Part I, Lem.\ 4.2.1]{rg71}.
\end{proof}

\begin{lemma}[cf. {\cite[Thm.\ 3]{Warfieldpure}}]
    \label{lem:Auslander-Warfield-lemma}
Let $(R, \fm)$ be a Noetherian local ring. Let $\varphi \colon M \to N$ be a map
of $R$-modules, where $M$ is finitely generated and $N$ is an arbitrary $R$-module. Let $\widehat{M}$ denote the $\fm$-adic completion of $M$. Then
the following are equivalent:
\begin{enumerate}
    \item[$(a)$] $\varphi$ is a pure map of $R$-modules.
    \item[$(b)$] The canonical map $\Hom_R(N, \widehat{M}) \xrightarrow{- \circ \varphi} \Hom_R(M, \widehat{M})$ is surjective.
\end{enumerate}
Thus, if $(R, \fm)$ is additionally $\fm$-adically complete, then $\varphi$ is pure
if and only if $\varphi$ admits an $R$-linear left-inverse.
\end{lemma}

\begin{proof}
    First suppose the canonical map is surjective.  Let $j: M \to \widehat M$ be the completion map.  Then there is some $h\in \Hom_R(N, \widehat{M})$ such that $h \circ \varphi = j$.  But since $j$ is faithfully flat, it is pure, and hence $\varphi$ is pure.
    
    Conversely, suppose $\varphi$ is pure.  Then $\varphi \otimes 1_{M^\vee} :M \otimes_R M^\vee \to N \otimes_R M^\vee$ is injective, where $(-)^\vee = \Hom_R(-, E_R(R/\fm))$ is the Matlis duality functor.  Applying said functor to the resulting map, we get a surjective map $(N \otimes_R M^\vee)^\vee \onto (M \otimes_R M^\vee)^\vee$.  But then by Hom-tensor adjointness, along with the fact that the double Matlis dual of a finite module is its completion, we get $\Hom_R(N,\widehat M) \cong \Hom_R(N, M^{\vee \vee}) \cong (N \otimes_R M^\vee)^\vee \onto (M \otimes_R M^\vee)^\vee \cong \Hom_R(M, M^{\vee \vee}) = \Hom_R(M, \widehat{M})$.  Moreover, since everything is natural, the map in question is the canonical map given in (b).
    
    For the final statement, in this case we have $M = \widehat{M}$, so if $\varphi$ is pure, then by what we have proved, the canonical map $\Hom_R(N,M) \to \Hom_R(M,M)$ is surjective.  In particular, there exists $h \in \Hom_R(N,M)$ with $1_M = h \circ \varphi$.  The converse is trivial.
\end{proof}

The next lemma gives a criterion for when the completion of a pure map of modules over a Noetherian local ring is pure over the completion. The result, however, is stated more generally. Note that no finiteness restrictions are imposed on the modules.

\begin{lemma}
    \label{lem:completion-purity}
    Let $R$ be a ring and $I$ a finitely generated ideal of $R$ such that $R/I$ is Noetherian. Let $M$ be a flat $R$-module. Then we have the following:
    \begin{enumerate}
        \item[$(1)$] $\widehat{M}^I$ is a flat module over the Noetherian ring $\widehat{R}^I$.
        \item[$(2)$] If $\varphi \colon N \to M$ is a pure map of $R$-modules, then the induced map on $I$-adic completions 
        $
        \widehat{\varphi}^I \colon \widehat{N}^I \to \widehat{M}^I    
        $
        is a pure map of $\widehat{R}^I$-modules.  
    \end{enumerate}
\end{lemma}

\begin{proof}
    (1) follows by \cite[\href{https://stacks.math.columbia.edu/tag/0AGW}{Tag 0AGW}]{stacks-project}. 

    (2) Consider the short exact sequence 
    \[
    0 \to N  \xrightarrow{\varphi} M \to \coker(\varphi) \to 0.   
    \]
    Since $\varphi$ is $R$-pure and $M$ is $R$-flat, both $N$ and $\coker(\varphi)$ are flat $R$-modules. This follows easily by the long exact sequence on $\Tor$ and purity, but see for instance, \cite[\href{https://stacks.math.columbia.edu/tag/058P}{Tag 058P}]{stacks-project} for a more elementary diagram-chase proof. We then get a short exact sequence of flat $\widehat{R}^I$-modules
    \[
    0 \to \widehat{N}^I \xrightarrow{\widehat{\varphi}^I} \widehat{M}^I \to \widehat{\coker{\varphi}}^I \to 0    
    \]
    by (1) and \cite[\href{https://stacks.math.columbia.edu/tag/0315}{Tag 0315} (3)]{stacks-project}. Then $\widehat{\varphi}^I$ is a pure map of $\widehat{R}^I$-modules by looking at the long-exact sequence on Tor because $\widehat{\coker(\varphi)}^I$ is $\widehat{R}^I$-flat, or alternatively, by using \cite[\href{https://stacks.math.columbia.edu/tag/058M}{Tag 058M}]{stacks-project}.
\end{proof}

We also recall the following helpful result about pure maps.

\begin{lemma} \cite[Cor.\ 2.4]{Lazard}
    \label{lem:pure-iff-split}
    Let $\varphi \colon M \to N$ be a pure map of $R$-modules such that $\coker(\varphi)$ is a finitely presented $R$-module. Then $\varphi$ splits, that is, $\varphi$ has a left-inverse in $\Mod_R$.
\end{lemma}

The next result is well-known, but for lack of an appropriate reference, we include a proof here.

\begin{lemma}
    \label{lem:pure-submodule-flat}
    Let $R$ be a ring and $M$ be a flat $R$-module. Let $N$ be a pure submodule of $M$ (i.e. the inclusion $\iota \colon N \hookrightarrow M$ is pure). Then $N$ is also a flat $R$-module.
\end{lemma}

\begin{proof}
    Let $f \colon P \to Q$ be an injective $R$-linear map. We get a commutative diagram
    \[
        \xymatrix{P \otimes_R N \ar[r]^{f \otimes_R \id_N} \ar[d]_{\id_P \otimes_R \iota} &  Q \otimes_R N \ar[d]^{\id_Q \otimes_R \iota} \\  P \otimes_R M \ar[r]_{f \otimes_R \id_M} &  Q \otimes_R M }
        \]
    The map $\id_P \otimes_R \iota$ is injective by purity of $\iota$ and $f \otimes_R \id_M$ is injective by flatness of $M$. Thus, by the commutativity of the above diagram, $f \otimes_R \id_N$ is also injective, that is, $N$ is $R$-flat.
\end{proof}

\subsection{Expansions of radical ideals}
We will need the following result on the preservation of radical ideals under expansion along the completion map.

\begin{lemma}
    \label{cor:expansion-radical-ideals}
    Let $A$ be a Noetherian ring such that for all prime ideals $\p$ of $A$, the fibers of $A_\p \to \widehat{A_\p}$ are reduced. Then for any radical ideal $I$ of $A$, $I\widehat{A_\p}$ is a radical ideal of $\widehat{A_\p}$. Thus, for all ideals $J$ of $A$, $\sqrt{J}\widehat{A_\p} = \sqrt{J\widehat{A_\p}}$.
\end{lemma}

\begin{proof}
    Since the property of being a radical ideal is preserved under arbitrary localicalization, we may assume $A$ is local and $\p$ is the maximal ideal $\fm$ of $A$. We let $\widehat{A}$ denote the $\fm$-adic completion of $A$. By hypothesis, the fibers of $A \to \widehat{A}$ are reduced. Let $I$ be a radical ideal of $A$. Then 
    $
    A/I \to \widehat{A}/I\widehat{A}.    
    $
    is also a flat map of Noetherian rings with reduced fibers. Since $A/I$ is reduced by assumption, $\widehat{A}/I\widehat{A}$ is reduced by \cite[\href{https://stacks.math.columbia.edu/tag/0C21}{Tag 0C21}]{stacks-project}. Thus, $I\widehat{A}$ is a radical ideal of $A$.

    The second statement follows easily from the first and its proof is omitted.
\end{proof}

\section{Stabilizers and cyclic stabilizers}
\label{sec:Stabilizers and cyclic stabilizers}
In this section, we will discuss in some detail the notions of domination and stabilizers developed in \cite[Part II]{rg71}. The motivation for doing so is to create a natural framework for introducing the notion of a Mittag-Leffler module (\autoref{def:Mittag-Leffler}) that makes the relation of the notion to the well-studied concept of a pure map of modules most transparent. 
For completeness and the convenience of the reader, we shall give detailed proofs of the most important results needed from \cite[Part II]{rg71}. We will also define a natural weakening of the notion of stabilizers, called cyclic stabilizers, and develop some of their basic properties. While stabilizer captures the Mittag-Leffler condition, we will see that cyclic stabilizers will be related to a notion that we call Ohm-Rush (\autoref{def:Ohm-Rush}). Ohm-Rush modules were first introduced by Ohm and Rush \cite{OhmRu-content} in relation with the content function for modules. The Ohm-Rush condition has been studied for the Frobenius homomorphism in prime characteristic commutative algebra in connection with test ideals \cite{BlickleMustataSmithDiscretenessAndRationalityOfFThresholds,SharpBigTestElements,SharpTestElementsforFpure}. 

\subsection{Domination and stabilizers}
In this subsection $R$ denotes an arbitrary commutative ring that is not necessarily Noetherian. Also, a local ring for us is just a ring with a unique maximal ideal, that is, local rings are not necessarily Noetherian unless otherwise specified.

\begin{definition}
\label{def:domination}
Let $R$ be a ring and $f \colon M \to P$ and $g \colon M \to Q$ be two
maps of $R$-modules with common domain $M$. We say $g$ 
\emph{dominates} $f$ if for all $R$-modules $N$,
$
\ker(f \otimes_R \id_N) \subseteq \ker(g \otimes_R \id_N).
$
\end{definition}

The notion of domination satisfies the following properties:

\begin{lemma}
\label{lem:domination}
Let $R$ be a ring and $f \colon M \to P$ and $g \colon M \to Q$ be
$R$-linear maps. Then we have the following:
\begin{enumerate}
    \item[$(1)$] $g$ dominates $f$ $\Longleftrightarrow$ for all finitely presented $R$-modules $N$, 
    $\ker(f \otimes_R \id_N) \subseteq 
    \ker(g \otimes_R \id_N)$.
    
    \item[$(2)$] If $g$ factors through $f$, that is, if there exists
    a map $\varphi \colon P \to Q$ such that $g = \varphi \circ f$,
    then $g$ dominates $f$. The converse holds if $\coker(f)$ is 
    finitely presented as an $R$-module.
    
    \item[$(3)$] Consider the pushout of $f$ and $g$
    \[
    \xymatrix{ M \ar[r]_ f \ar[d]_ g &  P \ar[d]^{g'} \\  Q \ar[r]^{f'} &  T }
    \]
    Then $g$ dominates $f$ $\Longleftrightarrow$ $f'$ is a pure map of $R$-modules.
    
    \item[$(4)$] Suppose there is a pure $R$-linear map $\varphi \colon P \to Q$ such that $g = \varphi \circ f$. 
    Then $g$ and $f$ dominate each other.
    
    \item[$(5)$] Consider a commutative diagram of $R$ modules
    \[
    \xymatrix{&M \ar[dl]_f \ar[d]_g \ar[dr]^h  \\
    P \ar[r] &Q \ar[r] &R}
    \]
    where $f$ and $h$ dominate each other. Then $f, g, h$  dominate each other in pairs.
    
    \item[$(6)$] Consider a commutative diagram of $R$-modules
    \[
    \xymatrix{&M \ar[dl]_f \ar[d]_g  \\
    P \ar[d]_h  &Q \\
    R\ar[ur]}
    \]
    If $f$ and $g$ dominate each other, then $h \circ f$ and $g$
     dominate each other.

    \item[$(7)$] Suppose $M$ is a finitely presented $R$-module and $Q$ is a flat $R$-module. Then any $R$-linear map $g \colon M \to Q$ dominates an $R$-linear map $f \colon M \to P$, where $P$ is a free $R$-module of finite rank.
        
    \item[$(8)$] Suppose we have a surjective $R$-linear map 
        $\varphi \colon M' \twoheadrightarrow M$. Then $g \circ \varphi$ dominates
        $f \circ \varphi$ if and only if $g$ dominates $f$.
\end{enumerate}
\end{lemma}

\begin{proof}
$(1)$ is \cite[\href{https://stacks.math.columbia.edu/tag/059C}{Tag 059C}]{stacks-project}.

For $(2)$, if $g$ factors through $f$ via $\varphi$, then it follows that for any $R$-module $N$,
\[
\ker(g \otimes_R \id_N) = \ker((\varphi \otimes_R \id_N) \circ (f \otimes_R \id_N)) \supseteq \ker(f \otimes_R \id_N).
\]
Hence $g$ dominates $f$. The converse follows by 
\cite[\href{https://stacks.math.columbia.edu/tag/059D}{Tag 059D}]{stacks-project}.

$(3)$ follows by \cite[\href{https://stacks.math.columbia.edu/tag/0AUM}{Tag 0AUM}]{stacks-project}.

$(4)$ Since $\varphi$ is pure, for all $R$-modules $N$, $\ker(\varphi \otimes_R \id_N) = 0$. Thus,
\[
\ker(g \otimes_R \id_N) = \ker((\varphi \otimes_R \id_N) \circ (f \otimes_R \id_N)) =
\ker(f \otimes_R \id_N),
\]
that is, $g$ and $f$  dominate each other.

$(5)$ Since $h$ factors through $g$ and $g$ factors through $f$, for all $R$-modules 
$N$, we have
\[
\ker(h \otimes_R \id_N) = \ker(f \otimes_R \id_N) \subseteq \ker(g \otimes_R \id_N)
\subseteq \ker(h \otimes_R \id_N) = \ker(f \otimes_R \id_N).
\]
Thus, $g$ and $h$  dominate each other and  $g$ and $f$  
dominate each other.

$(6)$ Note $g$ dominates $h \circ f$ by (2). On the other hand, for all $R$-modules $N$
\[
\ker((h \circ f) \otimes_R \id_N) = \ker((h \otimes_R \id_N) \circ (f \otimes_R \id_N)) \supseteq \ker(f \otimes_R \id_N) = \ker(g \otimes_R \id_N),
\]
where the last equality follows because $f$ and $g$  dominate each other. This shows that $h \circ f$ dominates $g$.

$(7)$ is \cite[Lem.\ 1.1]{Lazard}. It is the key lemma that leads to a proof of Lazard's Theorem on flat modules being precisely the ones that are filtered colimits of free modules of finite rank \cite[Thm.\ 1.2]{Lazard}.
    
    $(8)$ Since tensor product is right exact, for every $R$-module
    $N$, the induced map
    $
    \varphi \otimes_R \id_N \colon M' \otimes_R N \to M \otimes_R N
    $
    is also surjective. 
    Now note that 
    $\ker\left((g \circ \varphi) \otimes_R \id_N\right)$ 
    (resp. $\ker\left((f \circ \varphi) \otimes_R \id_N\right)$) is
    the pre-image of $\ker(g \otimes_R \id_N)$ (resp. 
    $\ker(f \otimes_R \id_N)$) along the surjective map 
    $\varphi \otimes_R \id_N$. Thus, we must have 
    \[
    \ker(g \otimes_R \id_N) \supseteq \ker(f \otimes_R \id_N)
    \Longleftrightarrow \ker\left((g \circ \varphi) \otimes_R \id_N\right) \supseteq \ker\left((f \circ \varphi) \otimes_R \id_N\right).
    \]
    This, $g$ dominates $f$ if and only if $g \circ \varphi$ dominates
    $f \circ \varphi$.
\end{proof}

\begin{corollary}
    \label{cor:mutual-domination-pure-loci}
    Let $R$ be a ring and $f \colon M \to P$ and $g \colon M \to Q$ be $R$-linear maps. If $g$ dominates $f$, then
$
\textrm{$\{\p \in \Spec(R) \colon g_\p$ is $R_\p$-pure$\} \subseteq \{\p \in \Spec(R) \colon f_\p$ is $R_\p$-pure$\}$}.
$
\end{corollary}

\begin{proof}
    By \autoref{lem:domination} (3) the map $f'$ in the pushout square
     \[
        \xymatrix{ M \ar[r]_ f \ar[d]_ g &  Q \ar[d]^{g'} \\  P \ar[r]^{f'} &  N }
        \]
    is $R$-pure. Hence for all $\p \in \Spec(R)$, $f'_\p$ is a pure map of $R_\p$-modules. Thus, if $g_\p$ is $R_\p$-pure, then 
    $
    g'_\p \circ f_\p = f'_\p \circ g_\p
    $
    is $R_\p$-pure as well, and consequently, $f_\p$ is a pure map of $R_\p$-modules.
\end{proof}

\begin{definition}
\label{def:stabilizer}
\cite[Part II, D\'ef.\ 2.1.3]{rg71}
Let $R$ be a ring and $f \colon M \to P$ and $g \colon M \to Q$ be $R$-linear maps.
Assume that both $M$ and $P$ are finitely presented $R$-modules.
We say
 $f$ \emph{stabilizes} $g$ (or $f$ is a \emph{stabilizer} of $g$) if
    $f$ and $g$  dominate one another. We will say that $g$ \emph{admits a stabilizer} if there exists an $f$ (with the same domain as $g$ and a finitely presented codomain) that stabilizes $g$.
\end{definition}

We have the following:

\begin{lemma}
\label{lem:stabilizers}
Let $R$ be a ring and $g \colon M \to Q$ be a map of $R$-modules such that $M$ is a finitely presented $R$-module. Then we have the following:
\begin{enumerate}
    \item[$(1)$] If $f \colon M \to P$ is a stabilizer of $g$, then $g$ factors through
    $f$.
    
    \item[$(2)$] Let $\{(L_i, u_{ij}) \colon i \in I\}$ be a direct system of finitely presented $R$-modules indexed by a filtered poset $(I, \leq)$ such that
    $
    Q = \colim_i L_i.
    $
    Let $u_i \colon L_i \to Q$ be the associated map for each $i \in I$. If $g$ admits a stabilizer (resp. admits a stabilizer that factors through $g$), then there exists an index $i \in I$ such that for  all $j \geq i$, we have a map 
    $f_j \colon M \to L_j$ that satisfies
    $
    g = u_j \circ f_j, \hspace{2mm} f_j = u_{ij} \circ f_i.
    $
    Furthermore, $f_j$ stabilizes $g$ (resp. $f_j$ stabilizes $g$ and factors through $g$).
    
    \item[$(3)$] In the situation of $(2)$, suppose the index $i \in I$ and the maps $f_j \colon M \to L_j$ for $j \geq i$ are chosen to satisfy the conclusion of $(2)$. Then for any $R$-module $N$ and for all $j \geq i$,
    \[
    \im\left(\Hom_R(L_j, N) \xrightarrow{- \circ f_j} \Hom_R(M, N)\right) = 
    \im\left(\Hom_R(L_i, N) \xrightarrow{- \circ f_i} \Hom_R(M, N)\right)
    \]
\end{enumerate}
\end{lemma}

\begin{proof}
$(1)$ If $f$ stabilizes $g$ then by definition $P$ is a finitely presented $R$-module. Then $g$ dominates $f$ and $\coker(f)$ is finitely
presented because this cokernel is the quotient of a finitely presented module by a finitely
generated submodule; see \cite[\href{https://stacks.math.columbia.edu/tag/0519}{Tag 0519} (4)]{stacks-project}. Thus, $g$ factors through $f$ by \autoref{lem:domination}$(2)$.

$(2)$ Let $f \colon M \to P$ be a stabilizer of $g$. Then $g$ factors through $f$
by $(1)$. Hence choose a map $\varphi \colon P \to Q$ such that
$
g = \varphi \circ f.
$
Since $M$ and $P$ are finitely presented and $Q$ is the
filtered colimit of finitely presented modules $L_i$, there exists an index $i \in I$
such that $\varphi \colon P \to Q$ admits a lift to a map
    $\varphi_i \colon P \to L_i$ 
    along $u_i$, that is, 
    \[
    P \xlongrightarrow{\varphi} Q = P \xlongrightarrow{\varphi_i} L_i \xlongrightarrow{u_i} Q.
    \]
    
For all $j \geq i$, define
\[
f_j \coloneqq M \xlongrightarrow{f} P \xlongrightarrow{\varphi_i} L_i \xlongrightarrow{u_{ij}} L_j.
\]
Then
\[
u_j \circ f_j = ((u_j \circ u_{ij}) \circ \varphi_i ) \circ f =  (u_i \circ \varphi_i) \circ f = \varphi \circ f = g
\]
by the choice of the index $i$. Moreover, since $u_{ii} = \id_{L_i}$, it follows
that $f_i = \varphi_i \circ f$, and so,
$
f_j = u_{ij} \circ f_i
$
by definition of $f_j$ for all $j \geq i$. It remains to show that $f_j$ stabilizes $g$.

Since $g$ factors through $f_j$, it follows that $g$ dominates $f_j$ by 
\autoref{lem:domination}$(2)$. On the other hand, for any $R$-module $N$
\[
\ker(f_j \otimes_R \id_N) = \ker((u_{ij} \circ \varphi_i \circ f) \otimes_R \id_N) \supseteq \ker(f \otimes_R \id_N) = \ker(g \otimes_R \id_N),
\]
where the last equality follows because $f$ stabilizes $g$. This shows that $f_j$
dominates $g$, and so, $g$ and $f_j$  dominate each other, that is, $f_j$ stabilizes $g$, as desired.

Now suppose the stabilizer $f$ also factors through $g$, that is, there exists $\phi \colon Q \to P$ such that
$
f = \phi \circ g.    
$
For the index $i \in I$ and the maps $f_j \colon M \to L_j$ for $i \leq j$ chosen as above, if we define
\[
\phi_j \coloneqq Q \xrightarrow{\phi} P \xrightarrow{\varphi_i} L_i \xrightarrow{u_{ij}} L_j,   
\]
then 
$
\phi_j \circ g = (u_{ij} \circ \varphi_i \circ \phi) \circ g = u_{ij} \circ \varphi_i \circ f = f_j,
$
where the last equality follows by the definition of $f_j$. Thus $f_j$ stabilizes $g$ and factors through $g$  for all $j \geq i$. 

$(3)$ For brevity, we will write the map $\Hom_R(L_j, N) \xrightarrow{- \circ f_j} \Hom_R(M, N)$ as $\Hom_R(f_j, N)$. In this new notation, we have to show that for all
$j \geq i$, 
\[
\im\left(\Hom_R(f_j, N)\right) = \im\left(\Hom_R(f_i, N)\right).
\]
Since $f_j = u_{ij} \circ f_i$ for all $j \geq i$, it follows that
\[
\im\left(\Hom_R(f_j,N)\right) = \im\left(\Hom_R(f_i,N) \circ \Hom_R(u_{ij},N)\right)
\subseteq \im\left(\Hom_R(f_i,N)\right).
\]
For all $j \geq i$, we have a commutative diagram
\[
    \xymatrix{&M \ar[dl]_{f_i} \ar[d]_{f_j} \ar[dr]^g  \\
    L_i \ar[r]_{u_{ij}} &L_j \ar[r]_{u_j} &Q.}
    \]
Since $f_i$ and $g$  dominate each other, by \autoref{lem:stabilizers} (5), for all $j \geq i$, $f_i$ and $f_j$ also  dominate each other. 
Since $\coker(f_j)$
is finitely presented, it follows that $f_i$ factors through $f_j$, that is, there
exists $v_{ji} \colon L_j \to L_i$ such that 
\[
M \xlongrightarrow{f_i} L_i = M \xlongrightarrow{f_j} L_j \xlongrightarrow{v_{ji}} L_i.
\]
Then
$
\im\left(\Hom_R(f_i,N)\right) = \im\left(\Hom_R(f_j,N) \circ \Hom_R(v_{ji},N)\right)
\subseteq \im\left(\Hom_R(f_j,N)\right).
$
\end{proof}

Using the previous lemmas, one can draw the following interesting consequence for stabilizers of maps to a flat module over a local ring. The result is \cite[Part II, Lem.\ 2.1.9]{rg71}, and its proof is provided for the convenience of the reader.

\begin{proposition}\cite[Part II, Lem.\ 2.1.9]{rg71} (cf. \autoref{lem:domination} (7))
\label{prop:RG-2.1.9}
Let $(R, \fram, \kappa)$ be a local ring (not necessarily Noetherian). Let $M$ be a flat $R$-module and $g \colon F \to M$ be an $R$-linear map, where $F$ is a finitely presented $R$-module. If $g$ admits a stabilizer, then there exists a submodule $L$ of $M$ such that all the following conditions are satisfied:
\begin{enumerate}
    \item[$(1)$] $L$ is free of finite rank.
    \item[$(2)$] $L \subseteq M$ is a pure extension.
    \item[$(3)$] $\im(g) \subseteq L$.
\end{enumerate}
If $L$ satisfies $(1)-(3)$, then the map $f \colon F \twoheadrightarrow \im(g) \hookrightarrow L$ stabilizes $g$.
\end{proposition}

\begin{proof}
If $L$ satisfies $(1)-(3)$ then the map $f$ stabilizes $g$ by \autoref{lem:domination} (4). So now we show the existence of such an $L$. 

By Lazard's Theorem \cite[Thm.\ 1.2]{Lazard}, $M$ is a colimit of a system $\{(L_{i}, u_{ij}) \colon i, j \in I, i\leq j\}$ of free $R$-modules of finite rank indexed by a filtered poset $(I,\leq)$. Here $u_{ij}$ denotes the transition map $L_i \to L_j$ for $i \leq j$. We also let $u_i \colon L_i \to M$ denote the associated maps for all $i$. Since $g$ admits a stabilizer, by \autoref{lem:stabilizers} parts (2) and (3), there exists an index $i \in I$ such that for all $j \geq i$, we have a map
$
f_j \colon F \to L_j
$
such that
$g = u_j \circ f_j, f_j = u_{ij} \circ f_i$, $f_j$ stabilizes $g$ and
$
\im(\Hom_R(f_j,R)) = \im(\Hom_R(f_i,R)).
$
Let us call this stable image as $S$. Note that $S$ is a finitely generated $R$-module, since it is a quotient of the free module $\Hom_R(L_i,R)$. Let $n$ be the minimal number of generators of $S$, that is,
$
n = \dim_k(S/\fram S).
$
We then have a surjective $R$-linear map
$
p \colon R^{\oplus n} \twoheadrightarrow S.
$
By the choice of $n$, the induced map
$
p \otimes_R \id_\kappa \colon \kappa^{\oplus n} \to S/\fram S
$
is an isomorphism.
Consider the diagram
\[
\begin{tikzcd}
   &\Hom_R(L_i,R) \arrow[d, twoheadrightarrow] \\
  R^{\oplus n} \arrow[r, twoheadrightarrow, "p"]
&S \end{tikzcd}
\]
The vertical map is just the one obtained by restricting the codomain of $\Hom_R(f_i,R)$ to $S = \im(\Hom_R(f_i,R))$. Since $\Hom_R(L_i,R)$ is a free, hence projective $R$-module, there exists an $R$-linear map 
$q \colon \Hom_R(L_i,R) \to R^{\oplus n}$ 
such that
\begin{equation}
\label{eq:pq}
\Hom_R(L_i,R) \twoheadrightarrow S = p \circ q.
\end{equation}
After tensoring by $\kappa$ and using that $p \otimes_R \id_\kappa$ is an isomorphism, it follows that $q \otimes_R \id_\kappa$ is also surjective. Then by Nakayama's lemma, we see that $q$ is surjective. Since the composition
\[
\Hom_R(L_j,R) \xrightarrow{\Hom_R(u_{ij}, R)} \Hom_R(L_i,R) \twoheadrightarrow S
\]
is surjective for all $j \geq i$ ($S$ is the stable image $\im(\Hom_R(f_j,R))$ for all $j \geq i$), it follows by the exact same reasoning as above that for all $j \geq i$,
\begin{equation}
    \label{eq:surj-free-modules}
q \circ \Hom_R(u_{ij}, R) \colon \Hom_R(L_j,R) \to R^{\oplus n}
\end{equation}
is also surjective. Applying $\Hom_R(\hspace{2mm}, R)$ to the maps
$
R^{\oplus n} \xlongrightarrow{p} S \hookrightarrow \Hom_R(F,R)
$
and
$
\Hom_R(L_i,R) \xrightarrow{q} R^{\oplus n},
$
we get maps
\[
s \coloneqq F \to \Hom_R(\Hom_R(F,R),R) \to \Hom_R(S,R) \xhookrightarrow{\Hom_R(p,R)} \Hom_R(R^{\oplus n}, R)
\]
and
\[
t \coloneqq \Hom_R(R^{\oplus n}, R) \xhookrightarrow{\Hom_R(q,R)} \Hom_R(\Hom_R(L_i,R),R) \xrightarrow{\simeq} L_i.
\]
Here $F \to \Hom_R(\Hom_R(F,R),R)$ is the canonical map to the double dual, and we can make the identification $\Hom_R(\Hom_R(L_i,R),R) \cong L_i$ since $L_i$ is free of finite rank. By \autoref{eq:pq},
$
t\circ s = f_i.
$
Then $u_i \circ (t \circ s) = u_i \circ f_i = g$, and so,
\begin{equation}
\label{eq:image}
\im(g) \subseteq \im(u_i \circ t).
\end{equation}
We also have
$
u_i \circ t = \colim_{j \geq i} u_{ij} \circ t,
$
where each $u_{ij} \circ t$ admits a left-inverse because it can be identified with the map obtained by applying $\Hom_R(\hspace{2mm}, R)$ to the surjective map of free modules \autoref{eq:surj-free-modules}
$
q \circ \Hom_R(u_{ij},R)
$
which always has a right-inverse. Since a filtered colimit of split maps is pure, it follows that 
$u_i \circ t \colon 
\Hom_R(R^{\oplus n},R) \to M$ 
is a pure map. Taking $L \coloneqq \im(u_i \circ t) \cong R^{\oplus n}$, 
\autoref{eq:image} shows that $L$ satisfies properties $(1)-(3)$.
\end{proof}

\subsection{Mittag-Leffler and strictly Mittag-Leffler modules}
The notion of a Mittag-Leffler module was introduced by Raynaud and Gruson to study faithfully flat descent of projectivity \cite[Part II]{rg71}.  There is a long  history of Mittag-Leffler modules, and, in particular, there are many equivalent definitions/characterizations for it. We choose Raynaud and Gruson's original definition, in part because this perspective makes the connection between the Mittag-Leffler property and pure (aka universally injective) maps of modules clearest.

\begin{definition}\cite[Part II, D\'ef. 2.1.3]{rg71}
    \label{def:Mittag-Leffler}
    Let $R$ be a ring. An $R$-module $M$ is \emph{Mittag-Leffler}, abbreviated ML, if any finitely presented $R$-module $P$ and $R$-linear map $g: P \to M$ admits a stabilizer. That is, there exists a finitely presented module $Q$ and an $R$-linear map $f: P \to Q$ such that $f$ and $g$  dominate each other.
\end{definition}

\begin{example}
    Any finitely presented $R$-module is automatically ML.
\end{example}

In \autoref{def:Mittag-Leffler}, since $f \colon P \to Q$ is a map of finitely presented $R$-modules,
$\coker(f)$ is also a finitely presented $R$-module by \cite[\href{https://stacks.math.columbia.edu/tag/0519}{Tag 0519} (4)]{stacks-project} because
$\im(f)$ is a finitely generated $R$-module and one has a short exact sequence
\[
0 \to \im(f) \to Q \to \coker(f) \to 0.
\]
Hence by \autoref{lem:domination}$(2)$,
the map $g \colon P \to M$ factors through $f$. However, $f \colon P \to Q$ does not necessarily factor through $g$ because $\coker(g)$ is not necessarily finitely presented -- in fact, $\coker(g)$
may not even be finitely generated if $M$ is not finitely generated. 
This observation leads to the strengthening of the notion of a Mittag-Leffler module.

\begin{definition}\cite[Part II, D\'ef. 2.3.1]{rg71}
\label{def:strict-Mittag-Leffler}
An $R$-module $M$ is \emph{strictly Mittag-Leffler}, abbreviated SML, if for any finitely presented $R$-module $P$ and
any $R$-linear map $g \colon P \to M$, there exists a finitely presented $R$-module $Q$ and
an $R$-linear map $f \colon P \to Q$ such that $g$ factors through $f$ and $f$ factors through
$g$.
\end{definition}

\begin{remark}
    \label{rem:Mittag-Leffler}
    Let $R$ be a ring and $M$ be an $R$-module.
    \begin{enumerate}
      \item If $M$ is a SML $R$-module then $M$ is a ML $R$-module by \autoref{lem:domination}$(2)$. The converse holds for modules over a Noetherian complete local ring; see \autoref{prop:ML-SML}.
      
      \item One can show that an $R$-module $M$ is ML if and only if for any family of $R$-modules $\{Q_i\}_{i \in \Lambda}$, the canonical map
      $
      \left( \prod _{i \in \Lambda } Q_{i} \right) \otimes_R M \to \prod _{i \in \Lambda } (Q_{i} \otimes_R M)
      $
      is injective \cite[Part II, Prop.\ 2.1.5]{rg71} (alternate reference \cite[\href{https://stacks.math.columbia.edu/tag/059M}{Tag 059M}]{stacks-project}). This alternate characterization implies that the property of being Mittag-Leffler is preserved under arbitrary base change. That is, if $M$ is a ML $R$-module and $R \to S$ is a ring map, then $S \otimes_R M$ is a ML $S$-module. Indeed, if $\{Q_i\}_{i \in \Lambda}$ is a collection of $S$-modules, then the natural map
      \[
      \left(\prod_{i \in \Lambda} Q_i\right) \otimes_S (S \otimes_R M) \to \prod_{i \in \Lambda} (Q_i \otimes_S (S \otimes_R M))
      \]
      can be identified with
      \[
      \left(\prod_{i \in \Lambda} Q_i\right) \otimes_R M \to \prod_{i \in \Lambda} (Q_i \otimes_R M).
      \]
      The latter map is injective because $M$ is a ML $R$-module.
      
      \item If $M$ is a ML (resp. SML) $R$-module and $N$ is a pure submodule of $M$, then $N$ is a ML (resp. SML) $R$-module. Indeed, if $g \colon P \to N$ is an $R$-linear map such that $P$ is finitely presented, then the composition $g' \coloneqq P \xrightarrow{g} N \subseteq M$ admits a stabilizer $f \colon P \to Q$ (resp. admits a stabilizer that $g'$ factors) because $M$ is ML (resp. is SML). Purity of the inclusion $N \subseteq M$ implies that for all $R$-modules $T$, $\ker(g \otimes_R \id_T) = \ker(g' \otimes_R \id_T) = \ker(f \otimes_R \id_T)$. Thus, $f$ and $g$  dominate each other, and so, $f$ factors $g$ by \autoref{lem:domination} (2) because $\coker(f)$ is finitely presented. If $g'$ factors $f$, then $g$ also factors $f$ because $g$ factors $g'$ by definition $g'$. Thus, $N$ is ML (resp. is SML).
      
      \item As a consequence of \autoref{lem:domination} (8),
      we have that an $R$-module $M$ is ML if and only if for every
      free module $F$ of finite rank, every $R$-linear map 
      $g \colon F \to M$ admits a stabilizer. In other words, in 
      \autoref{def:Mittag-Leffler} it suffices to assume that
      $P$ is a free module of finite rank. The key point here is that
      if $g' \colon P \to M$ is an $R$-linear map from a finitely 
      presented $R$-module $P$ that is not necessarily free, 
      then choosing an $R$-linear surjection
      $\varphi \colon F \twoheadrightarrow P$ from a free module 
      $F$ of finite rank, one sees that if $f \colon F \to Q$ is a map
      of finitely presented $R$-modules such that $f$ and $g' \circ \varphi$  dominate one another, then $f$ factors through $P$ along the surjective map $\varphi$.
      That is, there exists an $R$-linear map $f' \colon P \to Q$ such
      that $f = f' \circ \varphi$. This is because the hypothesis of 
      mutual domination between $f$ and $g' \circ \varphi$ ensures
      that $\ker(f) = \ker(g' \circ \varphi)  \supseteq \ker(\varphi)$.
      Since $f = f' \circ \varphi$ and $g' \circ \varphi$  dominate each other and $\varphi$
      is surjective, $f'$ and $g'$ also  dominate
      one another by \autoref{lem:domination} (8). 

      \item Let $\{(L_{i}, \varphi_{ij}) \colon i, j \in I, i \leq j\}$ be a system of $R$-modules indexed by a filtered poset $(I, \leq)$ such that each $L_i$ is a ML $R$-module and the transition maps $\varphi_{ij}$ are pure. Then $M \coloneqq \colim_{i \in I} L_i$ is also a ML $R$-module. Indeed, the associated maps $\varphi_i \colon L_i \to M$ are $R$-pure (being a filtered colimit of pure maps). If $P$ is a finitely presented $R$-module and $g \colon P \to M$ is an $R$-linear map, then there exists $i \in I$ and $g_i \colon P \to L_i$ such that
      $ g = \varphi_i \circ g_i.
        $
      Then $g_i$ and $g$  dominate each other by \autoref{lem:domination} (4).  
      Since $L_i$ is ML, $g_i$ admits a stabilizer $f \colon P \to Q$, that is, $Q$ is finitely presented and $g_i$ and $f$  dominate each other. Hence $g$ and $f$ also  dominate each other, that is, $f$ stabilizes $g$.

      \item As an application of (b) one can show the following fact about the behavior of ML modules under restriction of scalars. Let $R \to S$ be a ML ring map and let $M$ be a flat and ML $S$-module. Then $M$ is a ML $R$-module. The interested reader can see \cite[\href{https://stacks.math.columbia.edu/tag/05CT}{Tag 05CT}]{stacks-project} for a proof.

      \item Let $\varphi \colon P \to M$ be a pure $R$-linear map, where $P$ is finitely presented and $M$ is SML. Then $\varphi$ splits. Indeed, let $Q$ be a finitely presented $R$-module and $\phi \colon P \to Q$ be a linear map such that $\varphi$ factors through $\phi$ and $\phi$ factors through $\varphi$. Since $\varphi$ is $R$-pure and factors through $\phi$, it follows that $\phi$ is also $R$-pure. Since the cokernel of a map of finitely presented modules is also finitely presented, it follows that $\phi$ splits (\autoref{lem:pure-iff-split}), that is, there exists $u \colon Q \to P$ such that $u \circ \phi = \id_P$. Let $f \colon M \to Q$ be a linear map such that
      $
        \phi = f \circ \varphi.
      $
      Then $\id_P = u \circ \phi = (u \circ f) \circ \varphi$, that is, $\varphi$ splits.

      \item A free $R$-module $F$ is SML and hence is also Mittag-Leffler. Indeed, let $P$ be a finitely presented $R$-module and $g \colon P \to F$ be a linear map. Since $\im(g)$ is finitely generated, let $Q$ be a finitely generated free summand of $F$ that contains $\im(g)$ and let $f \colon P \to Q$ be the composition $P \xrightarrow{g} \im(g) \hookrightarrow Q$. By construction, $g$ factors through $f$ via the inclusion $Q \hookrightarrow F$. Also, since $Q$ is a summand of $F$, choosing a left-inverse $u \colon F \to Q$ of $Q \hookrightarrow F$, we get $u \circ g = f$. Thus, $f$ factors through $g$. 

      \item A consequence of (h) and (c) is that any projective $R$-module is SML and hence is also ML. Furthermore, if $M$ admits a countable generating set as an $R$-module, then $M$ is projective if and only if $M$ is a flat SML $R$-module \cite[Part II, Cor.\ 2.2.2]{rg71}.
    \end{enumerate}
    \end{remark}

    Let $\{(L_i, \varphi_{ij}) \colon i, j \in I, i \leq j\}$ be a direct system of $R$-modules indexed by a filtered poset $(I, \leq)$, and let 
    $
    M \coloneqq \colim_{i \in I} L_i.    
    $
    For all $i \in I$, let $\varphi_i \colon L_i \to M$ be the canonical maps. Then for any $R$-module $N$, we get a filtered inverse system of $R$-modules $\{(\Hom_R(L_i, N), \Hom_R(\varphi_{ij}, N)) \colon i, j \in I, i \leq j\}$. Note that for all $j \geq j$, 
    $
    \Hom_R(\varphi_i, N) = \Hom_R(\varphi_{ij}, N) \circ \Hom_R(\varphi_j, N),    
    $
    and so,
    $
        \im(\Hom_R(\varphi_i, N)) \subseteq  \bigcap_{j \geq i} \im(\Hom_R(\varphi_{ij}, N)).    
    $
    
    There are two natural questions one can ask about the above inverse system (that depends on $N$):
    \begin{enumerate}
        \item[($\bigstar$)]\label{it:MLstar} For each $i \in I$, does $\bigcap_{j \geq i} \im(\Hom_R(\varphi_{ij}, N))$ stabilize? That is, is there some $j_0 \in I$ (depending on $i$) such that $j_0 \geq i$ and
        $
            \bigcap_{j \geq i} \im(\Hom_R(\varphi_{ij}, N)) = \im(\Hom_R(\varphi_{ij_0}, N))?
        $
        Such a $j_0$ exists if and only if there exists a $j_0$ such that $j_0 \geq i$ and for all $k \geq j_0$,
        $
        \im(\Hom_R(\varphi_{ij_0}, N)) = \im(\Hom_R(\varphi_{ik}, N)).    
        $

        \item[($\bigstar\bigstar$)] A stronger question one can ask is: does  $\bigcap_{j\geq i} \im(\Hom_R(\varphi_{ij}, N))$ stabilize for each $i \in I$ and, moreover, does this stable set coincide with $\im(\Hom_R(\varphi_i, N))$? Equivalently, for each $i \in I$ is there a $j_0 \in I$ (depending on $i$) such that $j_0 \geq i$ and
        $
            \im(\Hom_R(\varphi_i, N)) =  \im(\Hom_R(\varphi_{ij_0}, N))?
        $
        Such a $j_0$ exists if and only if there exists a $j_0$ such that $j_0 \geq i$ and for all $k \geq j_0$,
        $
            \im(\Hom_R(\varphi_i, N)) =  \im(\Hom_R(\varphi_{ik}, N)).
        $
    \end{enumerate}

    The difference between ($\bigstar$) and ($\bigstar\bigstar$) is the difference between a ML $R$-module and a SML $R$-module.

    \begin{theorem}
        \label{thm:ML-systems}
        Let $\{(L_i, \varphi_{ij}) \colon i, j \in I, i \leq j\}$ be system of \emph{finitely presented} $R$-modules indexed by a filtered poset $(I, \leq)$, and let 
    $
    M \coloneqq \colim_{i \in I} L_i.    
    $
    For all $i \in I$, let $\varphi_i \colon L_i \to M$ be the canonical maps. Then we have the following:
    \begin{enumerate}
        \item[$(1)$] $M$ is ML if and only if for all $R$-modules $N$ and for all $i \in I$, $\bigcap_{j \geq i} \im(\Hom_R(\varphi_{ij}, N))$ stabilizes in the sense of \emph{($\bigstar$)}.
        \item[$(2)$] $M$ is SML if and only if for all $R$-modules $N$ and for all $i \in I$, $\bigcap_{j \geq i} \im(\Hom_R(\varphi_{ij}, N))$ stabilizes to $\im(\Hom_R(\varphi_i, N))$ in the sense of \emph{($\bigstar\bigstar$)}.
    \end{enumerate}
    \end{theorem}

    \begin{proof}
        (1) follows from \cite[Part II, Prop. 2.1.4$(i)\Leftrightarrow(iii)$]{rg71} (alternate reference \cite[\href{https://stacks.math.columbia.edu/tag/059E}{Tag 059E}]{stacks-project}). Note that the implication $\implies$ in (1) follows from \autoref{lem:stabilizers}. The assertion in (2) follows from \cite[Part II, Prop.\ 2.3.2$(i)\Leftrightarrow(iii)$]{rg71}.
    \end{proof}

    Recall that Lazard's Theorem characterizes flat $R$-modules as precisely those that can be written as a filtered colimit of free modules of finite rank. In general, one cannot guarantee the maps from these free modules to the flat module to be well-behaved in any way (for example, one cannot guarantee that a flat $R$-module is a filtered union of free submodules of finite rank). However, \autoref{prop:RG-2.1.9} has the following straightforward consequence for flat ML modules and flat SML modules over local rings.

\begin{proposition}
    \label{prop:ML-SML-modules-colimit-pure-split-free}
    Let $(R, \fm, \kappa)$ be a local ring (not necessarily Noetherian) and let $M$ be an $R$-module. We have the following:
    \begin{enumerate}
        \item[$(1)$] $M$ is flat and ML if and only if $M$ is a filtered union of finitely generated free submodules that are pure in $M$.
        
        \item[$(2)$] $M$ is flat and SML if and only if $M$ is a filtered union of finitely generated free submodules that are direct summands of $M$.
    \end{enumerate}
    
\end{proposition}

\begin{proof}
    (1) A filtered colimit of free modules is flat. Thus, the `if' implication follows by \autoref{rem:Mittag-Leffler} (e) because free modules of finite rank are finitely presented, and hence, ML. Now suppose $M$ is a flat ML $R$-module. We know that $M$ is a filtered union of its finitely generated submodules. Thus, it suffices to show that if $N$ is a finitely generated submodule of $M$, then there exists a submodule $L$ of $M$ such that $N \subseteq L$, $L$ is free of finite rank and $L \hookrightarrow M$ is pure. But this follows upon choosing a map $g \colon F \to M$ from a free module $F$ of finite rank that surjects onto $N$ and observing that since $g$ admits a stabilizer ($M$ is ML), it must admit one $f \colon F \to L$ where $L$ has the desired properties by \autoref{prop:RG-2.1.9}.

    (2) Since a SML module is ML, it follows by (a) that a flat and SML $R$-module is a filtered union of finitely generated free submodules that are pure in $M$. But such a submodule must also be direct summand of $M$ by \autoref{rem:Mittag-Leffler} (g). This proves the forward implication. Conversely, suppose $M$ is a filtered union of finitely generated free submodules that are direct summands of $M$. Then $M$ is flat since it is a filtered colimit of free modules. Moreover, free modules are SML by \autoref{rem:Mittag-Leffler} (h). 
    
    Thus, it suffices to show that if $M$ is a filtered union of SML submodules that are direct summands of $M$, then $M$ is also SML. Let $P$ be a finitely presented $R$-module and $g \colon P \to M$ be $R$-linear. Then there exists a SML direct summand $M'$ of $M$ such that $\im(g) \subseteq M'$. Let $u \colon M \to M'$ be a left-inverse of the inclusion $M' \hookrightarrow M$. Let $g' \colon P \to M'$ be the composition $P \xrightarrow{g} \im(g) \hookrightarrow M'$. Since $M'$ is SML, there exists a finitely presented $R$-module $Q$ and a map $f \colon P \to Q$ such that $g'$ factors through $f$ and $f$ factors through $g'$. Since $g$ factors through $g'$, it follows that $g$ factors through $f$. Furthermore, since $g'$ factors through $g$ via $u$, it follows that $f$ also factors through $g$. Thus, $M$ is SML.
\end{proof}

    The next result shows that the ML and SML properties coincide for modules over a Noetherian complete local ring.

    \begin{proposition}
        \label{prop:ML-SML}
        Let $(R, \fm)$ be a Noetherian local ring that is $\fm$-adically complete. Let $M$ be an $R$-module. Then the following are equivalent:
        \begin{enumerate}
        \item[$(1)$] $M$ is a ML $R$-module.
        
        \item[$(2)$] $M$ is a SML $R$-module.
        \end{enumerate}
    \end{proposition}
    
    \begin{proof}
        Since we are working over a Noetherian ring, finitely generated
        is the same as being finitely presented for modules.
        
        Given \autoref{rem:Mittag-Leffler} (a), for the equivalence of $(1)$ and $(2)$ it suffices to show that if $M$ is ML then $M$ is SML. Let $P$ be a finitely generated $R$-module and let
        $g \colon P \to M$ be an $R$-linear map. Since $M$ is ML, there exists a finitely generated $R$-module $Q$
        and an $R$-linear map $f \colon P \to Q$ such that $g$ and $f$ 
        dominate one another. Thus, if one forms the pushout square
        \[
        \xymatrix{P \ar[r]_ f \ar[d]_ g &  Q \ar[d]^{g'} \\  M \ar[r]^{f'} &  N ,}
        \]
        then both $f'$ and $g'$ are pure maps of $R$-modules by \autoref{lem:domination} (3).
        Since $\coker(f)$ is finitely presented, it follows by
        \autoref{lem:domination}$(2)$ that $g$ factors through $f$.
        Thus, it remains to show by \autoref{def:strict-Mittag-Leffler}
        that $f$ factors through $g$.
        
        The map $g' \colon Q \to N$ is a pure map of modules over a Noetherian complete local ring such that $Q$ is a finitely generated $R$-module. Thus,
        by \autoref{lem:Auslander-Warfield-lemma}, $g'$ 
        admits an $R$-linear left-inverse 
$\varphi \colon N \to Q,$ 
        that is, $\varphi \circ g' = \id_Q$. Then by
        the commutativity of the above diagram, 
       $
        (\varphi \circ f') \circ g = \varphi \circ (f' \circ g) = \varphi \circ (g' \circ f) = (\varphi \circ g') \circ f = f.
        $
        That is, $f$ factors through $g$ as desired.
    \end{proof}

\subsection{Cyclic domination and cyclic stabilizers}

Recall that we say that a map of $R$-modules $M \to N$ is 
\emph{cyclically pure} if for all cyclic $R$-modules $P$, the induced
map $M \otimes_R P \to N \otimes_R P$ is injective.

It is clear that pure maps are always cyclically pure. We have the following general situation where the converse holds; see also \cite{HochsterCyclicPurity} for other non-trivial instances where cyclic purity implies purity.

\begin{lemma}
\label{lem:cyclic-purity-flat}
Let $R$ be a ring and $\varphi \colon M \to N$ be an $R$-linear map such that $N$ is 
a flat $R$-module. Then $\varphi$ is cyclically pure if and only if
$\varphi$ is pure.
\end{lemma}

\begin{proof}
This is \cite[\href{https://stacks.math.columbia.edu/tag/0AS5}{Tag 0AS5}]{stacks-project}. 
\end{proof}

Given that the notion of domination is intimately tied to the notion of purity (see \autoref{lem:domination} (3)), it natural to develop a cyclic version of domination that will relate to cyclic purity.

\begin{definition}
\label{def:cyclic-domination}
Let $R$ be a ring and $f \colon M \to P$ and $g \colon M \to Q$ be
two maps of $R$-modules. Then $g$ \emph{cyclically dominates} $f$
if for all cyclic $R$-modules $N$, $\ker(f \otimes_R \id_N) \subseteq \ker(g \otimes_R \id_N)$.
\end{definition}

\begin{lemma}
\label{lem:cyclic-stab-pushout}
Let $R$ be a ring and $f \colon M \to P$ and $g \colon M \to Q$ be
$R$-linear maps. Consider the pushout of $f$ and $g$
    \[
    \xymatrix{ M \ar[r]_ f \ar[d]_ g &  P \ar[d]^{g'} \\  Q \ar[r]^{f'} &  T }
    \]
    Then $g$ cyclically dominates $f$ if and only if $f'$ is a cyclically pure map of $R$-modules. 
\end{lemma}

\begin{proof}
The proof of \cite[\href{https://stacks.math.columbia.edu/tag/0AUM}{Tag 0AUM}]{stacks-project} readily adapts to yield the assertion. We omit the details.
\end{proof}

\begin{lemma}
    \label{lem:factor-cyclically-pure}
    Suppose we have a commutative diagram of $R$-modules
    \[
    \xymatrix{&M \ar[dl]_f \ar[dr]^g \\  
    Q \ar[rr]_{\varphi}  && P}
    \]
    where $\varphi$ is cyclically pure. Then $f$ and $g$  cyclically dominate each other.
\end{lemma}

\begin{proof}
    Proof is similar to the proof of \autoref{lem:domination} (4) and is omitted.
\end{proof}

\begin{lemma}
    \label{lem:cyclic-stab-diagram}
    Let $R$ be a ring and consider a commutative diagram of $R$-modules
     \[
        \xymatrix{&M \ar[dl]_f \ar[d]_g \ar[dr]_h  \\
        P \ar[r] &Q \ar[r] &R}
        \]
        where $f$ and $h$  cyclically dominate each other. Then $f, g, h$  cyclically dominate each other in pairs.
    \end{lemma}
    
    \begin{proof}
    One can readily adapt the proof of the analogous result for mutual domination instead of mutual cyclic domination given in \autoref{lem:domination} (6). 
    \end{proof}

\begin{definition}
\label{def:cyclic-stabilizer}
Let $R$ be a ring and $f \colon M \to P$ and $g \colon M \to Q$ be $R$-linear maps,
where $M$ and $P$ are finitely presented $R$-modules. We say
$f$ \emph{cyclically stabilizes} $g$ (or $f$ is a \emph{cyclic stabilizer} of $g$) if 
$f$ and $g$  cyclically dominate each other. We will say $g$ \emph{admits a cyclic stabilizer} if there exists a finitely presented $R$-module $P$ and an $R$-linear map $f \colon M \to P$ such that $f$ cyclically stabilizes $g$.
\end{definition}

\begin{proposition}
    \label{prop:cyclic-stabilizers}
    Let $R$ be a ring and let $\{(L_i,u_{ij}) \colon i, j \in I, i\leq j\}$ be a system of finitely
    presented $R$-modules indexed by a filtered poset $(I, \leq)$. 
    Let
    $
    M \coloneqq \colim_i L_i
    $
    and let $u_i \colon L_i \to M$ denote the associated map for all $i \in I$.
    Suppose $v \colon G \to M$ is a map of $R$-modules where $G$ is finitely
    presented. If $v$ admits a cyclic stabilizer that $v$ dominates,
    then there exists an index $i \in I$ along with a map 
    $
    v_i \colon G \to L_i
 $
    such that $u_i \circ v_i = v$ and for all $j \geq i$, 
    $
    v_j \coloneqq (G \xrightarrow{v_i} L_i \xrightarrow{u_{ij}} L_j)
    $
    is a cyclic stabilizer of $v$ that $v$ dominates.
    \end{proposition}

    \begin{proof}
    Suppose $f \colon G \to P$ is a cyclic stabilizer of $v$ that $v$ dominates. Since
    $P$ is a finitely presented $R$-module, $v$ factors through $f$ (\autoref{lem:domination} (2)), that is, there exists an $R$-linear map
    $
    \varphi \colon P \to M
    $
    such that the following diagram commutes
    \[
    \xymatrix{&G \ar[dl]_f \ar[d]_v\\
    P \ar[r]_\varphi &M.}
    \]
    Since $P$ is a finitely presented $R$-module and $M = \colim_i L_i$ is a filtered colimit of finitely presented modules $L_i$, there
    exists $i \in I$ such that $\varphi$ admits a lift
    $\varphi_i \colon P \to L_i$ along $u_i \colon L_i \to M$, that is, 
    the following diagram commutes
    \[
    \xymatrix{&L_i \ar[d]^{u_i}\\
    P \ar[ur]^{\varphi_i} \ar[r]_{\varphi} &M.
    }
    \]
    Now for all $j \geq i$, define
    $
    v_j \coloneqq (G \xrightarrow{f} P \xrightarrow{\varphi_i} L_i \xrightarrow{u_{ij}} L_j).
    $
    Then by the commutativity of the above diagrams,
    $
    u_j \circ v_j = ((u_j \circ u_{ij}) \circ \varphi_i) \circ f = 
    (u_i \circ \varphi_i) \circ f = \varphi \circ f = v.
    $
    Thus, for all $j \geq i$, $v$ dominates $v_j$ (\autoref{lem:domination} (2)). Moreover, for all
    $j \geq i$, we have a commutative diagram
     \[
        \xymatrix@C=3em@R=3em{&G \ar[dl]_f \ar[d]_{v_j} \ar[dr]^v  \\
        P \ar[r]_{u_{ij} \circ \varphi_i} &L_j \ar[r]_{u_j} &M.}
        \]
        Since $f$ and $v$ cyclically dominate each other, by
        \autoref{lem:cyclic-stab-diagram} we have that $v_j$ and $v$ cyclically
        dominate each other for all $j \geq i$.
    \end{proof}

One now has the following analog of \autoref{lem:stabilizers} for cylic stabilizers.

\begin{corollary}
\label{cor:cyclic-stabilizer-flat}
Let $M$ be a flat $R$-module. Let $v \colon G \to M$ be an $R$-linear map such that $G$ is a finitely presented $R$-module. If $v$
admits a cyclic stabilizer that $v$ dominates, then $v$ admits a cyclic
stabilizer of the form $u \colon G \to F$ such that $v$ dominates
$u$ and $F$ is free of finite rank.
\end{corollary}

\begin{proof}
By Lazard's theorem we can express $M$ as a filtered colimit of
free modules of finite rank. The Corollary now follows by \autoref{prop:cyclic-stabilizers}.
\end{proof}

\begin{remark}
 In \autoref{prop:cyclic-stabilizers} and \autoref{cor:cyclic-stabilizer-flat}, we not only need to assume that $v \colon G \to M$ admits a cyclic stabilizer, but that $v$ admits a cyclic stabilizer that it dominates. In certain situations, domination will come for free from cyclic domination; see, for instance, \autoref{lem:cyclic-honest-domination}.
\end{remark}

\subsection{Content function and Ohm-Rush modules}
\label{subsection:contentfunctionandORmods}
We will now see that cyclic domination and cyclic stabilizers are related to the content function for modules. Recall that for any $R$-module $M$, there is a function
$
c_M \colon M \to \textrm{$\{$ideals of $R\}$}
$
given by
$
c_M(x) \coloneqq \bigcap_{x \in IM} I.
$
Here $I$ denotes an ideal of $R$. The function $c_M$ is called the \emph{content function} on $M$ and the ideal $c_M(x)$ is called the \emph{content} of the element $x$. For a subset $N \subseteq M$ one can similarly define the \emph{content of $N$}, denoted $c_M(N)$, to be the intersection of all ideals $I$ of $R$ such that $N \subseteq IM$. It is clear that if $\langle N \rangle$ is the submodule of $M$ generated by $N$, then $c_M(N) = c_M(\langle N \rangle)$.

\begin{remark}
    \label{rem:content-restriction-of-scalars}
    Let $R \to S$ be a ring map and consider $S$ as an $R$-module by restriction of scalars. Let $J$ be an $R$-submodule of $S$ and let $JS$ denote the ideal of $S$ that is generated by $J$. Since for any ideal $I$ of $R$, $IS$ is an ideal of $S$, it follows that $J \subseteq IS \Longleftrightarrow JS \subseteq IS$. Thus, $c_S(J) = c_S(JS)$. This shows that instead of considering the content of $R$-submodules of $S$, one can look at the content of ideals of $S$. The latter is more natural to consider.
\end{remark}

\begin{definition}
\label{def:Ohm-Rush}
An $R$-module $M$ is \emph{Ohm-Rush} if for all $x \in M$, $x \in c_M(x)M$. A ring homomorphism $R \to S$ is \emph{Ohm-Rush} if $S$ is an Ohm-Rush $R$-module by restriction of scalars.
\end{definition}

In other words, $M$ is \emph{Ohm-Rush} if for all $x \in M$, the collection of ideals $I$ of $R$ such that $x \in IM$ has a unique smallest element under inclusion. 

What we are calling an Ohm-Rush module is called a \emph{content module} in \cite[Def.\ 1.1]{OhmRu-content}, which is where this notion was introduced. Our alternate terminology, first used in 
\cite{EpsteinShapiroOhmRushI, EpsteinShapiroOhmRushII}, honors the contributions
of Ohm and Rush.

\begin{remark}
    \label{rem:content-finitely-generated}
    If $x \in c_M(x)M$, then $c_M(x)$ has to be a finitely generated ideal (we are not assuming $R$ is Noetherian).  See \cite[p. 51]{OhmRu-content}. 
\end{remark}

\begin{lemma}
    \label{lem:content-subset}
    Let $M$ be an Ohm-Rush $R$-module. For any $N \subseteq M$, 
    $c_M(N) = \sum_{x \in N} c_M(x)$ 
    and $N \subseteq c_M(N)M$. In other words, $c_M(N)$ is the smallest ideal $I$ of $R$ such that $N \subseteq IM$.
\end{lemma}

\begin{proof}
    Let $I \coloneqq \sum_{x \in N} c_M(x)$. Since $M$ is Ohm-Rush, for all $x \in N$, $x \in c_M(x)M \subseteq IM$. Thus, $N \subseteq IM$, and so, $c_M(N) \subseteq I = \sum_{x \in N} c_M(x)$. For the other inclusion, let $J$ be an ideal of $R$ such that $N \subseteq JM$. Then for all $x \in N$, $x \in JM$, which implies $c_M(x) \subseteq J$. Thus $\sum_{x \in N} c_M(x) \subseteq J$, for all ideals $J$ such that $N \subseteq JM$, that is, $\sum_{x \in N} c_M(x) \subseteq c_M(N)$.

    The fact that $N \subseteq c_M(N)M$ follows because for all $x \in N$, $x \in c_M(x)M \subseteq c_M(N)M$. Then by definition of $c_M(N)$, the latter is the smallest ideal $I$ of $R$ such that $N \subseteq IM$.
\end{proof}

As a consequence of the previous lemma, one can deduce the following:

\begin{corollary}
    \label{cor:content-union-sum}
    Let $M$ be an Ohm-Rush $R$-module. Let $\{N_\alpha \colon \alpha \in A\}$ be a family of subsets of $M$. Then
    \begin{enumerate}
        \item[$(1)$] $c_M(\bigcup_{\alpha \in A} N_\alpha)  = \sum_{\alpha \in A} c_M(N_\alpha)$.
        
        \item[$(2)$] If each $N_\alpha$ is a submodule of $M$, then $c_M(\sum_{\alpha \in A}N_\alpha) = \sum_{\alpha \in A} c_M(N_\alpha)$.
    \end{enumerate}
\end{corollary}

\begin{proof}
    (2) follows from (1) 
    because $\sum_{\alpha \in A} N_\alpha$ is the submodule of $M$ generated by $\bigcup_{\alpha \in A} N_\alpha$. For (1), by a double application of \autoref{lem:content-subset}, we get
    \[
        c_M(\bigcup_{\alpha \in A} N_\alpha) = \sum_{x \in \bigcup_{\alpha \in A}N_\alpha} c_M(x) = \sum_{\alpha \in A} \sum_{x \in N_\alpha} c_M(x) = \sum_{\alpha \in A} c_M(N_\alpha).\qedhere
    \]
\end{proof}

\begin{remark}
    \label{rem:content-additivity}
    Even though the content function of Ohm-Rush modules is additive for submodules as \autoref{cor:content-union-sum} (2) shows, we warn the reader that the function is not additive for elements. Namely, while it is always true for an Ohm-Rush module $M$ that 
    $
    c_M(x_1+ \dots + x_n) \subseteq \sum_{i=1}^n c_M(x_i),    
    $
    equality does not hold in general. An easy example can be obtained by take $M = R$. Note that for all $x \in R$, $c_R(x) = xR$, but in general the ideal $(x+y)R$ does not equal $xR + yR$ (for instance if $R = k[x,y]$ and we choose the indeterminates as our two elements).
\end{remark}

\begin{example}
    \label{eg:free-OR}
    A free $R$-module $F$ is Ohm-Rush. Indeed, if $\{e_i \colon I \in I\}$ is a basis of $F$, then for all $x \in F$, if we write
    $
    x = r_1e_{i_1} + \dots + r_ne_{i_n},    
    $
    where $r_1,\dots,r_n \in R \setminus \{0\}$, then using linear independence it follows that
    $
    c_F(x) = (r_1,\dots,r_n).    
    $
\end{example}

A straightforward reinterpretation of the Ohm-Rush condition is the following:

\begin{lemma}\cite[(1.2)]{OhmRu-content}
    \label{lem:OR-weak-intersection-flat}
    Let $M$ be a module over a ring $R$. Then $M$ is Ohm-Rush if and  only if for any collection of ideals $\{I_\alpha\}_\alpha$, we have $\left(\bigcap_\alpha I_\alpha \right)M = \bigcap_\alpha I_\alpha M$.
\end{lemma}

\begin{corollary}
    \label{cor:OR-composition}
    Let $R \to S$ be an Ohm-Rush ring map and $M$ an Ohm-Rush $S$-module. Then $M$ is an Ohm-Rush $R$-module.
\end{corollary}

\begin{proof}
    Let $\{I_\alpha \colon \alpha \in A\}$ be a collection of ideals of $R$. Then 
    \[
    \bigcap_{\alpha \in A} I_\alpha M = \bigcap_{\alpha \in A} (I_\alpha S)M  = \left(\bigcap_{\alpha \in A} I_\alpha S\right)M = \left(\left(\bigcap_{\alpha \in A} I_\alpha \right)S\right)M =  \left(\bigcap_{\alpha \in A} I_\alpha \right)M.
    \]
    The first and fourth equalities follow because $M$ is an $S$-module, the second equality follows because $M$ is an Ohm-Rush $S$-module and the third equality follows because $S$ is an Ohm-Rush $R$-module. Thus, $M$ is an Ohm-Rush $R$-module by \autoref{lem:OR-weak-intersection-flat}.
\end{proof}

The next result highlights how the content function behaves with respect to scaling.

\begin{lemma}\label{lem:factorcontent}
    Let $R$ be a ring and let $M$ be an $R$-module.  Let  $r\in R$ and $x\in M$.  Then $c_M(rx) \subseteq rc_M(x)$ in the following situations:
    \begin{enumerate}
        \item[$(1)$] $r=0$ or $r$ is a nonzerodivisor on $R$ (e.g. if $R$ is a domain),
        \item[$(2)$] $x \in c_M(x)M$ (e.g. if $M$ is Ohm-Rush, or if $c_M(x)=R$).
    \end{enumerate}
    If $M$ is flat, then $rc_M(x) \subseteq c_M(rx)$.
    \end{lemma}
    
    \begin{proof}
    Suppose $r=0$.  Then $rx=0$, so $c_M(rx) = 0 = rc_M(x)$.
    
    Suppose $r$ is nonzerodivisor on $R$.  Let $I$ be an ideal with $x\in IM$.  Then $rx \in rIM$, so $c_M(rx) \subseteq rI$.  Hence, $c_M(rx)$ is contained in the intersection of all such ideals $rI$, which, since $r$ is a nonzerodivisor, coincides with $r$ times the intersection of the ideals $I$.  That is, $c(rx) \subseteq rc(x)$.
    
    If $x \in c_M(x)M$, then $rx \in rc_M(x)M$, whence $c_M(rx) \subseteq rc_M(x)$.
    
    On the other hand, suppose $M$ is flat.  Let $I$ be an ideal with $rx \in IM$.  Then $x \in (IM :_M r) = (I:_R r)M$ (since $M$ is flat), so that 
    $c_M(x) \subseteq (I : r),$ 
    whence $rc_M(x) \subseteq I$.  Thus, $rc_M(x)$ is contained in the intersection of all such ideals $I$, which is to say $rc_M(x) \subseteq c_M(rx)$. 
    \end{proof}

The first connection between cyclic domination and the content function is given by the folllowing Lemma.

\begin{lemma}
\label{lem:cyclic-domination-content}
Let $R$ be a ring and $f \colon M \to P$ and $g \colon M \to Q$ be $R$-linear maps. Suppose $g$ cyclically dominates $f$. Then for all $x \in M$ and for all ideals $I$ of $R$
\begin{enumerate}
    \item[$(1)$] $f(x) \in IP \implies g(x) \in IQ$.
    \item[$(2)$] $c_Q(g(x)) \subseteq c_P(f(x))$.
\end{enumerate}
\end{lemma}

\begin{proof}
For all $x \in M$ and for all ideals $I$ of $R$
\[
f(x) \in IP \Longleftrightarrow x + I \in \ker(f \otimes_R \id_{R/I}) \subseteq \ker(g \otimes_R \id_{R/I}) \implies g(x) \in IQ.
\]
This immediately implies that $c_Q(g(x)) \subseteq c_P(f(x))$.
\end{proof}

\begin{lemma}
\label{lem:content-cyclic-purity}
If $\varphi \colon N \to M$ is a cyclically pure map of $R$-modules, then $c_M \circ \varphi = c_N$.
\end{lemma}

\begin{proof}
For all $x \in N$ and ideals $I$ of $R$, the injectivity of the induced map $\varphi \otimes \id_{R/I} \colon N/IN \to M/IM$ shows that $x \in IN$ if and only if $\varphi(x)$ in $IM$. Thus, $c_N(x) = c_M(\varphi(x)) = c_M \circ \varphi(x)$, that is, $c_N = c_M \circ \varphi$. 
\end{proof}

In the next Lemma, we describe a situation where domination comes to us for free from cyclic domination.

\begin{lemma}
\label{lem:cyclic-honest-domination}
Let $R$ be a ring and $v \colon R \to M$ and $u \colon R \to F$ be $R$-linear maps, where $F$ is a free module. Then we have the following:
\begin{enumerate}
    \item[$(1)$] There exists $\varphi \colon F \to M$ such that $v = \varphi \circ u$ if and only if $v(1) \in c_F(u(1))M$.
    \item[$(2)$] If $v$ cyclically dominates $u$, then $v$ factors through $u$. Consequently, $v$ dominates $u$.
\end{enumerate}
\end{lemma}

\begin{proof}
(2) follows from (1). Indeed, if $v$ cyclically dominates $u$, then since $u(1) \in c_F(u(1))F$ (free modules are Ohm-Rush), by \autoref{lem:cyclic-domination-content} (1)
we have
$
v(1) \in c_F(u(1))M. 
$
Then $v$ factors through $u$ by (1).

We first prove the backward implication of (1). Let $\{e_i \colon i \in I\}$ be a free basis of $F$ and let 
$u(1) = r_1e_{i_1} + \dots + r_ne_{i_n},$ 
where $r_1,\dots,r_n \in R \setminus\{0\}$. By \autoref{eg:free-OR}, 
$c_F(u(1)) = (r_1,\dots,r_n).$ 
Since $v(1) \in c_F(u(1))M$, there exist $m_1,\dots,m_n \in M$ 
such that
$
v(1) = r_1m_1 + \dots + r_nm_n.
$
Now consider any $R$-linear map
\begin{align*}
\varphi \colon F &\to M\\
e_{i_j} &\mapsto m_j,
\end{align*}
for $j = 1, \dots, n$.
By construction,
$
\varphi \circ u (1) = \varphi\left(\sum_{j=1}^nr_je_{i_j}\right) = \sum_{j=1}^n r_jm_j = v(1).
$
Thus, $
\varphi \circ u = v.
$
Conversely, suppose there exists $\varphi \colon F \to M$ such that $v = \varphi \circ u$. Since $F$ is Ohm-Rush, $u(1) \in c_F(u(1))F$. Thus,
$
v(1) = \varphi(u(1)) \in c_F(u(1))M    
$
by linearity of $\varphi$. This completes the proof of (1).
\end{proof}

\begin{proposition}
    \label{prop:content-homomorphisms-purity-flatness}
    Let $M$ be an $R$-module and $\varphi \colon N \to M$ be an $R$-linear map. Then for all $x \in N$, $c_M(\varphi(x)) \subseteq c_N(x)$. Furthermore, if $M$ is Ohm-Rush, then we have the following:
    \begin{enumerate}
        \item[$(1)$] $\varphi$ is injective if and only if for all $x \in N$, $c_M(\varphi(x)) \neq (0)$. 
        
        \item[$(2)$] $\varphi$ is cyclically pure if and only if $N$ is Ohm-Rush and $c_M \circ \varphi = c_N$. 
     
        \item[$(3)$] $M$ is flat if and only if for all $r \in R$ and $x \in M$, $c_M(rx) = rc_M(x)$.
        
        \item[$(4)$] If $M$ is flat and $S$ is a multiplicative set, then for all $x \in M$ and $s \in S$, $c_M(x)(S^{-1}R) = c_{S^{-1}M}(x/s)$. Thus, for all submodules $N$ of $M$, $c_M(N)(S^{-1}R) = c_{S^{-1}M}(S^{-1}N)$.
    \end{enumerate} 
\end{proposition}

\begin{proof}
    Let $I$ be an ideal of $R$ such that $x \in IN$. Then by linearity, $\varphi(x) \in IM$. Thus, $c_M(\varphi(x)) \subset I$ for all such ideals $I$, and so, $c_M(\varphi(x)) \subseteq c_N(x)$.

    We now assume $M$ is Ohm-Rush, that is, for all $y \in M$, $y \in c_M(y)M$.
    Then it is clear that $y \neq 0$ if and only if $c_M(y) \neq (0)$. Part (1) readily follows by this observation.
    Part (2) essentially follows by the argument in \cite[Thm.\ 1.3]{OhmRu-content} and we omit the details. Note that $\varphi$ is cyclically pure $\implies c_M \circ \varphi = c_N$ by \autoref{lem:content-cyclic-purity}. Part (3) is part of \cite[Cor.\ 1.6]{OhmRu-content}. Note that the forward implication of $(3)$ is shown in \autoref{lem:factorcontent}.

    The first part of (4) follows by \cite[Thm.\ 3.1]{OhmRu-content} and (3). If $N$ is a submodule of $M$, then 
    \begin{align*}
    c_M(N)(S^{-1}R) &= \left(\sum_{x \in N}c_M(x)\right)S^{-1}R
    = \sum_{x \in N} c_M(x)(S^{-1}R)\\ 
    &= \sum_{s \in S}\sum_{x \in N}c_{S^{-1}M}(x/s) 
    = c_{S^{-1}M}(S^{-1}N).
    \end{align*}
    The first equality follows by \autoref{lem:content-subset} because $M$ is Ohm-Rush, the second equality follows because expansions of ideals commute with arbitrary sums of ideals, the third equality follows by the first assertion in this statement, and the fourth equality follows because $S^{-1}M$ is an Ohm-Rush $S^{-1}R$-module by the equality $c_M(x)(S^{-1}R) = c_{S^{-1}M}(x/s)$.
\end{proof}

\begin{remark}
    \autoref{prop:content-homomorphisms-purity-flatness} (2) shows that the Ohm-Rush property descends under cyclically pure maps.
\end{remark}

We now examine the behavior of the content function under base change. The main result is:

\begin{proposition}
    \label{prop:ascent-relativelypure-content}
    Let $\varphi \colon R \to S$ be a ring map. Let $f \colon N \to M$ be an $R$-linear map such that $N$ is a flat $R$-module and $M$ is an $S$-module. Assume that the induced $S$-linear map $\tilde{f} \colon N \otimes_R S \to M$ is $S$-cyclically pure. Let $x \in N$ and $y \coloneqq f(x)$. If $x \in c_N(x)N$, then $y \in c_M(y)M$ and $c_M(y) = c_N(x)S$. 
\end{proposition}

\begin{proof}
    By \autoref{lem:content-cyclic-purity}, we have $c_{M} \circ \tilde{f} = c_{N \otimes_R S}$ since $\tilde{f}$ is $S$-cyclically pure. Thus, $c_M(y) = c_M(\tilde{f}(x \otimes 1)) = c_{N \otimes_R S}(x \otimes 1)$, and  we reduce to the case where $M = N \otimes_R S$, $f \colon N \to N \otimes_R S$ is the canonical map that sends $z \mapsto z \otimes 1$ and $y = x \otimes 1$. 

    Since $x \in c_N(x)N$, by linearity of $f$ we get $x \otimes 1 = f(x) \in (c_N(x)S)(N\otimes_R S)$. Thus, $c_{N \otimes_R S}(x \otimes 1) \subseteq c_N(x)S$. To finish the proof it remains to show $c_N(x)S \subseteq c_{N \otimes_R S}(x \otimes 1)$. Let $J$ be an ideal of $S$ such that $x \otimes 1 \in J(N \otimes_R S)$. Let $J^{\mathfrak c}$ denote the contraction of $J$ to $R$. Since the induced map $R/J^{\mathfrak c} \to S/J$ is injective, the flatness of $N$ implies that $N/J^{\mathfrak c}N \to (S/J) \otimes_R N$ is injective as well. This induced map sends $x + J^{\mathfrak c}N \mapsto 0$ since $x \otimes 1 \in J(S \otimes_R N)$. Thus, we have shown that if $J$ is an ideal of $R$ such that $x \otimes 1 \in J(N \otimes_R S)$, then $x \in J^{\mathfrak c}N$. In other words, $c_N(x) \subseteq J^{\mathfrak c}$ for all such ideals $J$. Since taking ideal contraction commutes with arbitrary intersections of ideals, and since the intersection of all the ideals $J$ of $S$ such that $x \otimes 1 \in J(N \otimes_R S)$ is precisely $c_{N \otimes_R S}(x \otimes 1)$, we get $c_N(x) \subseteq (c_{N \otimes_R S}(x \otimes 1))^{\mathfrak c}$, and so, $c_{N}(x)S \subseteq c_{N \otimes_R S}(x \otimes 1)$. 
\end{proof}

\begin{theorem}
\label{thm:content-equivalences}
Let $R$ be a ring and $M$ be an $R$-module. Let $v \colon R \to M$
be a linear map. Consider the following statements:
\begin{enumerate}
    \item[$(1)$] $v(1) \in c_M(v(1))M$.
    \item[$(2)$] $v$ admits a cyclic stabilizer $u \colon R \to F$ where
    $F$ is free of finite rank and such that $v$ dominates $u$.
    \item[$(3)$] $v$ admits a cyclic stabilizer that $v$ dominates.
    \item[$(4)$] $v$ admits a cyclic stabilizer $u \colon R \to P$ where
    $P$ is Ohm-Rush.
    \item[$(5)$] $v$ admits a cyclic stabilizer $u \colon R \to P$ such 
    that $u(1) \in c_P(u(1))P$.
\end{enumerate}
One always has $(2) \implies (3)$ and $(4) \implies (5)\implies (1)$. Moreover, if $M$ is a flat $R$-module, then 
$(1) \implies (2)$ and $(3) \implies (4)$. Hence all five
statements are equivalent for flat modules.
\end{theorem}

\begin{proof}
The implications $(2) \implies (3)$ and $(4) \implies (5)$ are
clear. For $(5) \implies (1)$, suppose $u \colon R \to P$ is a cyclic stabilizer of $v$ such that
$
u(1) \in c_P(u(1))P.
$
Since $u$ and $v$  cyclically dominate each other, by \autoref{lem:cyclic-domination-content} (2) we have 
$
c_P(u(1)) = c_M(v(1)).
$
Then $v(1) \in c_P(u(1))M = c_M(v(1))M$ by \autoref{lem:cyclic-domination-content} (1).

Now assume that $M$ is a flat $R$-module.

$(1) \implies (2):$ Since $v(1) \in c_M(v(1))M$, 
$c_M(v(1))$ is a finitely generated ideal of $R$ (\autoref{rem:content-finitely-generated}). Suppose 
$
c_M(v(1)) = (r_1,\dots,r_n).
$
Now consider the map $u \colon R \to R^{\oplus n}$ such that $u(1) = (r_1,\dots,r_n)$. We claim that $u$ is a cyclic stabilizer of $v$. Then
the fact that $v$ dominates $u$ will follow from \autoref{lem:cyclic-honest-domination} (2). For an ideal $I$ of $R$, we have 
\begin{equation}
    \label{eq:content-colon}
\ker(u \otimes_R \id_{R/I}) = (I \colon (r_1,\dots,r_n))/I = \left(I \colon c_M(v(1))\right)/I.
\end{equation}
Now suppose $r \in R$ such that $r + I \in \ker(v \otimes_R \id_{R/I})$.
This is equivalent to saying that $rv(1) \in IM$, and so, $c_M(rv(1))
\subseteq I$. Since $M$ is flat, \autoref{lem:factorcontent}
shows that $rc_M(v(1)) \subseteq c_M(rv(1)) \subseteq I$, and so, 
$r \in (I \colon c_M(v(1)))$. Thus,
$
\ker(v \otimes_R \id_{R/I}) \subseteq (I \colon c_M(v(1)))/I.
$

Conversely, suppose $r \in (I \colon c_M(v(1)))$. Then $rc_M(v(1)) \subseteq I$, and so, $rc_M(v(1))M \subseteq IM$. Since $v(1) \in c_M(v(1))M$, we then get
$
v(r) = rv(1) \in rc_M(v(1))M \subseteq IM.
$
Hence $r + I \in \ker(v \otimes_R \id_{R/I})$, which establishes the other
inclusion
$
(I \colon c_M(v(1)))/I \subseteq \ker(v \otimes_R \id_{R/I}).
$
Then for all ideals $I$ of $R$
\[
\ker(v \otimes_R \id_{R/I}) = (I \colon c_M(v(1)))/I \stackrel{\autoref{eq:content-colon}}{=\joinrel =} \ker(u \otimes_R \id_{R/I}).
\]
Hence $v$ and $u$ cyclically dominate each other.

Finally, if $M$ is flat then the implication $(3) \implies (4)$ follows by \autoref{cor:cyclic-stabilizer-flat} because free modules are Ohm-Rush by \autoref{eg:free-OR}.
\end{proof}

We now obtain the following characterizations of the Ohm-Rush property. The proof is omitted as the result follows readily from the definition of an Ohm-Rush module and \autoref{thm:content-equivalences}.

\begin{theorem}
\label{thm:OR-equivalences}
Let $R$ be a ring and $M$ be an $R$-module. Consider the following
statements:
\begin{enumerate}
    \item[$(1)$] $M$ is an Ohm-Rush $R$-module.
    \item[$(2)$] Every $R$-linear map $v\colon R \to M$ admits a cyclic stabilizer $u \colon R \to F$ where
    $F$ is free of finite rank and such that $v$ dominates $u$.
    \item[$(3)$] Every $R$-linear map $v \colon R \to M$ admits a cyclic stabilizer that $v$ dominates.
    \item[$(4)$] Every $R$-linear map $v \colon R \to M$ admits a cyclic stabilizer $u \colon R \to P$ where
    $P$ is Ohm-Rush.
    \item[$(5)$] Every $R$-linear map $v\colon R \to M$ admits a cyclic stabilizer $u \colon R \to P$ such 
    that $u(1) \in c_P(u(1))P$.
\end{enumerate}
One always has $(2) \implies (3)$, $(4) \implies (5)$ and $(5) \implies (1)$. Moreover, if $M$ is flat, then all five assertions are equivalent.
\end{theorem}

We isolate some important consequences of \autoref{thm:OR-equivalences}.

\begin{corollary}
Let $R$ be a ring and $M$ an $R$-module. Suppose that every $R$-linear map $v: R \to M$ admits a cyclic stabilizer $u: R \to P$, such that $P$ is Ohm-Rush.  Then $M$ is Ohm-Rush.
\end{corollary}
    
\begin{proof}
This follows from the implications (4) $\implies$ (5) $\implies$ (1) of \autoref{thm:content-equivalences}.
\end{proof}

We also obtain a connection between the ML and Ohm-Rush properties for flat modules.

\begin{corollary}
\label{cor:ML-implies-OR}
Let $R$ be a ring and $M$ be a flat $R$-module. If $M$ is ML then $M$ is Ohm-Rush. 
\end{corollary}

\begin{proof}
Let $v \colon R \to M$ be a linear map. Since $M$ is ML, $v$ admits a stabilizer $u \colon R \to P$. That is, $P$ is finitely presented and $v$ and $u$  dominate each other, and so, also  cyclically dominate each other. Thus $v$ admits a cyclic stabilizer that $v$ dominates. This shows $M$ is Ohm-Rush by \autoref{thm:OR-equivalences} (3)$\implies$(1).
\end{proof}

One may naturally wonder if there is any relationship between ML modules and Ohm-Rush modules outside the flat setting. The next result shows that no general relationship is possible because the Ohm-Rush condition often implies flatness.

\begin{proposition}[{\cite[p. 53, last paragraph of Section 1]{OhmRu-content}}]
\label{prop:OR-flat-torsionfree}
Let $R$ be a domain. Suppose $M$ is an Ohm-Rush $R$-module. Then $M$ is flat if and only if $M$ is torsion-free.
\end{proposition}

\begin{example}
\label{eg:ML-not-OR}
We can now construct many examples of non-flat ML-modules that are \emph{not} Ohm-Rush. Namely, if $R$ is a domain, then any finitely presented torsion-free $R$-module will be ML because finitely presented modules are always ML. However, such a module will be Ohm-Rush precisely when it is projective by \autoref{prop:OR-flat-torsionfree}.
\end{example}

\begin{remark}
It is more difficult to give examples of Ohm-Rush modules over a ring that are not ML. For instance, no such examples exist within the class of finitely presented modules (see also the remarks in \autoref{sec:futureworkandacknowledgements}). 
\end{remark}

\begin{corollary}
\label{cor:OR-ring-extensions-flat}
Suppose $\varphi \colon R \to S$ is an extension of domains. If $\varphi$ is Ohm-Rush, then $\varphi$ is flat.
\end{corollary}

\begin{proof}
In our setup, $S$ is a torsion-free $R$-module. Thus, $S$ is flat by \autoref{prop:OR-flat-torsionfree}.
\end{proof}

We next discuss how the content function behaves under filtered systems. One can view the stabilization property proved below as an analog of the stabilization properties proved for ML and SML modules in \autoref{thm:ML-systems}.

\begin{proposition}
\label{prop:colimit-OR}
        Let $R$ be a ring and let $\{(L_i, u_{ij}) \colon i, j \in I, i \leq j\}$ be a
        system of finitely presented $R$-modules indexed by a filtered poset $(I ,\leq)$. Let 
        $
        M \coloneqq \colim_i L_i.
        $
        For all $i \in I$, let $u_i \colon L_i \to M$ be the associated map. 
        Fix
        $m \in M$ and choose an index $i_0 \in I$ along with a lift
        $x_{i_0}$ of $m$ to $L_{i_0}$, that is, 
        $
        u_{i_0}(x_{i_0}) = m.
        $
        For all $j \geq i_0$, let $x_j \coloneqq u_{i_0 j}(x_{i_0})$.
        Then we have the following:
        \begin{enumerate}
            \item[$(1)$] For all $k \geq j \geq i_0$,
            $
            c_{L_k}(x_k) \subseteq c_{L_j}(x_j) \subseteq c_{L_{i_0}}(x_{i_0}).
            $
            
            \item[$(2)$] $c_M(m) = \bigcap_{j \geq i_0} c_{L_j}(x_j)$.
            
            \item[$(3)$] Suppose that $L_i$ is an Ohm-Rush $R$-module for all $i \in I$ (for instance, if each $L_i$ is free). Then $m \in c_M(m)M$ if and only if the inverse system of ideals $\{c_{L_j}(x_j) \colon j \geq i_0\}$ stabilizes under inclusion, that is, there exists $j_0 \geq i_0$ such that for all $k \geq j_0$, $c_{L_{j_0}}(x_{j_0}) = c_{L_k}(x_k)$.

            \item[$(4)$] Suppose that $L_i$ is a free $R$-module for all $i$. Let $v \colon R \to M$ be the unique $R$-linear map that sends $1 \mapsto m$. For all $j \geq i_0$, let $v_j \colon R \to L_j$ be the unique $R$-linear map that sends $1 \mapsto x_j$. Consider the following statements: 
            \begin{enumerate}
            \item[$(a)$] $m \in c_M(m)M$. 
            
            \item[$(b)$] There exists $j_0 \geq i_0$ such that for all  $k \geq j_0$, $v_k$ stabilizes $v$. 

            \item[$(c)$] There exists $j_0 \geq i_0$ such that for all $R$-modules $N$ and for all $k \geq j_0$,
            $
            \im(\Hom_R(v_{j_0}, N)) = \im(\Hom_R(v_k,N)).    
            $
            \end{enumerate}
            Then $(a)\Longleftrightarrow (b)\Longrightarrow(c)$.
        \end{enumerate}
        \end{proposition}
        
        \begin{proof}
        (1) By definition, 
        $
        u_{jk}(x_j) = u_{jk}(u_{i_0j}(x_{i_0})) = u_{i_0k}(x_{i_0}) = x_k.
        $
        Thus, (1) follows by applying \autoref{prop:content-homomorphisms-purity-flatness}
        to the map $u_{i_0j}$ (resp. $u_{jk}$) and the element $x_{i_0} \in L_{i_0}$ (resp. $x_j \in L_j$).

        (2) Since $u_{i_0}$ maps $x_{i_0} \mapsto m$, we see that for
        all $j \geq i_0$,
        $
        u_j(x_j) = u_j(u_{i_0j}(x_{i_0})) = u_{i_0}(x_{i_0}) = m.
        $
        Thus, by \autoref{prop:content-homomorphisms-purity-flatness} applied to the maps
        \begin{align*}
        u_j \colon L_j &\to M\\
        x_j &\mapsto m,
        \end{align*}
        we get
        \begin{equation}
            \label{eq:content-containment-1}
        c_M(m) \subseteq \bigcap_{j \geq i_0} c_{L_j}(x_j).
        \end{equation}

        Now suppose that $I$ is an ideal of $R$ such that $m \in IM$.
        Choose $a_1,\dots,a_n \in I$ and $m_1,\dots,m_n \in M$ such that
        $
        m = a_1m_1 + \dots + a_nm_n.
        $
        Consider the maps
        \begin{align*}
        f \colon R &\to M\\
        1 &\mapsto m,
        \end{align*}
        and
        \begin{align*}
        g \colon R &\longrightarrow R^{\oplus n}\\
        1 &\mapsto (a_1,\dots,a_n).
        \end{align*}
        Then $f$ factors through $g$ via the map
        $
        h \colon R^{\oplus n} \to M
        $
        that sends the standard basis element $e_\ell \mapsto m_\ell$, for
        $1 \leq \ell \leq n$. One can
        choose an index $i \in I$ such that $h$ lifts to a map 
        $
        h_i \colon R^{\oplus n} \to L_i
        $
        along $u_i \colon L_i \to M$, that is,
        $
        u_i \circ h_i = h.
        $
        Then
        $
        u_i(h_i(a_1,\dots,a_n)) = h(a_1,\dots,a_n) = h(g(1)) = f(1) = m.
       $
        Since $x_{i_0}$ and $h_i(a_1,\dots,a_n)$ are both lifts of $m$, by construction of a filtered colimit, there exists
        $j \in I$ such that $j \geq i_0, i$ and 
        $
        u_{ij}(h_i(a_1,\dots,a_n)) = u_{i_0j}(x_{i_0}) = x_j.
        $
        Then by \autoref{prop:content-homomorphisms-purity-flatness} applied to the map $u_{ij} \circ h_i \colon R^{\oplus n} \to L_j$, we have
        $
        c_{L_j}(x_j) \subseteq c_{R^{\oplus n}}(a_1,\dots,a_n) = (a_1,\dots,a_n) \subseteq I.
        $
        Thus, we have shown that for any ideal $I$ of $R$ such that $m \in IM$, there
        exists $j \geq i_0$ and $c_{L_j}(x_j) \subseteq I$.
        Consequently,
        $
        \bigcap_{j \geq i_0} c_{L_j}(x_j) \subseteq I
       $
        for all ideals $I$ of $R$ such that $m \in IM$, and so, $\bigcap_{j \geq i_0} c_{L_j}(x_j) \subseteq c_M(m)$. This, along with \autoref{eq:content-containment-1}, completes the proof of (2).
        
        (3) Suppose $m \in c_M(m)M$. In order to show that $\{c_{L_j}(x_j) \colon j \geq i_0\}$ stabilizes under inclusion, by (1) and (2) it
        suffices to show that there exists $j_0 \geq i_0$ such that
        $
        c_{L_{j_0}}(x_{j_0}) \subseteq c_M(m).
        $
        Let $c_M(m) = (a_1,\dots,a_n)$ (\autoref{rem:content-finitely-generated}). Then by repeating the argument in (2) 
        above, we see that there exists $j_0 \geq i_0$ such that
        $
        c_{L_{j_0}}(x_{j_0}) \subseteq (a_1,\dots,a_n) = c_M(m).
        $

        Conversely, suppose $\{c_{L_j}(x_j) \colon j \geq i_0\}$ stabilizes under inclusion, say with stable ideal $c_{L_{j_0}}(x_{j_0})$. Since
        $I$ is a directed set, by (1) and (2) we have
        $
        c_M(m) = \bigcap_{j\geq i_0} c_{L_j}(x_j) = \bigcap_{k \geq j_0} c_{L_k}(x_k) = c_{L_{j_0}}(x_{j_0}).
        $
        Since $L_{j_0}$ is an Ohm-Rush $R$-module, it follows that
        $x_{j_0} \in c_{L_{j_0}}(x_{j_0})L_{j_0}$. Using the 
        homomorphism 
        $
        u_{j_0} \colon L_{j_0} \to M,
        $
        we get $m = u_{j_0}(x_{j_0}) \in c_{L_{j_0}}(x_{j_0})M = 
     c_M(m)M$, as desired.

     (4) We first show $(b) \implies (a)$. If there exists $k \geq i_0$ such that $v_k$ stabilizes $v$, then since $L_k$ is Ohm-rush, by \autoref{thm:content-equivalences}$(4) \implies (1)$ we have that $m = v(1) \in c_M(v(1))M = c_M(m)M$.

     We next prove $(a) \implies (b)$. So assume that $m \in c_M(m)M$. By (3) and (2), there exists $j_0 \geq i_0$ such that for all $k \geq j_0$, 
     $
        c_{L_k}(x_k) = c_{L_{j_0}}(x_{j_0}) = c_M(m).   
     $
      By assumption, $L_{j_0}$ and $L_k$ are free $R$-modules, and hence Ohm-Rush. Thus, we have
     $
     v_{j_0}(1) = x_{j_0} \in c_{L_{j_0}}(x_{j_0})L_{j_0} = c_{L_k}(x_k)L_{j_0} = c_{L_k}(v_k(1))L_{j_0},   
     $
     and
     $
        v_{k}(1) = x_k \in c_{L_k}(x_k)L_k = c_{L_{j_0}}(x_{j_0})L_{k} = c_{L_{j_0}}(v_{j_0}(1))L_k.   
     $
     By \autoref{lem:cyclic-honest-domination} (1) we then conclude that $v_{j_0} \colon R \to L_{j_0}$ factors through $v_k \colon R \to L_k$ and vice-versa (here we are using that $L_{j_0}$ and $L_k$ are both free and not just Ohm-Rush). Thus, for all $k \geq j_0$, $v_{j_0}$ and $v_k$  dominate each other. Since 
     $
     v = \colim_{k \geq j_0} v_j,   
     $
     it follows that $v$ and $v_{j_0}$  dominate each other because  for all $R$-modules $N$,
     \begin{align*}
     \ker(v \otimes_R \id_N) &= \ker((\colim_{k \geq j_0} v_k)\otimes_R \id_N)
     = \ker(\colim_{k \geq j_0}(v_k \otimes_R \id_N))\\ 
     &= \colim_{k \geq j_0} \ker(v_k \otimes_R \id_N) 
     = \colim_{k \geq j_0} \ker(v_{j_0} \otimes_R \id_N)
     = \ker(v_{j_0} \otimes_R \id_N).  
     \end{align*}
     Here the second equality follows because tensor products commute with filtered colimits, the third equality follows because filtered colimits are exact in $\Mod_R$ \cite[\href{https://stacks.math.columbia.edu/tag/00DB}{Tag 00DB}]{stacks-project} and the fourth equality follows because $v_{j_0}$ and $v_k$  dominate each other. In fact, since $\ker(v_{j_0} \otimes_R \id_N) = \ker(v_k \otimes_R \id_N)$, the above chain of equalities also shows that $\ker(v \otimes_R \id_N) = \ker(v_k \otimes_R \id_N)$ for all $k \geq j_0$. Since $L_k$ is a finitely presented $R$-module by hypothesis, we get that $v_k$ stabilizes $v$ for all $k \geq j_0$. This completes the proof of $(a) \implies (b)$.

     It remains to show that (b) $\implies$ (c). Suppose there exists $j_0 \geq i_0$ such that for all $k \geq j_0$, $v_k$ stabilizes $v$. We have
     $
     v_k(1) = x_k = u_{i_0k}(x_0) = u_{j_0k}(u_{i_0j_0}(x_0)) = u_{j_0k}(x_{j_0}) = u_{j_0k}(v_{j_0}(1)),
     $
     and so, $v_k = u_{j_0k} \circ v_{j_0}$.
     Furthermore, 
     $
     u_k \circ v_k(1) = u_k(x_k) = u_k(u_{i_0k}(x_{i_0})) = u_{i_0}(x_{i_0}) = m = v(1).
     $
     Thus, for all $k \geq j_0$, $u_k \circ v_k = v$. Since $v_k$ stabilizes $v$, one can now apply \autoref{lem:stabilizers} (3) (with $f_k = v_k$, $g = v$ and $i = j_0$ in the notation of the Lemma) to conclude that for all $k \geq j_0$ and for all $R$-modules $N$, $\im(\Hom_R(v_{j_0},N)) = \im(\Hom_R(v_k,N))$.
\end{proof}

\autoref{prop:colimit-OR} allows us to add an additional statement in the list of equivalent statements for flat modules in \autoref{thm:content-equivalences}.

\begin{corollary}
\label{cor:content-equivalences-addition}
Let $R$ be a ring and $M$ be a flat $R$-module. Let $v \colon R \to M$ be an $R$-linear map. Then the following are equivalent:
\begin{enumerate}
\item[$(1)$] $v(1) \in c_M(v(1))M$.
\item[$(2)$] $v$ admits a stabilizer $u \colon R \to L$ where $L$ is free of finite rank.  
\item[$(3)$] $v$ admits a stabilizer.
\end{enumerate}
\end{corollary}

\begin{proof}
Note that any stabilizer of $v$ is a cyclic stabilizer of $v$ that $v$ dominates. Thus, $(3) \implies (1)$ follows by \autoref{thm:content-equivalences} (3)$\implies$(1) since $M$ is a flat $R$-module. The implication $(2) \implies (3)$ is clear. 
            
It suffices to show $(1) \implies (2)$. Since $M$ is flat we can express it as a colimit of a filtered system of free $R$-modules of finite rank $\{(L_i, u_{ij})\colon i,j \in I, i \leq j\}$. Then one can choose $i \in I$ and a lift $x_i$ of $v(1)$ to $L_i$ such that the unique map $u \colon R \to L_i$ that sends $1 \mapsto x_i$ stabilizes $v$ by \autoref{prop:colimit-OR} (4).
\end{proof}

We also have the following addition to the list of equivalent assertions for flat modules in Theorem \ref{thm:OR-equivalences}. The proof is clear and is omitted.

\begin{corollary}
\label{cor:OR-equivalences-addition}
Let $R$ be a ring and $M$ be a flat $R$-module. Then the following are equivalent:
\begin{enumerate}
\item[$(1)$] $M$ is Ohm-Rush.
\item[$(2)$] Every linear map $v \colon R \to M$ admits a stabilizer $u \colon R \to L$ where $L$ is free of finite rank.
\item[$(3)$] Every linear map $v \colon R \to M$ admits a stabilizer. 
\end{enumerate} 
\end{corollary}

The next result is as an analog of \autoref{prop:RG-2.1.9} for cyclic stabilizers.

\begin{proposition}\label{prop:analog-RG219}
Let $(R,\fm,\kappa)$ be a local ring (not necessarily Noetherian). Let $M$ be a flat $R$-module and $f\in M$.  Let $v: R \to M$ be the unique $R$-linear map that sends $1\mapsto f$. Then the following are equivalent:
\begin{enumerate}
\item[$(1)$] There is a finitely generated free submodule $L$ of $M$ that contains $f$, such that the inclusion map of $L$ into $M$ is pure.
\item[$(2)$] There is a finitely generated free submodule $L$ of $M$ that contains $f$, such that the inclusion map of $L$ into $M$ is cyclically pure.
\item[$(3)$] $v$ admits a cyclic stabilizer $u \colon R \to P$ such that $P$ is Ohm-Rush. 
\item[$(4)$] $v$ admits a cyclic stabilizer that $v$ dominates.
\item[$(5)$] $f \in c_M(f)M$.
\end{enumerate}
\end{proposition}
            
\begin{proof}
The equivalences of (3), (4) and (5) follow from \autoref{thm:content-equivalences} and hold even when $R$ is not local. The implication $(1) \implies (2)$ is clear while $(2) \implies (3)$ follows from \autoref{lem:factor-cyclically-pure} and the fact that free modules are Ohm-Rush (\autoref{eg:free-OR}).

It remains to show $(5) \implies (1)$. Since $M$ is flat and $f \in c_M(f)M$, $v$ admits a stabilizer by \autoref{cor:content-equivalences-addition}. Then (1) follows by \autoref{prop:RG-2.1.9} applied to $g = v$. 
\end{proof}

One can use \autoref{prop:analog-RG219} to now deduce that flat and non-trivial Ohm-Rush modules over local rings are faithfully flat.

\begin{corollary}
\label{cor:Nakayama-flat-OR}
    Let $(R, \fm, \kappa)$ be a non-trivial local ring (not necessarily Noetherian). Let $M$ be a flat and OR $R$-module. 
    We have the following:
    \begin{enumerate}
        \item[$(1)$] If $M = \fm M$, then $M = 0$. 
        \item[$(2)$] If $M \neq 0$ then $M$ is faithfully flat. 
        \item[$(3)$] If $M \neq 0$ and $I \neq J$ are ideals of $R$, then $IM \neq JM$.
        \item[$(4)$] If $M \neq 0$, then for any ideal $I$ of $R$, $c_M(IM) = I$.
    \end{enumerate}
    
\end{corollary}

\begin{proof}
    (1) Assume for contradiction that $M = \fm M$ and that $M \neq 0$. Let $f \in M$ be a nonzero element. By \autoref{prop:analog-RG219}, there exists a finitely generated free submodule $L$ of $M$ that contains $f$ and such that the inclusion $\iota \colon L \hookrightarrow M$ is pure. In particular, $L \neq 0$. Let $e \in L$ be a basis element. By \autoref{eg:free-OR}, $c_L(e) = R$. Since $c_M \circ \iota = c_L$ by \autoref{prop:content-homomorphisms-purity-flatness} (2), it follows that 
    $
     c_M(e) = c_M(\iota(e)) = c_L(e) = R.       
    $
    But this is a contradiction because $e \in M = \fm M$, and so, $c_M(e) \subseteq \fm$. 

    (2) follows from (1) by taking the contrapositive.

    (3) As we saw in (1), the hypothesis that $M \neq 0$ implies that there exists $e \in M$ such $c_M(e) = R$. Assume for contradiction that $IM = JM$. Let $i \in I$. Then $ie \in IM = JM$, and so, $iR = ic_M(e) = c_M(ie) \subseteq J$, where we use \autoref{prop:content-homomorphisms-purity-flatness} (3) for the equality $ic_M(e) = c_M(ie)$. Since $i \in I$ was arbitrary, this shows $I \subseteq J$. One similarly obtains $J \subseteq I$, contradicting that $I \neq J$.

    (4) By definition of the content of a subset of $M$, $c_M(IM) \subseteq I$. Moreover, since $M$ is Ohm-Rush, by \autoref{lem:content-subset}, $IM \subseteq c_M(IM)M$. Thus, $IM \subseteq c_M(IM)M \subseteq IM$, and so, by (3) we must have $I = c_M(IM)$.
\end{proof}

We also obtain the following comparison result between the content function of a flat Ohm-Rush module over a Noetherian local ring and the content function of the completion of the module.

\begin{corollary}
    \label{cor:content-completion}
    Let $(R,\fm,\kappa)$ be a Noetherian local ring. Let $M$ be a flat Ohm-Rush $R$-module. Let $\widehat{M}$ be the $\fm$-adic completion of $M$. 
    Let $c_M$ (resp. $c_{\widehat{M}}$) denote the content of $M$ (resp. of $\widehat{M}$) as an $R$-module (resp. $\widehat{R}$-module). Let 
$
    \varphi \colon M \to \widehat{M}
    $
    denote the canonical map. Then we have the following:
    \begin{enumerate}
        \item[$(1)$] For all $x \in M$, $\varphi(x) \in c_{\widehat{M}}(\varphi(x))\widehat{M}$ and $c_M(x)\widehat{R} = c_{\widehat{M}}(\varphi(x))$.
        \item[$(2)$] If $N \subseteq M$, then $\varphi(N) \subseteq c_{\widehat M}(\varphi(N))\widehat{M}$ and $c_M(N)\widehat{R} = c_{\widehat{M}}(\varphi(N))$. 
        \item[$(3)$] If $N$ is a submodule of $M$ and if $\widehat{R}N$ is the $\widehat{R}$-submodule of $\widehat{M}$ generated by $\varphi(N)$, then $c_M(N)\widehat{R} = c_{\widehat{M}}(\widehat{R}N)$. 
        \item[$(4)$] $\varphi$ is a pure map of $R$-modules. 
        \item[$(5)$] Let $\phi \colon P \to N$ be a map of $\widehat{R}$-modules such that $P$ is a flat Ohm-Rush $\widehat{R}$-module and $N$ is a flat $R$-module. Considering $\phi$ as a map of $R$-modules by restriction of scalars, we have that $\phi$ is $R$-pure if and only if $\phi$ is $\widehat{R}$-pure.
        \item[$(6)$] Let $S$ be a flat Ohm-Rush $R$-algebra. Let $I$ be an ideal of $S$ and let $\widehat{S}$ denote the $\fm$-adic completion of $S$ (i.e. $\widehat{S} = \widehat{S}^{\fm S}$). Then $c_S(I) \widehat{R} = c_{\widehat{S}}(I\widehat{S})$, where by $c_{\widehat{S}}$ we mean the content of $\widehat{S}$ as a $\widehat{R}$-module.
    \end{enumerate}
\end{corollary}

\begin{proof}
 By \autoref{lem:completion-purity}, $\widehat{M}$ is a flat $\widehat{R}$-module. Also note that by definition 
(see \autoref{subsection:contentfunctionandORmods}), the content of a subset of a module coincides with the content of the submodule generated by the subset. With these preliminary observations, we begin proving the assertions. 

    (1) By \autoref{prop:analog-RG219}, there exists a finitely generated free submodule $L$ of $M$ that contains $x$ such that the inclusion $\iota \colon L \hookrightarrow M$ is a pure map of $R$-modules. Then by \autoref{lem:completion-purity} and flatness of $M$, the induced map on $\fm$-adic completions 
    $
     \widehat{\iota} \colon \widehat{L} \to \widehat{M}   
   $
    is a pure map of $\widehat{R}$-modules. Note that $\widehat{L}$ is a free $\widehat{R}$-module and the canonical map $\varphi' \colon L \to \widehat{L}$ maps an $R$-basis of $L$ to a $\widehat{R}$-basis of $\widehat{L}$. We also have a commutative diagram
    \[
    \xymatrix{L \ar[r]^{\varphi'} \ar[d]_\iota &  \widehat{L} \ar[d]^{\widehat{\iota}} \\  M \ar[r]^{\varphi} &  \widehat{M}.}
    \]
    Thus, 
    $
    \varphi(x) = \varphi(\iota(x)) = \widehat{\iota}(\varphi'(x)) \in \im(\widehat{\iota}).
    $
    Furthermore, $\im(\widehat{\iota})$ is a free submodule of $\widehat{M}$ of finite rank such that $\im(\widehat{\iota}) \hookrightarrow \widehat{M}$ is $\widehat{R}$-pure. Since $\widehat{M}$ is a flat module over the local ring $\widehat{R}$, one can apply \autoref{prop:analog-RG219} (1)$\implies$(5) to conclude that
    $
    \varphi(x) \in c_{\widehat{M}}(\varphi(x))\widehat{M}.    
    $

    It remains to show that $c_M(x)\widehat{R} = c_{\widehat{M}}(\varphi(x))$.  By the $R$-purity of $\iota$, the $\widehat{R}$-purity of $\widehat{\iota}$ and \autoref{lem:content-cyclic-purity}, we get
   $
    \textrm{$c_M \circ \iota = c_L$ and $c_{\widehat{M}} \circ \widehat{\iota} = c_{\widehat{L}}$.}    
    $
    Thus,
    $
    c_M(x) = c_M(\iota(x)) = c_L(x)
    $
    and
    $
    c_{\widehat{M}}(\varphi(x)) = c_{\widehat{M}}(\varphi \circ \iota (x)) 
    = c_{\widehat{M}}(\widehat{\iota} \circ \varphi'(x)) 
    = c_{\widehat L} (\varphi'(x)).
    $ 
    Therefore, to prove $c_M(x)\widehat{R} = c_{\widehat{M}}(\varphi(x))$, it suffices to show that $c_L(x)\widehat{R} = c_{\widehat{L}}(\varphi'(x))$. If $\{e_1,\dots,e_n\}$ is a basis of $L$ and 
    $
    x = r_1e_1 + \dots + r_ne_n,    
    $
    for $r_i \in R$, then we have seen in \autoref{eg:free-OR} that
    $
    c_L(x) = (r_1,\dots,r_n).    
    $
    
    Moreover, if $\widehat{r_i}$ denotes the image of $r_i$ in $\widehat{R}$, then in terms of the basis $\{\varphi'(e_1),\dots,\varphi'(e_n)\}$ of $\widehat{L}$ over $\widehat{R}$, we have
    $
    \varphi'(x) = \widehat{r_1}\varphi'(e_1) + \dots + \widehat{r_n}\varphi'(e_n).    
    $
    Thus, by \autoref{eg:free-OR} again,
    $
    c_{\widehat{L}}(\varphi'(x)) = (\widehat{r_1},\dots,\widehat{r_n}) = (r_1,\dots,r_n)\widehat{R} = c_L(x)\widehat{R},    
    $
    as desired. 

    (2) We have 
    $
    c_M(N)\widehat{R} = \left(\sum_{x \in N}c_M(x)\right)\widehat{R} = \sum_{x \in N} c_M(x)\widehat{R} = \sum_{x \in N} c_{\widehat{M}}(\varphi(x)).
    $
    
    The first equality follows by \autoref{lem:content-subset} because $M$ is Ohm-Rush, the second equality follows because expansions of ideals commute with taking arbitrary sums of ideals and the third equality follows by (1).
    For all $x \in N$, we have again by (1) that
    $
    \varphi(x) \in c_{\widehat M}(\varphi(x))\widehat{M}.    
    $
    Thus,
    \begin{equation}
        \label{eq:conten-subset-completion}
    \varphi(N) \subseteq \left(\sum_{x \in N} c_{\widehat{M}}(\varphi(x))\right)\widehat{M},
    \end{equation}
    and hence, $c_{\widehat M}(\varphi(N)) \subseteq \sum_{x \in N} c_{\widehat{M}}(\varphi(x))$, since by definition, $c_{\widehat M}(\varphi(N))$ is the intersection of all ideals $J$ of $\widehat{R}$ such that $\varphi(N) \subseteq J\widehat{M}$. For any such ideal $J$ and for all $x \in N$,
    $
    \varphi(x) \in \varphi(N) \subseteq J\widehat{M},
    $
    which implies that for all $x \in M$, $c_{\widehat M}(\varphi(x)) \subseteq J$, that is, $\sum_{x \in N} c_{\widehat{M}}(\varphi(x)) \subseteq J$. Since this inclusion of sets holds for all such $J$, we then get $\sum_{x \in N} c_{\widehat{M}}(\varphi(x)) \subseteq c_{\widehat M}(\varphi(N))$, and so,
    $
        c_{\widehat M}(\varphi(N)) =  c_{\widehat{M}}(\varphi(x)).
   $
    Then by \autoref{eq:conten-subset-completion} we get $\varphi(N) \subseteq c_{\widehat{M}}(\varphi(N))\widehat{M}$. This completes the proof of (2).


    (3) By (2) we have $c_M(N)\widehat{R} = c_{\widehat{M}}(\varphi(N))$ and we have $c_{\widehat{M}}(\varphi(N)) = c_{\widehat{M}}(\widehat{R}N)$ because the content of a subset of a module equals the content of the submodule generated by the subset pretty much by definition.

    (4) Since $\widehat{M}$ is a flat $\widehat{R}$-module and hence is also a flat $R$-module ($R$ is Noetherian so $R \to \widehat{R}$ is flat), it suffices to show by \autoref{lem:cyclic-purity-flat} that $\varphi$ is a cyclically pure map of $R$-modules. That is, we have to show that if $I$ be an ideal of $R$, then 
    \begin{align*}
    \varphi \otimes_R \id_{R/I} \colon &M/IM \to \widehat{M}/(I\widehat{R})\widehat{M}\\
    &x + IM \mapsto \varphi(x) + (I\widehat{R})\widehat{M}    
    \end{align*}
    is injective. So suppose $\varphi(x) \in (I\widehat{R})\widehat{M}$. Then by (1), $c_M(x)\widehat{R} = c_{\widehat{M}}(\varphi(x)) \subseteq I\widehat{R}$. Contracting to $R$ and using purity of $R \to \widehat{R}$ gives us $c_M(x) \subseteq I$. Since $M$ is an Ohm-Rush $R$-module, we have $x \in c_M(x)M \subseteq IM$, which is precisely the assertion that $\varphi \otimes_R \id_{R/I}$ is injective.

    (5) If $\phi$ is a pure map of $\widehat R$-modules then $\phi$ is also a pure map of $R$-modules because purity is preserved by restriction of scalars. We now show the converse. So suppose $\phi \colon P \to N$ is pure as a map of $R$-modules. Since $N$ is a flat $R$-module, by \autoref{lem:completion-purity}, the induced map on $\fm$-adic completion
    $
    \widehat{\phi} \colon \widehat{P} \to \widehat{N}    
    $
    is a pure map of $\widehat{R}$-modules. Consider the commutative diagram
    \[
        \xymatrix{P \ar[r]^{\phi} \ar[d] &  N \ar[d] \\  \widehat{P} \ar[r]^{\widehat{\phi}} &  \widehat{N}.}  
    \]
    where the vertical maps are the canonical ones. Note that since $P$ and $N$ are $\widehat{R}$-modules, all the maps in this diagram are $\widehat{R}$-linear. Since $P$ is a flat Ohm-Rush $\widehat R$-module, by (4) applied to the ring $\widehat R$ and the module $P$, we see that the left-vertical map $P \to \widehat{P}$ is a pure map of $\widehat{R}$-modules. Thus, by the commutativity of the above diagram, $\phi \colon P \to N$ must be a pure map of $\widehat{R}$-modules.

    (6) In the notation of part (3), $\widehat{R}I$ is the $\widehat{R}$-submodule of $\widehat{S}$ that is generated by the image of $I$ under the canonical map $S \to \widehat{S}$. Said differently, $\widehat{R}I$ is the image of the ideal $\widehat{R} \otimes_R I$ of $\widehat{R} \otimes_R S$ under the canonical ring map $\widehat{R} \otimes_R S \to \widehat{S}$. So by (3) we get 
    $
    c_S(I)\widehat{R} = c_{\widehat{S}}(\widehat{R}I).    
    $
    But $\widehat{S}$ is a $\widehat{R}$-algebra, and so, by \autoref{rem:content-restriction-of-scalars}, $c_{\widehat{S}}(\widehat{R}I)$ coincides with the content of the ideal of $\widehat{S}$ that is generated by $\widehat{R}I$. This ideal is clearly $I\widehat{S}$. Thus, $c_S(I)\widehat{R} = c_{\widehat{S}}(\widehat{R}I) = c_{\widehat{S}}(I\widehat{S})$.
\end{proof}

\begin{remark}
    \label{rem:is-completion-OR}
    \autoref{cor:content-completion} shows that if $M$ is a flat Ohm-Rush module over a Noetherian local ring $(R,\fm)$, then for all $x \in \widehat{M}$ that are in the image of the canonical map $M \to \widehat{M}$, we have $x \in c_{\widehat{M}}(x)\widehat{M}$. This raises the obvious question of whether $\widehat{M}$ is an Ohm-Rush $\widehat{R}$-module, that is, if $x \in c_{\widehat{M}}(x)\widehat{M}$ for the elements $x \in \widehat{M}$ that are not in $\im(M \to \widehat{M})$. It will turn out that the answer to this question is yes, and does not have anything to do with $M$ being Ohm-Rush. In other words, we will see in \autoref{cor:complete-over-complete-great} that if $M$ is a flat module over a Noetherian local ring $(R,\fm)$, then $\widehat{M}$ is always a flat and OR $\widehat{R}$-module because $\widehat{M}$ is $\fm$-adically complete by \cite[\href{https://stacks.math.columbia.edu/tag/05GG}{Tag 05GG}]{stacks-project}. Note the flatness of $\widehat{M}$ follows by \autoref{lem:completion-purity}.
\end{remark}

We can summarize the key difference between the ML property and the Ohm-Rush property, at least for flat modules over local rings.

\begin{corollary}
    \label{cor:ML-OR-local-difference}
    Let $(R,\fm)$ be a local ring (not necessarily Noetherian) and $M$ be a flat $R$-module. Then we have the following:
    \begin{enumerate}
        \item[$(1)$] $M$ is ML if and only if for any \emph{finitely generated} submodule $N$ of $M$, there exists a submodule $L$ of $M$ such that $N \subseteq L$, $L$ is free of finite rank and $L \hookrightarrow M$ is pure. 
        
        \item[$(2)$] $M$ is Ohm-Rush if and only if for any \emph{cyclic} submodule $N$ of $M$, there exists a submodule $L$ of $M$ such that $N \subseteq L$, $L$ is free of finite rank and $L \hookrightarrow M$ is pure.  
    \end{enumerate}
\end{corollary}

\begin{proof}
    (1) follows from \autoref{prop:ML-SML-modules-colimit-pure-split-free} (1). 
    
    For (2), suppose $M$ is Ohm-Rush and $N$ is a cyclic submodule of $M$. Let $f$ be a generator of $N$ and consider the unique $R$-linear map $v \colon R \to M$ with image $N$ that sends $1 \mapsto f$. Since $f \in c_M(f)M$, by \autoref{prop:analog-RG219}$(5)\implies(1)$ we have that $f$ (hence $N$) is contained in a submodule $L$ of $M$ such that $L$ is free of finite rank and $L \hookrightarrow M$ is pure. 
    
    Conversely, suppose $f \in M$ and let $N = Rf$. Choose a submodule $L$ of $M$ containing $N$ such that $L$ is free of finite rank and $L\hookrightarrow M$ is pure. Let $v\colon R \to M$ be the unique $R$-linear map that sends $1 \mapsto f$ and let $u \colon R \to L$ be the unique $R$-linear map that sends $1 \mapsto f$. If $i \colon L \hookrightarrow M$ is the inclusion, then $v = i \circ u$. Since $i$ is pure and hence cyclically pure, $u$ is a cyclic stabilizer for $v$ by \autoref{lem:factor-cyclically-pure}. Thus, $f \in c_M(f)M$ by \autoref{thm:content-equivalences}$(4) \implies (1)$.
\end{proof}

\begin{remark}
Let $M$ be a flat module over a local ring $(R,\fm)$. If $\Sigma$ is the collection of $R$-submodules of $M$ that are pure in $M$ and are free of finite rank, then $\Sigma$ is filtered by inclusion when $M$ is ML. In this case $M$ is a filtered union of elements of $\Sigma$. On the other hand, when $M$ is  Ohm-Rush (and not necessarily ML) then $\Sigma$ may not be filtered by inclusion and $M$ is just a union (not a filtered union) of elements of $\Sigma$.
\end{remark}

\autoref{cor:ML-OR-local-difference} shows that for a flat Ohm-Rush module $M$ over a local ring $(R,\fm)$ and any $f \in M$, one can always find a free submodule $L$ of $M$ of finite rank containing $f$ such that $L$ is pure in $M$. However, the proof of the existence of $L$ is not very explicit, since it relies on the stabilization results of \autoref{prop:analog-RG219} and the equally less explicit \autoref{prop:RG-2.1.9}. The next result provides criteria using the content function that can be used to construct $L$ more explicitly; see \autoref{rem:constructing-free-submodule-pure-OR} for the construction.

\begin{proposition}\label{prop:localpuren}
    Let $(R,\fram,\kappa)$ be a local ring (not necessarily Noetherian) and $M$ a flat 
    $R$-module.  
    Let $L$ be a finitely generated
    submodule and $x_1, \ldots, x_n$ a minimal generating set. Let $\iota: L \hookrightarrow M$ be the inclusion map.  The following are equivalent:  
    \begin{enumerate}
        \item[$(1)$]\label{it:cpmingensn} For any $y\in L \setminus \fram L$, $c_M(y)=R$.
    
        \item[$(2)$]\label{it:cpmaxn} $L \cap \fram M = \fram L$, i.e.,  $\iota \otimes_R \kappa$ is injective.
        
        \item[$(3)$]\label{it:cpidealgensn} For any ideal $I$ of $R$, any minimal generating set $y_1, \ldots, y_n$ of $L$, and any $n$-tuple $a_1, \ldots, a_n \in R$, if $\sum_{i=1}^n a_i y_i \in IM$, then $(a_1, \ldots, a_n) \subseteq I$.
        
        \item[$(4)$]\label{it:cppn} $L$ is free and $\iota$ is pure.
        
        \item[$(5)$]\label{it:cpcpn} $L$ is free and $\iota$ is cyclically pure.
    
        \item[$(6)$]\label{it:cpcontentsumsn} For any $n$-tuple $a_1, \ldots, a_n \in R$, we have $c_M(\sum_{i=1}^n a_i x_i) = (a_1, \ldots, a_n)$.
    \end{enumerate}
    Consider also the following statement: 
    \begin{enumerate}
        \item[$(7)$]\label{it:cpsinglecontentsumn} There is some $a_1, \ldots, a_n \in R$ that form a minimal generating set for the ideal they generate, such that $c_M(\sum_{i=1}^n a_i x_i) = (a_1, \ldots, a_n)$.
    \end{enumerate}
    If $M$ is Ohm-Rush, then (7) implies the equivalent assertions (1)-(6).
    \end{proposition}
    
    \begin{proof}
    We will make frequent use of \autoref{lem:factorcontent}, which guarantees because $M$ is flat that $rc_M(f) \subseteq c_M(rf)$ for any $r\in R$, $f \in M$, with equality whenever $c_M(f)=R$.
    
    (1) $\Longleftrightarrow$ (2): Both statements amount to the assertion that $L \setminus \fram L = L \setminus \fram M$.

    (1) $\implies$ (3): Suppose $I$ is an ideal that violates the statement. Let $k$ be minimal with $1\leq k \leq n$ such that there exist some minimal generating set $y_1, \ldots, y_n$ of $L$ and elements $a_1, \ldots, a_k \in R$ such that $\sum_{i=1}^k a_i y_i \in IM$, but $a_k \notin I$. Note that $k > 1$ because if $a_1y_1 \in IM$, then $a_1$ has to be in $I$ because $c_M(y_1) = R$ by (1) and we have $a_1 \in a_1R = a_1c_M(y_1) \subseteq c_M(a_1y_1) \subseteq I$. 
    
    For our choice of $k$, we have 
    $
    a_k y_k \in (a_1, \ldots, a_{k-1})L + IM \subseteq ((a_1, \ldots, a_{k-1}) + I)M,
    $
    so that since $c_M(y_k)=R$ (again using (1)), we have 
    $a_k \in a_k R = a_kc_M(y_k) \subseteq c_M(a_k y_k) \subseteq (a_1, \ldots, a_{k-1}) + I.$  
    Say $a_k = u + \sum_{i=1}^{k-1} d_i a_i$, with $d_i \in R$ and $u \in I$.  Then
    \[
    \sum_{i=1}^{k-1} a_i (y_i + d_i y_k) = \left(\sum_{i=1}^{k-1} a_i y_i\right) + a_k y_k - uy_k = \left(\sum_{i=1}^{k} a_i y_i\right) - uy_k \in IM.
    \]
    But 
    $
    \{y_1 + d_1 y_k, y_2 + d_2y_k, \ldots, y_{k-1} + d_{k-1} y_k, y_k, \ldots, y_n\}
    $ 
    is also a minimal generating set for $L$.  Hence by minimality of $k$, we have $a_1, \ldots, a_{k-1} \in I$.  Thus, $a_k \in (a_1, \ldots, a_{k-1}) + I = I$, which contradicts the fact that $a_k \notin I$.  Hence, there is no ideal $I$ that violates the assertion of (3).

    (3) $\implies$ (5): For freeness, let $I=0$ in the statement of (3).  For cyclic purity, let $I$ be an ideal and $f\in IM \cap L$.  Then there exist $a_1, \ldots, a_n \in R$ with $f = \sum_{i=1}^n a_i x_i \in IM$, so that by (3), $(a_1, \ldots, a_n) \subseteq I$, whence $f = \sum_{i=1}^n a_i x_i \in (a_1, \ldots, a_n)L \subseteq IL$.

    (4) $\Longleftrightarrow$ (5): This follows from \autoref{lem:cyclic-purity-flat} because $M$ is flat.

    (5) $\implies$ (6): Note $\{x_1,\dots,x_n\}$ is a free basis of $L$. Since $\sum_{i=1}^n a_i x_i \in (a_1, \ldots, a_n)M$, we have $c_M(\sum_{i=1}^n a_i x_i) \subseteq (a_1, \ldots, a_n)$.
    For the reverse containment, let $I$ be an ideal such that $\sum_{i=1}^n a_i x_i \in IM$.  Then $\sum_{i=1}^n a_i x_i \in IM \cap L = IL$ by cyclic purity, so by linear independence of the $x_i$ we have $a_i \in I$ for all $i$.  Hence $(a_1, \ldots, a_n) \subseteq I$, and since $I$ was arbitrary with $\sum_i a_i x_i \in IM$, we have $(a_1, \ldots, a_n) \subseteq c_M(\sum_i a_i x_i)$.

    (6) $\implies$ (5): To see that $L$ is free, let $a_1, \ldots, a_n \in R$ with $\sum_i a_i x_i =0$.  Then 
    $(0)=c_M(0) = c_M\left(\sum_i a_i x_i\right) = (a_1, \ldots, a_n),$ 
    so that $a_1 = \cdots = a_n = 0$.
    
    To see that $\iota$ is cyclically pure, let $f\in IM \cap L$.  Since $f\in L$, there exist $a_1, \ldots, a_n \in R$ with $f=\sum_{i=1}^n a_i x_i$.  Then since $f\in IM$, we have $(a_1, \ldots, a_n) = c_M(\sum_i a_i x_i) = c_M(f) \subseteq I$, whence $a_i \in I$ for all $i$.  Thus, $f = \sum_{i=1}^n a_i x_i \in IL$. This shows that for all ideals $I$, $IM \cap L = IL$, or equivalently, that $\iota \otimes_R \id_{R/I}$ is injective. 

    (5) $\implies$ (2): This follows because by cyclic purity because $\iota \otimes_R \kappa$ is injective.

    This completes the proof of the equivalence of conditions (1) through (6).

    Now assume that $M$ is Ohm-Rush. It suffices to show that (7) $\implies$ (1) under these hypotheses. By hypothesis, $\{a_1,\dots,a_n\}$ is a minimal generating set of the ideal $(a_1,\dots,a_n)$.  
    Let $y \in L \setminus \fram L$ and expand $y$ to form a minimal generating set $y=y_1, y_2, \ldots, y_n$ for $L$.  Then there is some invertible matrix $(\lambda_{ij})$ with entries in  $R$ such that for all $1\leq j \leq n$, we have $x_j = \sum_{i=1}^n \lambda_{ij} y_i$.  Set $b_i := \sum_{j=1}^n \lambda_{ij} a_j$.  Then $\{b_1,\dots,b_n\}$ is also a minimal generating set for the ideal $(b_1,\dots,b_n) = (a_1, \ldots, a_n)$.  Moreover, $\sum_i b_i y_i=\sum_i \left(\sum_j \lambda_{ij} a_j\right) y_i = \sum_j a_j \left(\sum_i \lambda_{ij} y_i\right) = \sum_j a_j x_j.$
    
    Now, suppose for contradiction $c_M(y) = c_M(y_1) \subseteq \fram$.  Then we have 
    \begin{align*}
    (b_1, \ldots, b_n) &= (a_1, \ldots, a_n) 
    = c_M(\sum_j a_j x_j) 
    = c_M(\sum_i b_i y_i)\\ 
    &\subseteq \sum_i c_M(b_iy_i)
    \subseteq b_1 c(y_1) + (b_2, \ldots, b_n) 
    \subseteq \fram b_1 + (b_2, \ldots, b_n).
    \end{align*}  
    Here for the inclusion $c_M(\sum_i b_i y_i) \subseteq \sum_i c_M(b_iy_i)$ we are crucially using that $M$ is Ohm-Rush. By Nakayama's lemma, we then have $b_1 \in (b_2, \ldots, b_n)$, contradicting the fact that the $b_i$ form a minimal generating set of $(b_1,\dots,b_n)$.  Thus, $c_M(y) = R$.
    \end{proof}

    \begin{remark}
        \label{rem:constructing-free-submodule-pure-OR}
        Let $M$ be a flat Ohm-Rush module over a local ring $(R,\fm)$ and let $f \in M$. \autoref{prop:localpuren} provides the following, more explicit, method of constructing a submodule $L$ of $M$ containing $f$ such that $L$ is free of finite rank and $L \hookrightarrow M$ is pure. Namely, since 
        $f \in c_M(f)M,$
        $c_M(f)$ is finitely generated. Let $a_1,\dots,a_n$ be a minimal set of generators of the ideal $c_M(f)$. Then there exists $x_1,\dots,x_n \in M$ such that
        $
        f = a_1x_1 + \dots + a_nx_n.    
        $
        Let $L$ be the submodule of $M$ generated by $x_1,\dots,x_n$. By construction, $f \in L$ and
        $
        c_M(\sum_{i=1}^na_ix_i) = c_M(f) = (a_1,\dots,a_n).    
        $
        Thus, by \autoref{prop:localpuren}$(7)\implies(4)$, $L$ is free and $L \hookrightarrow M$ is pure provided that we can show that $x_1,\dots,x_n$ is a minimal generating set for $L$. If not, then without loss of generality $x_n \in \sum_{i=1}^{n-1} Rx_i$.  Say $x_n = \sum_{i<n} b_i x_i$.  Then 
        $f = \sum_{i=1}^n a_i x_i = \sum_{i=1}^{n-1} (a_i + a_n b_i)x_i.$  Let $J$ be the ideal generated by the $n-1$ elements $a_i + a_n b_i$, $1 \leq i < n$. Clearly $J \subseteq (a_1,\ldots, a_n) = c_M(f)$.  On the other hand, $f \in JM$, so $c_M(f) \subseteq J$.  Thus $J=c_M(f)$, so $c_M(f)$ can be generated by fewer than $n$ elements, a contradiction.
    \end{remark}
    
Perhaps a surprising aspect of \autoref{prop:localpuren} is that if $M$ is a flat module over a local ring $(R,\fm,\kappa)$ and $L$ is a finitely generated submodule of $M$, then injectivity of the induced map $L/\fm L \to M/\fm M$ is enough to ensure that $L$ is pure in $M$. Note that the freeness of $L$ is really a consequence of the purity of $L \hookrightarrow M$. Indeed, since $M$ is a flat $R$-module, it follows from the purity of $L \hookrightarrow M$ that $L$ is also a flat $R$-module by \autoref{lem:pure-submodule-flat}. Then $L$ is free because a finitely generated flat module over a local ring (not necessarily Noetherian) is free using the equational criterion of flatness \cite[Thm.\ 7.10]{MatsumuraCommutativeRingTheory}.

    The assumption that $L$ is a finitely generated \emph{submodule} of a flat module $M$ can be relaxed to deduce purity of a map $L \to M$ from injectivity of the induced map $L/\fm L \to M/\fm M$ if we assume $L$ is finitely presented. 

    \begin{proposition}
        \label{prop:local-purity-from-residue-field-map}
        Let $(R,\fm,\kappa)$ be a local ring (not necessarily Noetherian). Let 
        $
        \varphi \colon L \to M    
        $
        be a linear map where $L$ is a finitely presented $R$-module and $M$ is a flat $R$-module. Suppose that the induced map  $\varphi \otimes_R \id_{\kappa}$ is injective. Then $\varphi$ is a pure map.
    \end{proposition}

    \begin{proof}
        By Lazard's theorem \cite{Lazard}, we can express $M$ as a colimit of a filtered system $\{(F_i, u_{ij}) \colon i, j \in (I, \leq), i \leq j\}$ of finite free modules $F_i$. Since $L$ is finitely presented, there exists $i_0 \in I$ and a lift $\varphi_{i_0} \colon L \to F_{i_0}$ of $\varphi$ along $F_{i_0} \to M$. Then 
        $
        \varphi = \colim_{j \geq i_0} u_{i_0j} \circ \varphi_{i_0}.
        $
        Since a filtered colimit of universally injective maps of modules is universally injective, it suffices to show that for all $j \geq i_0$, $u_{i_0 j} \circ \varphi_{i_0}$ is a pure map. Moreover, since $u_{i_0 j} \circ \varphi_{i_0}$ factors $\varphi$, it follows by the injectivity of $\varphi \otimes_R \id_{\kappa}$ that for all $j \geq i_0$,
        $
        (u_{i_0 j}\circ \varphi_{i_0}) \otimes_R \id_{\kappa} \colon L/\fm L \to F_j/\fm F_j  
        $
        is also injective. The upshot is that replacing $M$ by $F_j$ we may assume that $M$ is finitely presented. Then $\varphi$ is injective and $\coker(\varphi)$ is a flat $R$-module by \cite[\href{https://stacks.math.columbia.edu/tag/046Y}{Tag 046Y}]{stacks-project}. But any injective map with flat cokernel is pure \cite[\href{https://stacks.math.columbia.edu/tag/058M}{Tag 058M}]{stacks-project}.
    \end{proof}

\autoref{prop:localpuren} and \autoref{prop:local-purity-from-residue-field-map} have the following global consequences.

    \begin{corollary}
        \label{cor:cyclic-purity-radical-ideals}
        Let $R$ be a ring and $v \colon P \to M$ be a linear map where $M$ is flat. Suppose that either of the following conditions hold:
        \begin{enumerate}
            \item[$(a)$] $P$ is finitely generated and $v$ is injective.
            \item[$(b)$] $P$ is finitely presented.
        \end{enumerate}
        Then we have the following:
        \begin{enumerate}
        \item[$(1)$] \textrm{$\{\p \in \Spec(R)  \colon v \otimes_R \id_{R/\p}$ is injective$\} \subseteq \{\p \in \Spec(R) \colon v_\p$ is pure in $\Mod_{R_\p}\}$}. Moreover, if $P$ is $R$-flat, then the two sets are equal.
        
        \item[$(2)$] Suppose that for all maximal ideals $\fm$ of $R$, the induced map $v \otimes_R \id_{R/\fm} \colon P/\fm P \to M/\fm M$ is injective. Then $v$ is a pure map.
        \end{enumerate}
    \end{corollary}
     
    \begin{proof}
        Assume $(a)$. Identifying $P$ with its isomorphic image $\im(v)$, we will further assume that $P$ is a finitely generated submodule of $M$ and $v$ is the inclusion map $\iota$ of $P$ into $M$.

        (1) Suppose $\iota \otimes_R \id_{R/\p} \colon P/\p P \to M/ \p M$ is injective. Localizing at $\p$ we then get 
        $
        \iota_\p \otimes_{R_\p} \kappa(\p) \colon P_\p/\p P_\p \to M_\p/\p M_\p
        $
        is injective. Since $\iota_\p$ is injective, \autoref{prop:localpuren} (2)$\implies$(4) shows that $\iota_\p$ is pure in $\Mod_{R_\p}$. Thus,
        $
            \textrm{$\{\p \in \Spec(R) \colon \iota \otimes_R \id_{R/\p}$ is injective$\} \subseteq \{\p \in \Spec(R) \colon \iota_\p$ is pure in $\Mod_{R_\p}\}$}.    
        $

        Suppose $P$ is a flat $R$-module. We want to show then that the other inclusion also holds. Let $\p \in \Spec(R)$ such that $\iota_\p$ is pure. Consider the commutative diagram
        \[
        \begin{tikzcd}
            P/\p P \arrow[r, "\iota \otimes_R \id_{R/\p}"] \arrow[d]
              & M/\p M \arrow[d] \\
            P_\p/\p P_\p \arrow[r, "\iota_\p \otimes_{R_\p} \kappa(\p)"]
          & M_\p/\p M_\p. \end{tikzcd}
          \]
          The bottom horizontal map is injective by purity of $\iota_\p$. Since $P$ is $R$-flat, the left vertical map is also injective because it can be identified with the map obtained by applying $\otimes_R \id_{P}$ to the injection $R/\p \hookrightarrow R_\p/\p R_\p = (R/\p)_\p$. By commutativity, it follows that the top horizontal map is injective, that is,
          \[
            \textrm{$\{\p \in \Spec(R) \colon \iota_\p$ is pure in $\Mod_{R_\p}\} \subseteq \{\p \in \Spec(R) \colon \iota \otimes_R \id_{R/\p}$ is injective$\}$}.  
          \]

        (2) Since purity of a map of $R$-modules can be checked locally at maximal ideals, it follows that if $\iota_\fm$ is $R_\fm$-pure for all maximal ideals $\fm$ of $R$, then $\iota$ is $R$-pure. By hypothesis, the induced map
        $
        \iota \otimes_R \id_{R/\fm} \colon P/\fm P \to M/\fm M    
        $
        is injective. Thus by (1), $\iota_\fm$ is $R_\fm$-pure.

        The proof of $(1)$ and $(2)$ assuming $(b)$ is similar once we have \autoref{prop:local-purity-from-residue-field-map}. We omit the details.
    \end{proof}

\subsection{Descent of (cyclic) stabilizers} We want to show that the Ohm-rush property satisfies pure descent for flat modules (\autoref{thm:descentOhm-Rush}). But first, we will prove some preliminary results about descending (cyclic) purity and (cyclic) stabilizers. We begin with the following well-known result.

\begin{lemma}
\label{lem:descent-cyclic-purity}
Let $R \to S$ and $M \to N$ be an 
$R$-linear map. 
\begin{enumerate}
    \item[$(1)$] If $R \to S$ is pure and $S \otimes_R M \to S \otimes_R N$ is pure as a map of $S$-modules, then $M \to N$ is pure as a map of $R$-modules.
    \item[$(2)$] If $R \to S$ is cyclically pure, $M$ is flat and $S \otimes_R M \to S \otimes_R N$ is cyclically pure as a map of $S$-modules, then $M \to N$ is cyclically pure as a map of $R$-modules. 
\end{enumerate}
\end{lemma}

\begin{proof}
Consider the commutative diagram
\[
\xymatrix{&P \otimes_R M \ar[d] \ar[r] &P \otimes_R N \ar[d]\\
&S \otimes_R (P \otimes_R M) \ar[r] &S \otimes_R (P \otimes_R N).}
\]

(1) If $R \to S$ is pure, the left (resp. right) vertical map is obtained by tensoring
the pure ring map $R \to S$ with $P \otimes_R M$ (resp. $P \otimes_R N$). Thus, both the left and right vertical maps are injective.

The bottom map $S \otimes_R (P \otimes_R M) \to S \otimes_R (P \otimes_R N)$ is injective because it can be identified with 
$
(S \otimes_R P) \otimes_S (S \otimes_R M) \to (S \otimes_R P) \otimes_S (S \otimes_R N),
$
which is injective because by the assumption $S \otimes_R M \to S \otimes_R N$ is pure as a map of $S$-modules. Then by the commutativity of the above diagram, $P \otimes_R M \to P \otimes_R N$ must be injective as well.

Since the base change of a cyclic $R$-module along $S$ is a cyclic $S$-module, the proof of (2) follows using an argument similar to (1).
\end{proof}

Our next result is about descending (cyclic) stabilizers along pure maps.

\begin{proposition}
\label{prop:descent-cyclic-stabilizers}
Let $R \to S$ be a ring homomorphism. Let $\{(L_{i}, u_{ij}) \colon i, j \in I, i \leq j\}$ be a system of
finitely presented $R$-modules indexed by a filtered poset $(I, \leq)$. Set 
$
M \coloneqq \colim_i L_i.
$
For all $i \in I$, let $u_i \colon L_i \to M$ be the associated map. 
Let $F$ be a finitely
presented $R$-module and $v \colon F \to M$ be an $R$-linear map.
\begin{enumerate}
    \item[$(1)$] If $R \to S$ is pure and $\id_S \otimes_R v$ admits a stabilizer in $\Mod_S$, then there exists an index 
    $i \in I$ and a map $v_i \colon F \to L_i$ such that 
    $
    v = u_i \circ v_i
    $
    and $v_i$ is a stabilizer of $v$. 
    
    \item[$(2)$] If $R \to S$ is pure and $\id_S \otimes_R v$ admits a cyclic stabilizer in $\Mod_S$ that
    $\id_S \otimes_R v$ dominates, then there exists an index 
    $i \in I$ and a map $v_i \colon F \to L_i$ such that 
    $
    v = u_i \circ v_i
    $
    and $v_i$ is a cyclic stabilizer of $v$ that $v$ dominates.

    \item[$(3)$] If $R \to S$ is cyclically pure, $L_i$ is flat (equivalently, projective) for all $i$ and $\id_S \otimes_R v$ admits a cyclic stabilizer in $\Mod_S$ that $id_S \otimes_R v$ dominates, then there exists an index 
    $i \in I$ and a map $v_i \colon F \to L_i$ such that 
    $
    v = u_i \circ v_i
    $
    and $v_i$ is a cyclic stabilizer of $v$ that $v$ dominates.
\end{enumerate}
\end{proposition}

\begin{proof}
We will prove (1) and (2) simultaneously. Note that $S \otimes_R M = \colim_i S \otimes_R L_{i}$ since 
tensor products commute with filtered colimits. Since $F$ is a finitely presented 
$R$-module, there exists $i \in I$ and a map $v_i \colon F \to L_i$ such that the 
following diagram commutes
\[
\xymatrix@C=4em{&L_i \ar[d]^{u_i}\\
F \ar[ur]_{v_i} \ar[r]_v & M.}
\]

Then 
$
v = u_i \circ v_i = (\colim_{j \geq i} u_{ij}) \circ v_i = \colim_{j \geq i} (u_{ij} \circ v_i).
$
For all $j \geq i$, define 
$v_j \coloneqq u_{ij} \circ v_i.$ 
We have
$
S \otimes_R M = \colim_{j \geq i} S \otimes_R L_j,
$
and
$
\id_S \otimes_R v = \colim_{j \geq i} \id_S \otimes_R v_j.
$
Let 
$f \colon S \otimes_R F \to P$ 
be a stabilizer of $\id_S \otimes_R v$ (resp. a 
cyclic stabilizer of $\id_S \otimes_R v$ that $v$ dominates).
By \autoref{lem:domination} (2), $\id_S \otimes_R v$ factors through $f$, that is,
there exists $\varphi \colon P \to S \otimes_R M$ such that 
\[
\xymatrix@C=4em{&S \otimes_R F \ar[dl]_{f} \ar[d]^{\id_S \otimes_R v}\\
P \ar[r]_{\varphi} & S \otimes_R M}
\]
commutes. Since $P$ is a finitely presented
$R$-module, replacing $i$ by a larger index, we may assume that there
exists a lift 
$
\varphi_i \colon P \to S \otimes_R L_i
$
of $\varphi$ along $\id_S \otimes_R u_i$, that is, 

\[
\xymatrix@C=4em{&S \otimes_R L_i \ar[d]^{\id_S \otimes_R u_i}\\
P \ar[ur]_{\varphi_i} \ar[r]_\varphi & S \otimes_R M}
\]
commutes. Then we have two maps
\[
\begin{tikzcd}
   S \otimes_R F \ar[r,shift left=.75ex,"\id_S \otimes_R v_i"]
  \ar[r,shift right=.75ex,swap,"\varphi_i \circ f"]
&
S \otimes_R L_i
\end{tikzcd}
\]
such that
$
(\id_S \otimes_R u_i) \circ (\id_S \otimes_R v_i) = \id_S \otimes_R (u_i \circ v_i)
= \id_S \otimes_R v = \varphi \circ f = (\id_S \otimes_R u_i) \circ (\varphi_i \circ f).
$
Let $f_1, \dots, f_n$ be generators of the $S$-module $S \otimes_R F$. Since 
$M = \colim_{j \geq i} S \otimes_R L_j$, there exists $j \geq i$ such that for all
$\ell = 1, \dots, n$, $f_\ell$ has the same image in $S \otimes_R L_j$
under the two maps
\[
\begin{tikzcd}
   S \otimes_R F \ar[r,shift left=.75ex,"\id_S \otimes_R v_i"]
  \ar[r,shift right=.75ex,swap,"\varphi_i \circ f"]
&
S \otimes_R L_i
\ar[r, "\id_S \otimes_R u_{ij}"]
&
S \otimes_R L_j.
\end{tikzcd}
\]
Since an $S$-linear map whose domain is $S \otimes_R F$ is completely determined
by the images of the generators $f_1,\dots,f_n$, this implies 
$
\id_S \otimes_R v_j = (\id_S \otimes_R u_{ij}) \circ (\id_S \otimes_R v_i) = 
(\id_S \otimes_R u_{ij}) \circ (\varphi_i \circ f). 
$
Let $\varphi_j \coloneqq (\id_S \otimes_R u_{ij}) \circ \varphi_i$. We then see
that the following diagram commutes
 \[
    \xymatrix@C=6em@R=4em{&S \otimes_R F \ar[dl]_f \ar[d]_{\id_S \otimes_R v_j} \ar[dr]^{\id_S \otimes_R v}  \\
    P \ar[r]_{\varphi_j} &S\otimes_R L_j \ar[r]_{\id_S \otimes_R u_j} &S \otimes_R M.}
    \]
Since $f$ is a stabilizer (resp. a cyclic stabilizer) of $\id_S \otimes_R v$, 
by \autoref{lem:domination} (5) (resp. by \autoref{lem:cyclic-stab-diagram}), we
see that $\id_S \otimes_R v_j$ is a stabilizer (resp. a cyclic stabilizer) of
$\id_S \otimes_R v$.
Now consider the pushout
\[
    \xymatrix{F \ar[r]_{v_j} \ar[d]_ v &  L_j \ar[d]^{v'} \\  M \ar[r]_{v_j'} &  T.}
    \]
Since pushouts are preserved by base change, 
\begin{equation}
    \label{fig:pushout}
    \begin{gathered}
    \xymatrix@C=3em@R=3em{S\otimes_R F \ar[r]_{\id_S \otimes_R v_j} \ar[d]_{\id_S \otimes_R v} &  S \otimes_R L_j \ar[d]^{\id_S \otimes_R v'} \\  S \otimes_R M \ar[r]_{\id_S \otimes_R v_j'} & S \otimes_R T.}
    \end{gathered}
\end{equation}
    is also a pushout diagram. The fact that $\id_S \otimes_R v_j$ is a stabilizer
    (resp. a cyclic stabilizer) for $\id_S \otimes_R v$ is equivalent to the assertion
    that the maps $\id_S \otimes_R v'$ and $\id_S \otimes_R v_j'$ are both pure
    (resp. both cyclically pure) in $\Mod_S$ by \autoref{lem:domination} (3) (resp. by \autoref{lem:cyclic-stab-pushout}). Since $R \to S$ is a pure ring map,
    \autoref{lem:descent-cyclic-purity} (1) then implies that
    $v'$ and $v_j'$ are pure (resp. cyclically pure) in $\Mod_R$. Thus 
    $v_j$ is a stabilizer (resp. a cyclic stabilizer) for $v$. Note that $v$ dominates $v_j$ since $v  = u_j \circ v_j$ factors through $v_j$. 

    (3) If the $L_i$ are all flat, then $M = \colim_i L_i$ is also flat. 
    The proof of descent of cyclic stabilizers is now identical to (1) and (2), except in the step where we analyze the pushout \autoref{fig:pushout}. Since $M$ and $L_j$ are both flat, the cyclic purity of $\id_S \otimes_R v'_j \colon S \otimes_R M \to S \otimes_R T$ and $\id_S \otimes_R v' \colon S \otimes_R L_j \to S \otimes_R T$ as maps of $S$-modules implies that $v'_j$ and $v'$ are cyclically pure maps of $R$-modules by \autoref{lem:descent-cyclic-purity} (2) since we are assuming that $R \to S$ is cyclically pure.
\end{proof}

Since any module can be expressed as a filtered colimit of finitely presented modules and since flat modules can be expressed as a filtered colimit of finite free modules, we then obtain the following:

\begin{corollary}
    \label{cor:descent-cyclic-stabilizers}
    Let $R \to S$ be a ring map. Let $v \colon F \to M$ be an $R$-linear map where $F$ is a finitely presented $R$-module. Then we have the following:
    \begin{enumerate}
        \item[$(1)$] If $R \to S$ is pure and $\id_S \otimes_R v$ admits a stabilizer in $\Mod_S$, then $v$ admits a stabilizer in $\Mod_R$.
    
        \item[$(2)$] If $R \to S$ is pure and $\id_S \otimes_R v$ admits a cyclic stabilizer that $\id_S \otimes_R v$ dominates in $\Mod_S$, then $v$ admits a cyclic stabilizer in $\Mod_R$ that $v$ dominates.
    
        \item[$(3)$] If $R \to S$ is cyclically pure, $M$ is flat and $\id_S \otimes_R v$ admits a cyclic stabilizer that $\id_S \otimes_R v$ dominates in $\Mod_S$, then $v$ admits a cyclic stabilizer in $\Mod_R$ that $v$ dominates.
    \end{enumerate}
\end{corollary}

Descent of cyclic stabilizers gives descent of the Ohm-Rush property. We note that our contribution is part (2) of the next result because descent of the ML property along pure maps was shown in \cite[Part II, Prop.\ 2.5.1]{rg71}. We believe the proof of descent in \cite{rg71} is correct even though there is some confusion about its veracity. In any case, we reprove their result here since the proof follows easily by the theory already developed.

\begin{theorem}
\label{thm:descentOhm-Rush}
Let $R \to S$ be a ring map and $M$ be an $R$-module.
\begin{enumerate}
    \item[$(1)$] If $R \to S$ is pure and $S \otimes_R M$ is a ML S-module, then
    $M$ is a ML $R$-module.
    
    \item[$(2)$] If $R \to S$ is cyclically pure, $M$ is flat and $S \otimes_R M$ is an Ohm-Rush
    $S$-module, then $M$ is an Ohm-Rush $R$-module.
\end{enumerate}
\end{theorem}

\begin{proof}
$(1)$ If $F$ is a finitely presented $R$-module and $v \colon F \to M$
is an $R$-linear map, we have to show that $v$ admits a stabilizer. 
Upon expressing $M$ as a directed colimit of finitely presented $R$-modules, this follows by \autoref{prop:descent-cyclic-stabilizers} (1) 
because $\id_S \otimes_R v$
admits a stabilizer as a map of $S$-modules.

$(2)$ By \autoref{thm:content-equivalences}, we have to show that every map $v \colon R \to M$ admits a cyclic stabilizer that $v$ dominates. Since $S \otimes_R M$ is a flat Ohm-Rush $S$-module, $\id_S \otimes_R v$ admits a cyclic stabilizer in $\Mod_S$ that $\id_S \otimes_R v$ dominates by \autoref{thm:content-equivalences}. Since $M$ is flat, $v$ admits a cyclic stabilizer in $\Mod_R$ that $v$ dominates by \autoref{cor:descent-cyclic-stabilizers} (3).
\end{proof}

\begin{remark}
    \label{rem:cyclic-pure-descent-OR}
    \mbox{}
    \begin{enumerate}
        \item We do not know if one can drop the flatness assumption on $M$ in \autoref{thm:descentOhm-Rush} (2).
        
        \item Cyclically pure descent of the Ohm-Rush property for flat modules follows from the more general assertion about descent of cyclic stabilizers. However, one can also give a more direct argument of descent of the Ohm-Rush property that we include for the reader's convenience. Let $\{I_\alpha \colon \alpha \in A\}$ be a collection of ideals of $R$. Assuming that $R \to S$ is cyclically pure and $S \otimes_R M$ is an Ohm-Rush $S$-module for a flat $R$-module $M$, we need to show that 
        $
        \bigcap_{\alpha \in A} I_\alpha M = \left(\bigcap_{\alpha \in A}I_\alpha\right) M.    
        $
        Since taking inverse images commutes with arbitrary intersection, cyclic purity of $R \to S$ implies that the induced map
        $
        \frac{R}{\bigcap_{\alpha \in A} I_\alpha} \to \frac{S}{\bigcap_{\alpha \in A} I_\alpha S}   
        $
        is injective. Since $M$ is a flat $R$-module, this implies that 
        \begin{equation}
            \label{eq:descent-1}
            \frac{M}{\left(\bigcap_{\alpha \in A}I_\alpha\right) M} \to \frac{S \otimes_R M}{\left(\bigcap_{\alpha \in A}I_\alpha S\right)(S\otimes_R M)}
        \end{equation}
        is injective as well. Since $S \otimes_R M$ is an Ohm-Rush $S$-module, we then get
        \begin{equation}
            \label{eq:descent-2}
            \left(\bigcap_{\alpha \in A}I_\alpha S\right)(S\otimes_R M) = \bigcap_{\alpha \in A}I_\alpha S (S \otimes_R M).
        \end{equation}
        Then \autoref{eq:descent-1} and \autoref{eq:descent-2} imply that the composition
        \begin{equation}
            \label{eq:descent-3}
            \frac{M}{\left(\bigcap_{\alpha \in A}I_\alpha\right) M} \to \frac{S \otimes_R M}{\left(\bigcap_{\alpha \in A}I_\alpha S\right)(S\otimes_R M)} \to \prod_{\alpha \in A} \frac{S \otimes_R M}{(I_\alpha S)(S \otimes_R M)}  
        \end{equation}
        is injective. But \autoref{eq:descent-3} has a factorization
\[\begin{tikzcd}
	{\frac{M}{\left(\bigcap_{\alpha \in A}I_\alpha\right) M}} && {\frac{S \otimes_R M}{\left(\bigcap_{\alpha \in A}I_\alpha S\right)(S\otimes_R M)}} && {\prod_{\alpha \in A} \frac{S \otimes_R M}{(I_\alpha S)(S \otimes_R M)} } \\
	\\
	&& {\prod_{\alpha \in A} \frac{M}{I_\alpha M}}
	\arrow[from=1-1, to=1-3]
	\arrow[from=1-3, to=1-5]
	\arrow[from=3-3, to=1-5]
	\arrow[from=1-1, to=3-3]
\end{tikzcd}\]
So the first map 
$
    {\frac{M}{\left(\bigcap_{\alpha \in A}I_\alpha\right) M}} \to    {\prod_{\alpha \in A} \frac{M}{I_\alpha M}}
$
is injective as well. But this precisely means that
 $\bigcap_{\alpha \in A} I_\alpha M = \left(\bigcap_{\alpha \in A}I_\alpha\right) M.$
    \end{enumerate}
\end{remark}

\begin{remark}
    We will see in future work that the SML property does not satisfy pure/faithfully flat descent, that is, if $R \to S$ is a faithfully flat ring map and $M$ is an $R$-module (even a flat one) such that $S \otimes_R M$ is a SML $S$-module, then it is not true in general that $M$ is a SML $R$-module. Our example will rely on the Frobenius map of an excellent local ring of prime characteristic and the connection of the SML property for flat modules with the Ohm-Rush trace property (\autoref{thm:SML-ORT}).
\end{remark}

\subsection{Openness of pure loci}

Let $v \colon M \to N$ be an $R$-linear map such that both $M$ and $N$ are finitely presented $R$-modules. Suppose $\p \in \Spec(R)$ such that the induced map $v_\p \colon M_\p \to N_\p$ splits. Since $\Hom_R(N,M)_\p = \Hom_{R_\p}(N_\p,M_\p)$, a left-inverse of $v_\p$ spreads to a distinguished open neighborhood of $\p$, that is, there exists $f \in R \setminus \p$ such that $v_f \colon M_f \to N_f$ also splits. Thus,
$
\textrm{$\{\p \in \Spec(R) \colon v_\p$ splits in $\Mod_{R_\p}\}$}   
$
is open in $\Spec(R)$. Note that in this setup, $\coker(v)$ is a finitely presented $R$-module as well by \cite[\href{https://stacks.math.columbia.edu/tag/0519}{Tag 0519} (4)]{stacks-project}. Thus,
$
\textrm{$\{\p \in \Spec(R) \colon v_\p$ is pure in $\Mod_{R_\p}\} = \{\p \in \Spec(R) \colon v_\p$ splits in $\Mod_{R_\p}\}$}   
$
by \autoref{lem:pure-iff-split}. 

\begin{notation}
    \label{not:pure-locus}
    For a map of $R$-modules $v \colon M \to N$, let $\Pure(v)$ denote the locus of primes $\p \in \Spec(R)$ where $v_\p$ is pure and we let $\CPure(v)$ denote the locus of primes $\p \in \Spec(R)$ where $v_\p$ is cyclically pure.
\end{notation}

We have the following:

\begin{lemma}
    \label{lem:stabilizers-pure-loci}
    Let $R$ be a ring and $v \colon M \to N$ be a map of $R$-modules where $M$ is finitely presented. If $v$ admits a stabilizer, then $\Pure(v)$ is open in $\Spec(R)$.
\end{lemma}

\begin{proof}
By hypothesis, there exists a finitely presented $R$-module $P$ and a map $u \colon M \to P$ such that $u$ and $v$  dominate each other. Then 
$\Pure(v) = \Pure(u)$
by \autoref{cor:mutual-domination-pure-loci}. Since $\Pure(u)$ is open by the argument above, we have the desired result.
\end{proof}

\begin{corollary}
\label{cor:ML-open-pure-loci}
Let $R$ be a ring and $M$ be a ML $R$-module. Let $P$ be a finitely presented $R$-module and $v \colon P \to M$ be a linear map. Then $\Pure(v)$ is open.
\end{corollary}

\begin{proof}
    Since $M$ is ML, $v$ admits a stabilizer. So we are done by \autoref{lem:stabilizers-pure-loci}.
\end{proof}

If $M$ is a flat Ohm-Rush $R$-module, then any linear map $v \colon R \to M$ admits a stabilizer by \autoref{cor:OR-equivalences-addition}. Thus $\Pure(v)$ is open for such maps $v$. However, we can say more about the pure locus in terms of the content function.

\begin{lemma}
\label{lem:purity-open-locus}
    Let $R$ be a ring and $M$ be a flat $R$-module. Let $v \colon R \to M$ be a linear map. Then we have the following:
    \begin{enumerate}
        \item[$(1)$] $v$ is pure if and only if $c_M(v(1)) = R$.
        \item[$(2)$] If $M$ is Ohm-Rush, then $\Pure(v) = \CPure(v) = \Spec(R) \setminus \mathbf{V}(c_M(v(1)))$.  
\end{enumerate}
\end{lemma}
    
\begin{proof}
(1) Suppose $v$ is pure. Then for all ideals $I$ of $R$, $v \otimes_R R/I$ is injective. Since the image of $1$ in $R/I$ is a nonzero element if $I \subsetneq R$, this mean that for all proper ideals $I$ of $R$, $v(1) + I \neq 0$ in $M/IM$. But that is equivalent to saying that for all proper ideals $I$ of $R$, $v(1) \notin IM$. Then $c_M(v(1)) = R$ by definition of content. Note this implication does not need $M$ to be flat.
    
Conversely, suppose $c_M(v(1)) = R$. In order to show that $v$ is pure, it suffices to show by \autoref{lem:cyclic-purity-flat} that $v$ is cyclically pure because $M$ is a flat $R$-module. Let $I$ be an ideal of $R$ and suppose $r \in R$ such that $r + I \in \ker(v \otimes_R R/I)$. Then $v(r) \in IM$, and so, $rv(1) \in IM$. This means that
    $
    c_M(rv(1)) \subseteq I.
    $
    
Since $M$ is $R$-flat, by \autoref{lem:factorcontent},
    $
    rc_M(v(1)) \subseteq c_M(rv(1)) \subseteq I.
    $
Hence $c_M(v(1)) \subseteq (I \colon r)$. As $c_M(v(1)) = R$, we get $r \in I$, which shows that $v \otimes_R R/I$ is injective.

(2) Since $M$ is flat we have $\Pure(v) = \CPure(v)$ by \autoref{lem:cyclic-purity-flat}. Since $M$ is also Ohm-Rush, by \autoref{prop:content-homomorphisms-purity-flatness} (4) we have that for all $x \in M$ and for all $\p \in \Spec(R)$,
$
c_M(x)R_\p = c_{M_\p}(x/1),
$
where $c_{M_\p}$ is the content of $M_\p$ as an $R_\p$-module. By (1), $v_\p$ is $R_\p$-pure if and only if 
$
c_M(v(1))R_\p = c_{M_\bp}(v_\p(1/1)) = R_\p.
$
This precisely means that
$
\Pure(v) = \{\p \in \Spec(R) \colon c_M(v(1)) \nsubseteq \p\},$
as desired.
\end{proof}

\subsection{Local-to-global criteria}
We have seen that a lot more can be said about the ML and Ohm-Rush properties when working over a local ring; see for instance \autoref{prop:RG-2.1.9}, \autoref{prop:ML-SML}, \autoref{prop:analog-RG219}, \autoref{prop:localpuren}, \autoref{cor:ML-OR-local-difference} and \autoref{rem:constructing-free-submodule-pure-OR}. This makes it desirable to have some local-to-global results that allow us to deduce that a module is ML or Ohm-Rush if all its localizations at primes are. This is the content of the main result of this subsection.

\begin{theorem}
\label{thm:local-to-global-ML-OR}
Let $R$ be a ring and let $M$ be a flat $R$-module. Consider the conditions:
\begin{itemize}
        \item[($\dagger'$)] For all finite free $R$-modules $F$ and \emph{injective} linear maps $\varphi \colon F \hookrightarrow M$, the set $\{\p \in \Spec(R) \colon \varphi \otimes_R \id_{R/\p} \hspace{1mm} \text{is injective}\}$ is open in $\Spec(R)$.
        \item[($\dagger$)] For all finite free $R$-modules $F$ and linear maps $\varphi \colon F \to M$, the set $\{\p \in \Spec(R) \colon \varphi \otimes_R \id_{R/\p} \hspace{1mm} \text{is injective}\}$ is open in $\Spec(R)$.
        \item[($\dagger\dagger$)] For all finite free $R$-modules $F$ and linear maps $\varphi \colon F \to M$, the cyclically pure locus of $\varphi$ is open in $\Spec(R)$.
        \item[($\dagger\dagger\dagger$)] For all finitely presented $R$-modules $P$ and linear maps $\varphi \colon P \to M$, the pure locus of $\varphi$ is open in $\Spec(R)$.
\end{itemize} 
We have $(\dagger\dagger\dagger) \implies (\dagger\dagger) \Longleftrightarrow (\dagger) \implies (\dagger')$. In addition, the following assertions hold:
\begin{enumerate}
        \item[$(1)$] Assume that for all $\p \in \Spec(R)$, $M_\p$ is a ML $R_\p$-module. Then the following are equivalent: 
        \begin{enumerate}
            \item[$(1a)$] $M$ is ML.
            \item[$(1b)$] $(\dagger\dagger)$ holds.
            \item[$(1c)$] $(\dagger\dagger\dagger)$ holds.
            \item[$(1d)$] $(\dagger)$ holds. 
        \end{enumerate}
        Moreover, if $R$ is a domain then $(1a)-(1d)$ is equivalent to:
        \begin{enumerate}
            \item[$(1d')$] $(\dagger')$ holds. 
        \end{enumerate}
        \item[$(2)$] Assume that for all $\p \in \Spec(R)$, $M_\p$ is an Ohm-Rush $R_\p$-module. 
        \begin{enumerate}
            \item[$2(i)$] If $(\dagger\dagger)$ (equivalently $(\dagger)$) holds, then $M$ is Ohm-Rush.
            \item[$2(ii)$] If $R$ is a domain and $(\dagger')$ holds, then $M$ is Ohm-Rush.
        \end{enumerate}
\end{enumerate}
\end{theorem}

Part (1) of \autoref{thm:local-to-global-ML-OR} strengthens \cite[Part II, Lem.\ 2.5.6]{rg71}. Furthermore, \autoref{thm:local-to-global-ML-OR} (2) is new.

\begin{proof}
Since $M$ is flat, $\CPure(\varphi) = \Pure(\varphi)$ by \autoref{lem:cyclic-purity-flat}. Moreover, if $F$ is finite free (hence finitely presented and flat), then $\{\p \in \Spec(R) \colon \varphi \otimes_R \id_{R/\p} \hspace{1mm} \text{is injective}\} = \Pure(\varphi)$ by \autoref{cor:cyclic-purity-radical-ideals}.
    Hence $(\dagger) \Longleftrightarrow (\dagger\dagger)$. Moreover, $(\dagger\dagger\dagger) \implies (\dagger\dagger)$ and $(\dagger) \implies (\dagger')$ are clear.

    $(1a) \implies (1c)$ follows by \autoref{cor:ML-open-pure-loci}, and we already saw that $(1c) \implies (1b) \Longleftrightarrow (1d)$ regardless of whether $M$ is locally ML.

    We will finish the proof by demonstrating $(1b) \implies (1a)$ and $2(i)$ simultaneously because the proofs are similar.

    So suppose for all $\p \in \Spec(R)$, $M_\p$ is ML (resp. Ohm-Rush) in $\Mod_{R_\p}$. Let $P$ be a finitely presented $R$-module  (resp. let $P = R$) and let $v \colon P \to M$ be a linear map. It suffices to show that $v$ admits a stabilizer ($M$ will be Ohm-Rush in $(2)$ when $P = R$ by \autoref{cor:OR-equivalences-addition}). 
    
    By \autoref{cor:ML-OR-local-difference}, in either case there exists a submodule $N(\p)$ of $M_\p$ such that 
    $
    \im(v)_\p = \im(v_\p) \subseteq N(\p) \subseteq M_\p,    
    $
    $N(\p)$ is free of finite rank and $N(\p) \hookrightarrow M_\p$ is pure in $\Mod_{R_\p}$. Clearing denominators, if necessary, there exist $m_1,\dots,m_n \in M$ such that $m_1/1,\dots,m_n/1$ form a free basis of $N(\p)$ in $\Mod_{R_\p}$. Consider the $R$-linear map
    \begin{equation}
        \label{eq:spreading-purity}
    \varphi(\p) \colon R^{\oplus n} \to M    
    \end{equation}
    that sends the standard basis vector $e_i$ of $R^{\oplus n}$ to $m_i$ for all $i = 1, \dots, n$. By construction, $\varphi(\p)_\p$ is an isomorphism onto its image $N(\p)$, and hence, $\varphi(\p)_\p$ is $R_\p$-pure. Note that by $(\dagger\dagger)$,
    $
    \Pure(\varphi(\p)) = \CPure(\varphi(\p))   
    $
    is open. Since $\p \in \Pure(\varphi(\p))$, one can choose $f \in R \setminus \p$ such that $\varphi(\p)_f$ is pure in $\Mod_{R_f}$. Furthermore, since $\im(v)$ is finitely generated (it is the image of the finitely presented module $P$) and 
    $
    \im(v)_\p = \im(v_\p) \subseteq N(\p) = \im(\varphi(\p))_\p.    
    $
    Replacing $f$ by a multiple in $R \setminus \p$, one may further assume that
    $
    \im(v_f) = \im(v)_f \subseteq \im(\varphi(\p))_f  = \im(\varphi(\p)_f).   
    $
    By purity, $\varphi(\p)_f \colon R_f^{\oplus n} \to M_f$ is an isomorphism onto its image. Let
    $\phi \colon \im(\varphi(\p)_f) \to R^{\oplus n}_f$ be the inverse map. Then we have
    \[
    v_f = P_f \xrightarrowdbl{v_f} \im(v_f) \hookrightarrow \im(\varphi(\p)_f) \xrightarrow{\phi} R^{\oplus n}_f \xrightarrow{\varphi(\p)_f} M_f.    
    \]
   Taking $u_f \coloneqq P_f \xrightarrowdbl{v_f} \im(v_f) \hookrightarrow \im(\varphi(\p)_f) \xrightarrow{\phi} R^{\oplus n}_f$, the purity of $\varphi(\p)_f$ implies that $u_f$ is a stabilizer of $v_f$ in  $\Mod_{R_f}$ by \autoref{lem:domination} (4). 
    
    Thus, for all $\p \in \Spec(R)$, there exists $f \in R \setminus \p$ such that $v_f$ admits a stabilizer in $\Mod_{R_f}$. By quasicompactness, choose $f_1,\dots,f_k$ such that $(f_1,\dots,f_k) = R$ and $v_{f_i}$ admits a stabilizer for all $i$. Then the canonical ring map
    $
    \pi \colon R \to \prod_{i=1}^k R_{f_i} 
    $
    is faithfully flat, and by construction,
    $
    \id_{\pi} \otimes_R v = \prod_{i=1}^k v_{f_i}
    $
    admits a stabilizer as a map of $\prod_{i=1}^k R_{f_i}$-modules. Then $v$ admits a stabilizer by descent (\autoref{cor:descent-cyclic-stabilizers}). This finishes the proof of $(1b) \implies (1a)$ and $2(i)$. Since we have already seen that the implications $(1a) \implies (1c) \implies (1b)$ hold, we then obtain the equivalence of $(1a)-(1d)$ when $M$ is locally ML.

  It remains to show that if $R$ is a domain, then $(1d')$ is equivalent to $(1a)-(1d)$ and $2(ii)$ holds. If $R$ is a domain and $M$ is locally ML (resp. is locally Ohm-Rush), then for $v$ as above and a prime ideal $\p$, the map $\varphi(\p) \colon R^n \to M$ in \autoref{eq:spreading-purity} is injective. This follows by the commutativity of the diagram
    \[
        \begin{tikzcd}
            R^{\oplus n} \arrow[r, "\varphi(\p)"] \arrow[d]
              & M \arrow[d] \\
            R_\p^{\oplus n} \arrow[r, "\varphi(\p)_\p"]
          & M_\p \end{tikzcd}
          \]
            because the bottom horizontal map $\varphi(\p)_\p$ is injective by purity and the left horizontal map $R^{\oplus n} \to (R^{\oplus n})_\p = R_\p^{\oplus n}$ is injective because it is a direct sum of the canonical injection $R \hookrightarrow R_\p$. Again, note that the  condition $(\dagger')$ for the injective map $\varphi(\p)$ is equivalent to $\Pure(\varphi(\p))$ being open in $\Spec(R)$ by \autoref{cor:cyclic-purity-radical-ideals} (1). Then one just repeats the argument above to see $v$ admits a stabilizer, thereby establishing $(1d') \implies (1a)$ and $2(ii)$. 
\end{proof}

Recall that a domain $R$ is called a \emph{Pr\"ufer domain} if for all maximal ideals $\fm$ of $R$, $R_\fm$ is a valuation ring. For example, Dedekind domains are precisely the Noetherian Pr\"ufer domains. For Pr\"ufer domains, one has the following necessary and sufficient local-to-global statement for the Ohm-Rush property.

\begin{proposition}
\label{prop:local-to-global-Prufer}
Let $R$ be a Pr\"ufer domain and $M$ be a torsion-free $R$-module. Suppose that for all prime ideals $\p$ of $R$, $M_\p$ is an Ohm-Rush $R_\p$-module. Then $M$ is Ohm-Rush if and only if every injective $R$-linear map $\varphi \colon R \to M$ has open cyclically pure locus.
\end{proposition}

\begin{proof}
    Over a Pr\"ufer domain, being torsion-free is equivalent to being flat. This follows after reducing to the local case by \cite[\href{https://stacks.math.columbia.edu/tag/0539}{Tag 0539}]{stacks-project} since torsion-freeness is equivalent to flatness for modules over valuation rings.  Thus, $M$ is a flat $R$-module. The `only if' implication then follows by \autoref{lem:purity-open-locus}. 
    
    Now assume that for any injective $R$-linear map $\varphi \colon R \to M$, $\CPure(\varphi)$ is open. Again, by flatness of $M$, $\CPure(\varphi) = \Pure(\varphi)$, so the pure locus of $\varphi$ is open as well. 
    
    Let $v \colon R \to M$ be a linear map. 
    It suffices to show by \autoref{cor:OR-equivalences-addition} that $v$ admits a stabilizer. Since $M$ is torsion-free over the domain $R$, $v$ is either the zero map or $v$ is injective. If $v$ is the zero map, then $R \to 0$ is a stabilizer of $v$.  
    
    So suppose $v$ is injective. Let $f \coloneqq v(1)$. Then $f \neq 0$ because $v$ is an injection. Thus for all prime ideals $\p$, $f \neq 0$ in $M_\p$ because $M$ is torsion-free. Hence, $c_{M_\p}(f) \neq 0$ because $f \in c_{M_\p}(f)M_\p$. Furthermore, since $c_{M_\p}(f)$ is a finitely generated ideal of the valuation ring $R_\p$ by \autoref{rem:content-finitely-generated}, the minimal number of generators of $c_{M_\p}(f)$ equals $1$. Then $f$, and hence $\im(v_\p)$, is contained in a rank $1$ free $R_\p$-submodule $N(\p)$ of $M_\p$ such that $N(\p) \hookrightarrow M_\p$ is pure by \autoref{rem:constructing-free-submodule-pure-OR}. We can then construct an an $R$-linear map $\varphi(\p) \colon R \to M$ 
    as in the proof of \autoref{thm:local-to-global-ML-OR} part 2$(i)$ such that the image of the localization $\varphi(\p)_\p$ equals $N(\p)$, that is, $\varphi(\p)_\p$ is $R_\p$-pure. Moreover, $\varphi(\p)$ will be injective following the reasoning in the proof of \autoref{thm:local-to-global-ML-OR} part 2$(ii)$. Using the openness of the pure locus of $\varphi(\p)$, one then obtains a stabilizer for $v_f$ for some $f \notin \p$. By quasicompactness and descent, we again get a stabilizer for $v$. The interested reader can fill in the details following the proof of \autoref{thm:local-to-global-ML-OR} part 2$(i)$.
\end{proof}

The next result shows that if $M$ is a flat and locally Ohm-Rush $R$-module, then $M$ globally satisfies the defining condition of the Ohm-Rush property for \emph{radical ideals} assuming only the openness of pure loci of maps $R \to M$. In other words, the characterization of \autoref{prop:local-to-global-Prufer} almost holds for flat modules over an arbitrary ring.

\begin{proposition}
    \label{prop:close-to-Ohm-Rush}
    Let $R$ be a ring and $M$ be a flat $R$-module. 
    Consider the following statements:
    \begin{enumerate}
        \item[$(1)$] For all $R$-linear maps
        $v \colon R \to M$, $\Pure(v)$ is open in $\Spec(R)$.

        \item[$(2)$] For all $x \in M$, the collection of ideals 
        $
        \textrm{$\{I \colon I = \sqrt{I}$ and $x \in IM\}$}
        $
        has a smallest element under inclusion. Equivalently, for any collection of radical 
        ideals $\{I_\alpha\}_\alpha$ of $R$,
        $
        \bigcap_\alpha I_\alpha M = \left(\bigcap_\alpha I_\alpha\right)M.
        $
    \end{enumerate}
Then $(2) \implies (1)$. Moreover, if $M_\p$ is an Ohm-Rush $R_\p$-module for all $\p \in \Spec(R)$, then $(1) \implies (2)$.
    \end{proposition}
    
\begin{proof}
    $(2) \implies (1)$: Let $v \colon R \to M$ be a linear map. Let $x \coloneqq v(1)$, and let $\scr{I}$ be the intersection of all radical ideals $I$ of $R$ such that $x \in IM$. By the hypothesis of (2), $x \in \scr{I}M$. We will show that 
$
\Pure(\varphi) = \Spec(R) \setminus \mathbf{V}(\scr{I}).    $

Suppose $\p \in \Pure(\varphi)$. Since $v_\p$ is pure, the map $v \otimes_R \id_{R/\p} \colon R/\p \to M/\p M$ is injective by the assertion of equalityof sets in \autoref{cor:cyclic-purity-radical-ideals} (1) because the domain of $v$, namely $R$, is flat and finitely presented. This shows that $x \notin \p M$.
Since $x \in \scr{I}M$, it follows that $\p \in \Spec(R) \setminus \mathbf{V}(\scr{I})$. Thus, $\Pure(\varphi) \subseteq \Spec(R) \setminus \mathbf{V}(\scr{I})$.

Suppose $\p \in \Spec(R) \setminus \mathbf{V}(\scr{I})$. Then $x \notin \p M$ by definition of $\scr{I}$. Since $M/\p M$ is a flat module over the domain $R/\p$ by base change, we see that $M/\p M$ is a torsion-free $R/\p$-module. Thus, the induced $R/\p$-linear map $v \otimes_R \id_{R/\p} \colon R/\p \to M/\p M$ is injective because it maps $1 + \p$ to the nonzero (and hence torsion-free) element $x + \p M$. Then $v_\p$ is $R_\p$-pure by the inclusion of sets in \autoref{cor:cyclic-purity-radical-ideals} (1). So, $\Spec(R) \setminus \mathbf{V}(\scr{I}) \subseteq \Pure(\varphi)$, establishing the other inclusion.

    $(1) \implies (2)$: We will now assume $M$ is locally Ohm-Rush, that is, for all $\p \in \Spec(R)$, $M_\p$ is an Ohm-Rush $R_\p$-module. Consider the map $v \colon R \to M$ that sends $1 \mapsto x$. Let $\mathfrak{I}_x$ denote the unique radical ideal of $R$ such that $\Spec(R) \setminus \mathbf{V}(\mathfrak{I}_x)$ defines the pure locus of $v$. We claim that $\mathfrak{I}_x$ is the smallest radical ideal $I$ of $R$ with the property that $x \in IM$.
    
    Let $\bp \in \Spec(R)$. We claim that
    $
    \sqrt{c_{M_\bp}(x/1)} = \mathfrak{I}_x R_\bp.
    $
    Since $\mathfrak{I}_x R_\bp$ is also a radical ideal, in order to establish the above equality, it suffices to show that
    \begin{equation}
    \label{eq:radical}
    \textrm{$\{Q \in \Spec(R_\bp) \colon (v_\bp)_Q$ is $(R_\bp)_Q$ pure$\}$} = \Spec(R_\bp) \setminus \mathbf{V}(\mathfrak{I}_x R_\bp).
    \end{equation}
    This is because  $M_\bp$ is a flat Ohm-Rush $R_\bp$-module and \autoref{lem:purity-open-locus} shows that
    \[
    \textrm{$\{Q \in \Spec(R_\bp) \colon (v_\bp)_Q$ is $(R_\bp)_Q$ pure$\}$} = 
    \Spec(R_\bp) \setminus \mathbf{V}\left(\sqrt{c_{M_\bp}(x/1)}\right).
    \]
    Any $Q \in \Spec(R_\bp)$ is the expansion of a unique $\bq \in \Spec(R)$ such that
    $\bq\subseteq \bp$. Moreover, $(R_\bp)_Q$ can then be identified with $R_\bq$ and $(v_\bp)_Q$ can be identified with the $R_\bq$-linear map $v_\bq$. Thus,
    \[
    \textrm{$(v_\bp)_Q$ is $(R_\bp)_Q$-pure $\Longleftrightarrow v_\bq$ is $R_\bq$-pure 
    $\Longleftrightarrow \mathfrak{I}_x \nsubseteq \bq \Longleftrightarrow \mathfrak{I}_x R_\bp \nsubseteq \bq R_\bp = Q$,} 
    \]
    which establishes \autoref{eq:radical}. Here the second equivalence follows by the definition of $\mathfrak{I}_x$. In the third equivalence, the non-trivial implication is $\mathfrak{I}_x \nsubseteq \bq \implies \mathfrak{I}_xR_\p \nsubseteq \bq R_\p = Q$. But this follows because otherwise, for all $i \in \mathfrak{I}_x$, there would exist $s \in R \setminus \p \subseteq R \setminus \bq$ such that $si \in \bq$. Since $\bq$ is prime, we would get $i \in \bq$, whence, $\mathfrak{I}_x \subseteq \bq$, which is a contradiction.
    
    Thus for all prime ideals $\bp$ of $R$ because $M_\p$ is an Ohm-Rush $R_\p$-module, we have
    $
    x/1 \in c_{M_\bp}(x/1)M_\bp \subseteq \sqrt{c_{M_\bp}(x/1)}M_\bp = (\mathfrak{I}_x R_\bp) M_\bp = (\mathfrak{I}_x M)_\bp,
    $
    and so, $x \in \mathfrak{I}_xM$. Moreover, if $I$ is a radical ideal of $R$ such that $x \in IM$, then for all
    prime ideals $\bp$ of $R$, $c_{M_\bp}(x/1) \subseteq IR_\bp$, and so, $\mathfrak{I}_x R_\bp = \sqrt{c_{M_\bp}(x/1)}$ is contained in the radical ideal $IR_\p$. Then $\mathfrak{I}_x \subseteq I$ because the inclusion $\mathfrak{I}_x \subset I$ can be checked locally. Hence $\mathfrak{I}_x$ is indeed the smallest radical ideal $I$ of $R$
    with the property that $x \in IM$.\qedhere
    \end{proof}

\subsection{Radical content}    
Let $R$ be a ring, $M$ be an $R$-module and $x \in M$. Let us call $c_{\rad,M}(x)$ the intersection of all \emph{radical ideals} $I$ of $R$ such that $x \in IM$. More generally, for $N \subseteq M$, let $c_{\rad,M}(N)$ be the intersection of all radical ideals $I$ of $R$ such that $N \subseteq IM$. Clearly, if $\langle N \rangle$ is the $R$-submodule of $M$ generated by $N$, then $c_{\rad,M}(N) = c_{\rad, M}(\langle N \rangle)$. We call the function
$
\textrm{$c_{\rad,M} \colon M \to \{$radical ideals of $R\}$}
$
the \emph{radical content of $M$.} Clearly, $\sqrt{c_M(N)} \subseteq c_{\rad,M}(N)$.

Suppose $M$ is a flat $R$-module such that for any collection $\{I_\alpha \colon \alpha \in A\}$ of radical ideals of $R$, 
    \begin{equation}
        \label{eq:radical-content}
    (\bigcap_{\alpha} I_\alpha)M = \bigcap_{\alpha \in A} I_\alpha M.
    \end{equation} 
    This is equivalent to the assertion that for any $N \subseteq M$, $N \subseteq c_{\rad,M}(N)M$. Equivalently, the collection of radical ideals $I$ of $R$ such that $N \subseteq IM$ has a smallest element under inclusion. The following shows that $c_{\rad, M}$ behaves like the content function $c_M$ when \autoref{eq:radical-content} is satisfied.

    \begin{proposition}
        \label{prop:radical-content-properties}
        Let $R$ be a ring and $M$ be a flat $R$-module that satisfies condition \autoref{eq:radical-content}. Let $N \subseteq M$ be a submodule.
        \begin{enumerate}
            \item[$(1)$]  Let $P$ be any $R$-module and $Q$ be a submodule such that $Q \subseteq c_P(Q)P$. Then $c_{\rad,P}(Q) = \sqrt{c_P(Q)}$.
            
            \item[$(2)$] If $S \subset R$ is multiplicative, then for all $x \in M, s \in S$, $x/s \in c_{\rad, S^{-1}M}(x/s)(S^{-1}M)$ and $c_{\rad,M}(x)(S^{-1}R) = c_{\rad, S^{-1}M}(x/s)$.
            
            \item[$(3)$] If $S \subset R$ is multiplicative, then $c_{\rad,M}(N)(S^{-1}R) = c_{\rad, S^{-1}M}(S^{-1}N)$. 
            
            \item[$(4)$] If $M$ is locally Ohm-Rush, then for all $\p \in \Spec(R)$, $c_{\rad, M}(N)R_\p = \sqrt{c_{M_\p}(N_\p)}$.
            
            \item[$(5)$] Suppose $R$ is Noetherian and satisfies the property that for all $\p \in \Spec(R)$ the fibers of $R_\p \to \widehat{R_\p}$ are reduced. If $M$ is locally Ohm-Rush, $\widehat{M_\p}$ is the $\p R_\p$-adic completion of $M$ and $\widehat{R_\p}N$ denotes the $\widehat{R_\p}$-submodule of $\widehat{M_\p}$ generated by the image of $N$ under the canonical $R$-linear map $M \to M_\p \to \widehat{M_\p}$, then $\widehat{R_\p}N \subseteq c_{\rad, \widehat{M_\p}}(\widehat{R_\p}N)\widehat{M_\p}$ and  $c_{\rad, M}(N)\widehat{R_\p} = c_{\rad, \widehat{M_\p}}(\widehat{R_\p}N)$.
        \end{enumerate}
    \end{proposition}

    \begin{proof} 
    (1) is clear. For (2) let $\pi \colon R \to S^{-1}R$ be the canonical map. Recall that if $J$ is a radical ideal of $S^{-1}R$, then $\pi^{-1}(J)$ is a radical ideal of $R$ such that $\pi^{-1}(J)S^{-1}R = J$. Let $\{J_\alpha \colon \alpha \in A\}$ be a collection of radical ideals of $S^{-1}R$. We want to show that $\bigcap_{\alpha \in A} J_\alpha(S^{-1}M) = (\bigcap_{\alpha \in A} J_\alpha)S^{-1}M$. It is enough to show that if $p \colon M \to S^{-1}M$ is the canonical map, then $p^{-1}(\bigcap_\alpha J_\alpha(S^{-1}M)) = p^{-1}((\bigcap_{\alpha} J_\alpha)S^{-1}M)$. Let $I_\alpha \coloneqq \pi^{-1}(J_\alpha)$. By the flatness of $M$, we get an injective map $M/I_\alpha M \hookrightarrow S^{-1}M/J_\alpha (S^{-1}M)$, that is, for all $\alpha \in A$, $p^{-1}(J_\alpha (S^{-1}M)) = I_\alpha M$. Similarly, $p^{-1}((\bigcap_{\alpha} J_\alpha)S^{-1}M) = (\bigcap_{\alpha}I_\alpha)M$. Thus,
        \[
        p^{-1}\left(\bigcap_{\alpha \in A} J_\alpha (S^{-1}M)\right) = \bigcap_{\alpha \in A} I_\alpha M = (\bigcap_{\alpha \in A}I_\alpha)M = p^{-1}((\bigcap_{\alpha \in A}J_\alpha)S^{-1}M),
        \]
        that is, the $S^{-1}R$-module $S^{-1}M$ satisfies the property that for any collection of radical ideals $\{J_\alpha \colon \alpha \in A\}$ of $S^{-1}R$, 
        $
            \bigcap_{\alpha \in A} J_\alpha(S^{-1}M) = (\bigcap_{\alpha \in A}J_\alpha)S^{-1}M.    
        $
        Let $x \in M, s \in S$.
        Since $x \in c_{\rad,M}(x)M$ and since radical ideals of $R$ expand to radical ideals of $S^{-1}R$, we get $c_{\rad, S^{-1}M}(x/s) \subseteq c_{\rad, M}(x)S^{-1}R$. Also, since $\ba \coloneqq \pi^{-1}(c_{\rad,S^{-1}M}(x/s))$ is a radical ideal of $R$ and since $x/s \in c_{\rad, S^{-1}M}(x/s)(S^{-1}M)$, we get  $x \in \ba M$ by flatness of $M$. So, $c_{\rad,M}(x) \subseteq \ba$, and consequently, $c_{\rad,M}(x)(S^{-1}R) \subseteq \ba(S^{-1}R) = c_{\rad,S^{-1}M}(x/s)$. This shows that for all $x \in M$, $s \in S$,
        $
        c_{\rad,M}(x)(S^{-1}R) = c_{\rad, S^{-1}M}(x/s).    
        $
        In other words, $c_{\rad,M}$ commutes with localization.

        (3) One can check that for a submodule $N$ of $M$, $c_{\rad, M}(N) =  \sqrt{\sum_{x \in N} c_{\rad, M}(x)}$.
        Moreover, since for any ideal $I$ of $R$, $\sqrt{I}(S^{-1}R) = \sqrt{I(S^{-1}R)}$, we get
        \begin{align*}
        c_{\rad, M}(N)(S^{-1}R) &= \left(\sqrt{\sum_{x \in N} c_{\rad,M}(x)}\right) (S^{-1}R)  
        = \sqrt{\sum_{x \in N} c_{\rad, M}(x)(S^{-1}R)}\\
        &\stackrel{(2)}{=}  \sqrt{\sum_{x \in N,s\in S}c_{\rad, S^{-1}M}(x/s)}
        =c_{\rad, S^{-1}M}(S^{-1}N). 
        \end{align*}
        The second equality follows because localization commutes with arbitrary sums of ideals.
        
       (4) If $M$ is locally Ohm-Rush, then for all submodules $N$ of $M$ and prime ideals $\p$ of $R$,
        \[
        c_{\rad, M}(N)R_\p \stackrel{(3)}{=} c_{\rad, M_\p}(N_\p) \stackrel{(1)}{=} \sqrt{c_{M_\p}(N_\p)}.  
        \]

        (5) Note that $\widehat{R_\p}N$ is also the $\widehat{R_\p}$-submodule of $\widehat{M_\p}$ that is generated by the image of $N_\p$ under the canonical map $M_\p \to \widehat{M_\p}$. Hence, since $M_\p$ is an Ohm-Rush $R_\p$-module, \autoref{cor:content-completion} shows that
        $c_{M_\p}(N_\p) \widehat{R_\p} = c_{\widehat{M_\p}}(\widehat{R_\p}N)$ and $\widehat{R_\p}N \subseteq c_{\widehat{M_\p}}(\widehat{R_\p}N)\widehat{M_\p}$.
        
        Since $c_{\widehat{M_\p}}(\widehat{R_\p}N) \subseteq c_{\rad, \widehat{M_\p}}(\widehat{R_\p}N)$, we get $\widehat{R_\p}N \subseteq c_{\rad, \widehat{M_\p}}(\widehat{R_\p}N)\widehat{M_\p}$. 
        Moreover, since the formal fibers of $R_\p \to \widehat{R_\p}$ are reduced, by \autoref{cor:expansion-radical-ideals},$\sqrt{c_{M_\p}(N_\p)}\widehat{R_\p} = \sqrt{c_{M_\p}(N_\p)\widehat{R_\p}} = \sqrt{c_{\widehat{M_\p}}(\widehat{R_\p}N)}$.    
        Also, by (4) we get $c_{\rad, M}(N)R_\p = \sqrt{c_{M_\p}(N_\p)}$. Thus,
        $
            c_{\rad, M}(N)\widehat{R_\p} =   (c_{\rad, M}(N)R_\p)\widehat{R_\p} = \sqrt{c_{M_\p}(N_\p)}\widehat{R_\p} = \sqrt{c_{\widehat{M_\p}}(\widehat{R_\p}N)}.
        $
        Finally, $\sqrt{c_{\widehat{M_\p}}(\widehat{R_\p}N)} = c_{\rad, \widehat{M_\p}}(\widehat{R_\p}N)$ by (1) because $\widehat{R_\p}N \subseteq c_{\widehat{M_\p}}(\widehat{R_\p}N)\widehat{M_\p}$.
    \end{proof}

\section{Ohm-Rush trace, intersection flatness and Mittag-Leffler modules}
\label{sec:ORT-IF-ML}

In this section we introduce the notions of Ohm-Rush trace modules and intersection flatness. We will then build connections between these new notions, Mittag-Leffler modules, strictly Mittag-Leffler modules and Ohm-Rush modules.

\subsection{Ohm-Rush trace modules} Suppose that $R$ is a 
   ring. For an $R$-module $M$ and $x \in M$,
    let
    $
        \Tr_M(x) \coloneqq \{f(x) \colon f \in \Hom_R(M,R)\},
    $ 
    that is, $\Tr_M(x)$ is the image of the evaluation at $x$ map
    $\Hom_R(M,R) \to R$. Thus, $\Tr_M(x)$ is an ideal of
    $R$, often called the \emph{trace of $x$}. 
    Similarly, if $N \subseteq M$, we define $\Tr_M(N) \coloneqq \sum_{x \in N} \Tr_M(x)$. If $N$ is an $R$-submodule of $M$, then it follows that 
    \begin{equation}
        \label{eq:trace-submodule}
    \Tr_M(N) = \sum_{f \in \Hom_R(M,R)} f(N).
\end{equation}

    The following notion codifies the idea that the $R$-dual of a module $M$ has `sufficiently many' maps.

\begin{definition}
\label{def:Ohm-Rush-trace}
    For a ring $R$ and an $R$-module $M$,
    we say $M$ is an
    \emph{Ohm-Rush trace} (abbrv. ORT) $R$-module if for all $x \in M$,
$x \in \Tr_M(x)M$. A ring homomorphism $R \to S$ is \emph{Ohm-Rush trace} 
    if $S$ an ORT $R$-module.
\end{definition}

What we call an Ohm-Rush trace module is called a \emph{trace module} in 
\cite[Section 7, Pg.\ 66]{OhmRu-content}. Again, our alternate terminology honors the contributions of Ohm and Rush.

\begin{remark}
\label{rem:ORT-modules}
Suppose $M$ is an ORT $R$-module.
\begin{enumerate}
    \item Let $x \in M$ and $I$ be an ideal of $R$ such that $x \in IM$. By $R$-linearity,
    for any $f \in \Hom_R(M,R)$, $f(x) \in f(IM) \subseteq If(M) \subseteq I$. Thus $\Tr_M(x)$ is the smallest ideal
     $I$ of $R$ such that $x \in IM$, that is, 
     $\Tr_M(x) = c_M(x).$ 
    This shows that an
    Ohm-Rush trace (ORT) $R$-module is an Ohm-Rush $R$-module. We will continue using the notation $\Tr_M(x)$ even though this ideal coincides with $c_M(x)$. This is to emphasize the origin of $c_M(x)$ via maps $M \to R$ when $M$ is ORT. Note also that since $M$ is an Ohm-Rush module, for any $N \subseteq M$, 
    $
    \Tr_M(N) \coloneqq \sum_{x \in N} \Tr_M(x) = \sum_{x \in N} c_M(x) c_M(N),    
    $
    where $c_M(N)$ is the intersection of all ideals $I$ of $R$ such that $N \subseteq IM$. Here the last equality follows by \autoref{lem:content-subset}. By the same lemma, $N \subseteq c_M(N)M = \Tr_M(N)M$, that is, $\Tr_M(N)$ is the smallest ideal $I$ of $R$ such that $N \subseteq IM$. Moreover, if $\langle N \rangle$ is the submodule of $M$ generated by $N$, then $\Tr_M(\langle N \rangle) = c_M(\langle N \rangle) = c_M(N) = \Tr_M(N)$,
    where the second equality is clear by the definition of content of a subset of $M$. Thus, for any subset $N$ of $M$, $\Tr_M(N) = \sum_{f \in \Hom_R(M,R)} f(\langle N \rangle)$ by \autoref{eq:trace-submodule}.

    \item For any $x \in M$, $\Tr_M(x)$ is always
    a finitely generated ideal of $R$. This follows by (a) because $c_M(x)$ is always a finitely generated ideal of $R$ by \autoref{rem:content-finitely-generated}.

    \item If $x \in M$ is a nonzero element, then
    $\Tr_M(x) \neq 0$ because $x \in \Tr_M(x)M$. Said differently, there exists $f \in \Hom_R(M,R)$ such 
    that $f(x) \neq 0$, that is, the canonical map
    $
    M \to \Hom_R(\Hom_R(M,R), R)
    $
    is injective (i.e., $M$ is \emph{torsionless}). In \autoref{lem:ORT-cyclic-purity} we will have more to say about this map.

    \item $M$ is torsion-free. Indeed, if 
    $x \in M -\{0\}$, then upon choosing $f \in \Hom_R(M,R)$ such that $f(x) \neq 0$, one sees that if $r \in R$ is a nonzerodivisor such that $rx = 0$, then $rf(x) = f(rx) = 0$, which contradicts $r$ being a nonzerodivisor on $R$.

    \item If $R$ is a domain and $R \to S$ is an ORT ring map, then $R \to S$ is injective. Otherwise, $1 \in S$ is an $R$-torsion element, contradicting (d).

    \item An ORT ring map $R \to S$ splits provided that for every maximal ideal $\fm$ of $R$, 
    $\fm S \neq S$ (for instance, if $\Spec(S) \to \Spec(R)$ is surjective). Indeed, since $1 \in \Tr_S(1)S$, this means that $\Tr_S(1)$ cannot be contained in any maximal ideal of $R$. Consequently $\Tr_S(1) = R$, or equivalently, that there is a splitting of $R \to S$.

    \item \cite[Sec.\ 7, Pg.\ 66]{OhmRu-content} Let $x \in M$ and $r \in R$. Since for all $f \in \Hom_R(M,R)$ we have $f(rx) = rf(x)$ by linearity, it follows by (a) that $c_M(rx) = \Tr_M(rx) = r\Tr_M(x) = rc_M(x)$. Thus, $M$ is a flat $R$-module by \autoref{prop:content-homomorphisms-purity-flatness} (3).

    \item Let $f \colon M \to N$ be an $R$-linear map. Then for $x \in M$, the composition 
    \[
    \Hom_R(N,R) \xrightarrow{\Hom_R(f,R)} \Hom_R(M,R) \xrightarrow{\ev @ x} R    
    \]
    equals $\ev @ f(x) \colon \Hom_R(N,R) \to R$. Thus, for all $x \in M$, we have $\Tr_N(f(x)) \subseteq \Tr_M(x)$.

    \item Let $R \to S$ be a ring homomorphism. Then $\Tr_R(1_S) = \sum_{x \in S} \Tr_R(x)$. Clearly, we have $\Tr_R(1_S) \subseteq \sum_{x \in S} \Tr_R(x)$. For the other inclusion, for any $x \in S$, if $\ell_x \colon S \to S$ denotes left-multiplication by $x$, then 
    \[
    \Hom_R(S,R)  \xrightarrow{\ev @ x} R = \Hom_R(S,R) \xrightarrow{\Hom_R(\ell_x,R)} \Hom_R(S,R) \xrightarrow{\ev @ 1_S} R,   
    \]
    and so, $\Tr_R(x) = \im(\ev @ x) \subseteq \im(\ev @ 1_S) = \Tr_R(1_s)$.
\end{enumerate}
\end{remark}

\begin{lemma}
    \label{lem:ORT-cyclic-purity}
    Let $R$ be a ring and $M$ be an ORT $R$-module. Set $M^{**} := \Hom_R(\Hom_R(M,R),R)$. Then the canonical
    map 
    $
    M \to M^{**}$
    is cyclically pure as a map of $R$-modules. Moreover, if $(R, \fm)$ is
    a Noetherian local ring that is $\fm$-adically complete, then the 
    canonical map is a pure map of $R$-modules.
\end{lemma}

\begin{proof}
    Suppose $M$ is an ORT $R$-module and $I$ is an ideal
    of $R$. 
    We need to show that for all ideals $I$ of $R$, the induced map
    $M/IM \to M^{**}/IM^{**}$ is injective. Let $x \in M$. Under the canonical map $M \to M^{**}$, the image of $x$ is $\ev @ x \colon \Hom_R(M,R) \to R$. Suppose $\ev @ x \in IM^{**}$. Then there exist
    $\varphi_1, \dots, \varphi_n \in M^{**}$ and
    $i_1, \dots, i_n \in I$ such that
    $
    \ev @ x = i_1\varphi_1 + \dots + i_n\varphi_n.
    $
    This means that for all $f \in \Hom_R(M,R)$, $f(x) = i_1\varphi_1(f) + \dots + i_n\varphi_n(f) \in I$.
    Thus, $\Tr_M(x) \subseteq I$, and since $x \in \Tr_M(x)M$, it follows that $x \in IM$. This is precisely the injectivity of $M/IM \to M^{**}/IM^{**}$.

    For the second assertion, purity of $M \to M^{**}$ is guaranteed by
    cyclic purity and \autoref{lem:cyclic-purity-flat} if $M^{**}$ is a flat $R$-module. Note that an  ORT $R$-module is always $R$-flat by \autoref{rem:ORT-modules} (g). It is shown in \cite[Part II, (2.4.3)]{rg71} that if $(R, \fm)$ is a Noetherian local ring that is $\fm$-adically complete, then the $R$-dual $N^*$ is always a flat $R$-module for \emph{any} flat $R$-module $N$ using an $\Ext$-rigidity theorem due to Jensen \cite[Thm.\ 1]{Jensenlimvanishing}. Thus, over a Noetherian complete local ring $R$, if $M$ is a ORT $R$-module, then $M^*$ and hence also $M^{**}$ are $R$-flat, completing the proof of purity of $M \to M^{**}$ using its cyclic purity.
\end{proof}

\begin{remark}
\autoref{lem:ORT-cyclic-purity} will be important for \autoref{thm:ML-SML-ORT-intersection-flat}. In fact, \autoref{thm:ML-SML-ORT-intersection-flat} shows that a partial converse of \autoref{lem:ORT-cyclic-purity} holds. Namely, if $(R, \fm)$ is a Noetherian local ring that is $\fm$-adically complete and if $M$ is a flat $R$-module, then the cyclic purity of the canonical map $M \to \Hom_R(\Hom_R(M, R), R)$ implies that $M$ is an ORT $R$-module.
\end{remark}

\begin{lemma}
\label{lem:projective-ORT}
    Let $R$ be a ring. Let $M$ an ORT $R$-module and $N$ be a submodule of $M$.
    \begin{enumerate}
        \item[$(1)$] If $N$ is cyclically pure in $M$, then for all $x \in N$, $\Tr_N(x) = \Tr_M(x)$. Moreover, then $N$ is ORT.

        \item[$(2)$] Any projective $R$-module is ORT.

        \item[$(3)$] Suppose that any linear map $v \colon R \to M$ admits a stabilizer $u \colon R \to P$ such that $u$ factors through $v$ and $P$ is ORT (equivalently, $P$ is projective). Then $M$ is ORT.
    \end{enumerate}
\end{lemma}

\begin{proof}
(1) Let $\iota \colon N \hookrightarrow M$ denote the inclusion. Then the composition
\[
\Hom_R(M,R) \xrightarrow{\Hom_{R}(\iota,R)} \Hom_R(N,R) \xrightarrow {\ev @ x} R    
\]
is also $\ev @ x$. Thus, $\Tr_M(x) \subseteq \Tr_N(x)$. Since $N$ is cyclically pure in $M$, the map $N/\Tr_M(x)N \to M/\Tr_M(x)M$ is injective. But $x \in \Tr_M(x)M$ because $M$ is an ORT $R$-module. By injectivity, it follows that $x \in \Tr_M(x)N$. Then by \autoref{rem:ORT-modules} (a), we get $\Tr_N(x) = c_N(x) \subseteq \Tr_M(x)$, and so, $\Tr_M(x) = \Tr_N(x)$. Moreover, we then also have $x \in \Tr_N(x)N$. Thus, $N$ is ORT.

(2) Since a projective module is a direct summand of a free module, by (1) it suffices to show that a free $R$-module $F$ is ORT. Let 
    $\{e_\alpha\}_{\alpha \in \sA}$ be a basis of $F$. For all
    $\alpha \in \sA$, let $\phi_\alpha \colon F \to R$ be the projection onto the $\alpha$-th factor.
    Take 
    $x \in F$ and write $x = \sum_\alpha r_\alpha \cdot e_\alpha$ uniquely in terms of the basis. Then
    $    x = \sum_\alpha r_\alpha \cdot e_\alpha = \sum_\alpha \phi_\alpha(x) \cdot e_\alpha \in \Tr_F(x)F.
    $
    That is, $F$ is ORT. Note also that $\Tr_F(x) = c_F(x) = (r_\alpha \colon \alpha \in \sA)$ by \autoref{eg:free-OR}.

    (3) Since $u \colon R \to P$ is a stabilizer for $v$, by definition $P$ is a finitely presented $R$-module. Since ORT modules are flat (\autoref{rem:ORT-modules} (g)), it follows that $P$ is ORT if and only if $P$ is projective.
    
    Since $u$ stabilizes $v$, we know $v$ factors through $u$ (\autoref{lem:stabilizers} (1)), say $v = \varphi \circ u$ for some $\varphi \colon P \to M$. By assumption, $u$ factors through $v$ that is there exists $\phi \colon M \to P$ such that  $u = \phi \circ v$. Let $x \coloneqq v(1)$ and $y \coloneqq u(1)$. Then $\textrm{$\varphi(y) = x$ and $\phi(x) = y$.}$ Thus, by \autoref{rem:ORT-modules} (h), 
    $\Tr_M(x) = \Tr_M(\varphi(y)) \subseteq \Tr_P(y)$ and $\Tr_P(y) = \Tr_P(\phi(x)) \subseteq \Tr_M(x)$.    
    That is, $\Tr_P(x) = \Tr_P(y)$. Since $P$ is ORT, $y \in \Tr_P(y)P$, and so, $x = \varphi(y) \in \Tr_P(y)M = \Tr_M(x)M$. Since every $x \in M$ is the image of $1 \in R$ under some linear $v \colon R \to M$, it follows that $M$ is ORT.
\end{proof}

\begin{remark}
    The proof of \autoref{lem:projective-ORT} shows that if $M$ is a direct summand of a free $R$-module $F$ with basis $\{e_\alpha\}_\alpha$, then for any $x \in M$, if we express $x = \sum_\alpha r_\alpha \cdot e_\alpha$, for $r_\alpha \in R$, then $\Tr_M(x)$ is the ideal generated by the $r_\alpha$.
\end{remark}

The next result discusses the behavior of the ORT property under restriction of scalars.

\begin{lemma}
\label{lem:composing-ORT}
Let $R \to S$ be a homomorphism of rings and $N$ be a $S$-module. If $S$ is an ORT $R$-module and $N$ is an ORT $S$-module, then $N$ is an ORT $R$-module.
\end{lemma}

\begin{proof}
    Let $x \in N$. Since $N$ is an ORT $S$-module, one can choose $f_1, \dots, f_n \in \Hom_S(N,S)$ and $y_1, \dots, y_n \in N$ such that $x = f_1(x)\cdot y_1 + \dots + f_n(x) \cdot y_n$. Moreover, since $S$ is an ORT $R$-module, for each $i = 1, \dots, n$, one can choose $g_{i1}, \dots, g_{im_i} \in \Hom_R(S,R)$ and $s_{i1}, \cdots, s_{im_i} \in S$ such that
$    f_i(x) = g_{i1}(f_i(x))s_{i1} + \dots + g_{im_i}(f_i(x))s_{im_i}.
$
Then $x = \sum_{i = 1}^n \sum_{j = 1}^{m_i} (g_{ij} \circ f_i) (x) \cdot (s_{ij} \cdot y_i)$. Since $g_{ij} \circ f_i \in \Hom_R(N, R)$, $N$ is an ORT $R$-module.
\end{proof}

The following is now immediate.

\begin{corollary}
\label{cor:composition-ORT-ringmaps}
If $R \xrightarrow{\phi} S \xrightarrow{\varphi} T$ are ring homomorphisms such that $\phi,\varphi$ are ORT, then $\varphi\circ\phi$ is ORT.
\end{corollary}

We next show that the ORT condition satisfies pleasing permanence properties.

\begin{proposition}
    \label{prop:ORT-indeterminate}
Let $R \to S$, $R \to T$ be ring maps and $P$ be an $R$-module. Assume that $S$, $P$ are ORT $R$-modules.
Let $x$ be an indeterminate. Then we have the following:
\begin{enumerate}
    \item[$(1)$] $P \otimes_R T$ is an ORT $T$-module. In particular, $S \otimes_R T$ is an ORT $T$-module.
    
    \item[$(2)$] $P[x] \coloneqq P \otimes_R R[x]$ is an ORT $R$-module. 
    In particular, $S[x]$ is an ORT $R[x]$-module.
    
    \item[$(3)$] If $R, S$ are Noetherian, then $S[[x]]$ is an ORT $R[[x]]$-module.
    
    \item[$(4)$] If $R, S$ are Noetherian and $I$ is an ideal of $R$, 
    then the map of $I$-adic completions $\widehat{R}^I \to \widehat{S}^{IS}$ is ORT.
\end{enumerate}
\end{proposition}

\begin{proof}
(1) Every element $f \in P \otimes_R T$ can be expressed as a finite sum of elementary tensors $f = \sum_{i = 1}^n p_i \otimes t_i$, where $p_i \in P$ and $t_i \in T$. We will use induction on $n$, the number of elementary tensors. If $n = 1$, then $f = p \otimes t$. Since $P$ is ORT, there exist $\phi_1, \dots, \phi_m \in \Hom_R(P,R)$ and $a_1,\dots,a_m \in P$ such that
$
p = \phi_1(p)\cdot a_1 + \dots + \phi_m(p)\cdot a_m.
$
Then $f = p \otimes t = \phi_1 \otimes_R \id_T(p \otimes t)\cdot(a_1 \otimes 1) + \dots + \phi_m \otimes_R \id_T(p \otimes t)\cdot(a_m \otimes 1)$. Technically, $\phi_i \otimes_R \id_T$ is a $T$-linear map from $P \otimes_R T \to R \otimes_R T$. But we can identify $R \otimes_R T$ canonically with $T$ and hence consider $\phi_i \otimes \id_T$ as a $T$-linear map $P \otimes_R T \to T$. With this identification, the base case $n=1$ follows. 

Now suppose $f =  \sum_{i = 1}^n p_i \otimes t_i$.
Choose $\phi_1, \dots, \phi_m \in \Hom_R(P,R)$ and $a_1, \dots, a_m \in P$ such that
$
p_n = \phi_1(p_n)\cdot a_1 + \dots + \phi_m(p_n)\cdot a_m.
$
Then
\begin{align*}
f - \sum_{i = 1}^m \phi_i \otimes_R \id_T(f) \cdot (a_i \otimes 1) &=
f - \sum_{i = 1}^m\bigg{[}\phi_i \otimes_R \id_T\bigg{(}\sum_{j=1}^n p_j \otimes t_j\bigg{)} \cdot (a_i \otimes 1)\bigg{]}\\
&= f - \sum_{i=1}^m\bigg{(}\sum_{j=1}^n \phi_i(p_j) \otimes t_j\bigg{)}\cdot (a_i \otimes 1)\\
&= \bigg{(}\sum_{j=1}^n p_j \otimes t_j\bigg{)}  - \sum_{j=1}^n\bigg{(}\sum_{i=1}^m \phi_i(p_j)\cdot a_i\bigg{)} \otimes t_j\\
&= \sum_{j=1}^{n-1} \bigg{(}p_j - \sum_{i=1}^m \phi_i(p_j) \cdot a_i\bigg{)} \otimes t_j
\end{align*}
is an element of $P \otimes_R T$ that can be expressed as a sum of at most $n-1$ elementary tensors. By induction, there exist $\psi_1, \dots, \psi_k \in \Hom_T(P \otimes_R T, T)$ and $g_1, \dots, g_k \in P \otimes_R T$ such that $f' \coloneqq f - \sum_{i = 1}^m \phi_i \otimes \id_T(f) \cdot (a_i \otimes 1)$ satisfies
$
f' = \sum_{j = 1}^k \psi_j(f') \cdot g_j.
$
Substituting the expression for $f'$ we then get
\begin{align*}
f - \sum_{i = 1}^m \phi_i \otimes \id_T(f) \cdot (a_i \otimes 1) &=
\sum_{j=1}^k \psi_j\bigg{(}f - \sum_{i = 1}^m \phi_i \otimes \id_T(f) \cdot (a_i \otimes 1)\bigg{)}\cdot g_j\\
&=\sum_{j=1}^k\psi_j(f) \cdot g_j - \sum_{j=1}^k\sum_{i=1}^m\psi_j\big{(}\phi_i \otimes \id_T(f) \cdot (a_i\otimes 1)\big{)} \cdot g_j.
\end{align*}
Thus,
\begin{equation}
\label{eq:horrible-equation}
f = \sum_{i = 1}^m \phi_i \otimes_R \id_T(f)\cdot (a_i \otimes 1) + \sum_{j=1}^k\psi_j(f) \cdot g_j - \sum_{j=1}^k\sum_{i=1}^m\psi_j\big{(}\phi_i \otimes \id_T(f) \cdot (a_i\otimes 1)\big{)} \cdot g_j.
\end{equation}
The point now is that for all $i = 1, \dots, m$ and all
$j = 1, \dots, k$, the element 
$\psi_j\big{(}\phi_i \otimes \id_T(f) \cdot (a_i\otimes 1)\big{)} = \psi_j\big{(}a_i \otimes (\phi_i \otimes \id_T(f))\big{)}$ 
(here we are identifying the codomain, $R \otimes_R T$, of $\phi_i \otimes \id_T$
with $T$) is evaluation at $f$ of the 
$T$-linear map
\[
P \otimes_R T \xrightarrow{\phi_i \otimes \id_T} T \longrightarrow P \otimes_R T \xrightarrow{\psi_j} T,
\]
where the middle map is given by
\begin{align*}
    T &\to P \otimes_R T\\
    t &\mapsto a_i \otimes t.
\end{align*}
This shows that $f \in \Tr_{P \otimes_R T}(f)(P \otimes_R T)$ by \autoref{eq:horrible-equation}. Consequently, $P \otimes_R T$ is an ORT $T$-module by induction.

(2) This statement follows from (1) by base change along $R \to R[x]$, that is, take $T = R[x]$ in part (1).

(3) Given $\phi \in \Hom_R(S,R)$, we will let 
$\phi[[x]] \colon S[[x]] \to R[[x]]$ be the $R[[x]]$-linear map given by
$
\phi[[x]](\sum_i a_ix^i) = \sum_i \phi(a_i)x^i.
$
One can also think of $\phi[[x]]$ as the $xR[x]$-adic completion of the $R[x]$-linear map $\phi[x] \colon S[x] \to R[x]$.

Let $f \in S[[x]]$. We first claim that
$
\Tr_{S[[x]]}(f)S[[x]] = \bigcap_{n \geq 0} (\Tr_{S[[x]]}(f)S[[x]] + x^nS[[x]]).
$
Here $\Tr_{S[[x]]}$ is computed via the $R[[x]]$-module structure on $S[[x]]$. Since $xS[x]$ is a finitely generated ideal and $S[[x]]$ is the completion of $S[x]$ with respect to $xS[x]$, it follows by \cite[\href{https://stacks.math.columbia.edu/tag/05GG}{Tag 05GG}]{stacks-project} that $S[[x]]$ is $xS[[x]]$-adically complete. Consequently, $xS[[x]]$ is contained in the Jacobson radical of $S[[x]]$ \cite[Thm.\ 8.2]{MatsumuraCommutativeRingTheory}, and now, because $S[[x]]$ is Noetherian, $\Tr_{S[[x]]}(f)S[[x]]$ is closed in $S[[x]]$ in the $xS[[x]]$-adic topology by Krull's intersection theorem \cite[Thm.\ 8.10(i)]{MatsumuraCommutativeRingTheory}. But the closure of $\Tr_{S[[x]]}(f)S[[x]]$ in the $xS[[x]]$-adic topology is precisely $\bigcap_{n \geq 0} (\Tr_{S[[x]]}(f)S[[x]] + x^nS[[x]])$, proving our claim. 

Thus, in order to show that
$f \in \Tr_{S[[x]]}(f)S[[x]]$, it suffices to show that for all integers $n \geq 0$, 
$
f \in \Tr_{S[[x]]}(f)S[[x]] + x^nS[[x]].
$

We proceed by induction on $n$ with the case $n = 0$ being clear because
$\Tr_{S[[x]]}(f)S[[x]] + x^0S[[x]] = S[[x]]$. Suppose the statement holds for $n \geq 0$. Then
choose maps $\psi_1, \dots, \psi_m \in \Hom_{R[[x]]}(S[[x]], R[[x]])$ and 
$g_1,\dots,g_m \in S[[x]]$ such that $
f - \sum_{i=1}^m g_i\psi_i(f) \in x^nS[[x]].
$
Let $a$ be the coefficient of $x^n$ in $f - \sum_{i=1}^m g_i\psi_i(f)$. Since $S$ is an ORT $R$-module, there exist $\phi_1, \dots, \phi_k \in \Hom_R(S,R)$ and $s_1,\dots, s_k \in S$ such that $a = s_1 \phi_1(a) + \dots + s_k \phi_k(a)$.
Then
\[
\big{(}f - \sum_{i=1}^m g_i\psi_i(f)\big{)} - \sum_{j=1}^ks_j\phi_j[[x]](f - \sum_{i=1}^m g_i\psi_i(f)) \in x^{n+1}S[[x]]
\]
because now we have killed the $ax^n$ term.
Now 
\begin{align*}
    &\big{(}f - \sum_{i=1}^m g_i\psi_i(f)\big{)} - \sum_{j=1}^ks_j\phi_j[[x]](f - \sum_{i=1}^m g_i\psi_i(f)) =\\
    & f - \bigg{(}\sum_{i=1}^m g_i\psi_i(f) + \sum_{j=1}^ks_j\phi_j[[x]](f) -
    \sum_{j=1}^k\sum_{i=1}^ms_j\phi_j[[x]](g_i\psi_i(f))\bigg{)}.
\end{align*}
 As in (1), $\phi_j[[x]](g_i\psi_i(f))$ is the $R[[x]]$-linear map $S[[x]] \xrightarrow{\psi_i} R[[x]] \to S[[x]] \xrightarrow{g_i\cdot} S[[x]]
 \xrightarrow{\phi_j[[x]]} R[[x]]$
 evaluated at $f$. Thus $f \in \Tr_{S[[x]]}(f)S[[x]] + x^{n+1}S[[x]]$,
 finishing the proof of the inductive step.

 (4) Suppose $I = (i_1,\dots,i_n)$. Then $IS$ is generated by the
 images of $i_j$ in $S$. By \cite[Thm.\ 8.12]{MatsumuraCommutativeRingTheory}, the ring map
 $
 \varphi : R[[x_1,\dots,x_n]] \to \widehat{R}^I
$
that sends $x_j \mapsto i_j$ is surjective with kernel $(x_1 - i_1, \dots, x_n - i_n)$. Then
$
S[[x_1,\dots,x_n]] \otimes_{R[[x_1,\dots,x_n]]}\widehat{R}^I \cong
\frac{S[[x_1,\dots,x_n]]}{(x_1-i_1,\dots,x_n-i_n)S[[x_1,\dots,x_n]]}
\cong \widehat{S}^{IS},
$
where the second isomorphism again follows by \cite[Thm.\ 8.12]{MatsumuraCommutativeRingTheory}.
The upshot is that $\widehat{R}^I \to \widehat{S}^{IS}$ can be identified with the base change of the canonical map $R[[x_1,\dots,x_n]] \to S[[x_1,\dots,x_n]]$ (that sends $x_j \mapsto x_j$) along $\varphi \colon R[[x_1,\dots,x_n]] \twoheadrightarrow \widehat{R}^I$. By a repeated application of (3), $R[[x_1,\dots,x_n]] \to S[[x_1,\dots,x_n]]$ is ORT, and by (1), the ORT property is preserved under base change. Thus, $\widehat{R}^I \to \widehat{S}^{IS}$ is ORT as well.
\end{proof}

\begin{corollary}
\label{cor:completions-ORT}
    Let $\varphi \colon (R, \fm) \to (S, \mathfrak n)$ be a local homomorphism of Noetherian local rings such that $\sqrt{\fm S} = \mathfrak n$. If $\varphi$ is ORT, then $\widehat{R}^{\fm} \to \widehat{S}^{\mathfrak n}$ is also ORT.
\end{corollary}

\begin{proof}
    By \autoref{prop:ORT-indeterminate} (4),
    $\widehat{R}^{\fm} \to \widehat{S}^{\fm S}$ is ORT. But
    $\widehat{S}^{\fm S} \cong \widehat{S}^{\mathfrak n}$ because
    $\sqrt{\fm S} = \mathfrak n$.
\end{proof}

We also obtain that the ORT property is preserved under localization. Note that this was already observed in \cite[7.1(b)]{OhmRu-content}.

\begin{corollary}
    \label{cor:ORT-localization}
    Let $R$ be a ring and $M$ be an ORT $R$-module. Then for a multiplicative set $S$ of $R$, $S^{-1}M$ is an ORT $S^{-1}R$-module. Moreover, for all $x \in M$ and $s \in S$, 
   $
    \Tr_{S^{-1}M}(x/s) = \Tr_{M}(x)(S^{-1}R).    
    $
\end{corollary}

\begin{proof}
    Base change along $R \to S^{-1}R$ gives us that $S^{-1}M$ is an ORT $R$-module by \autoref{prop:ORT-indeterminate} (1). Since ORT modules are flat (\autoref{rem:ORT-modules} (g)) and Ohm-Rush (\autoref{rem:ORT-modules} (a)), it follows by \autoref{prop:content-homomorphisms-purity-flatness} (4) and \autoref{rem:ORT-modules} (a) that for all $x \in M$ and $s \in S$,
    \[
    \Tr_M(x)(S^{-1}R) = c_M(x)(S^{-1}R) = c_{S^{-1}M}(x/s) = \Tr_{S^{-1}M}(x/s). \qedhere
    \]
\end{proof}

\subsection{Intersection flatness}
The notion of intersection flatness seems to have been first studied in the special case of the Frobenius map in \cite{HochsterHunekeFRegularityTestElementsBaseChange}. The general theory of intersection flatness was developed recently in \cite{HochsterJeffriesintflatness}.

\begin{notation}
\label{not:for-IF}
If $L,M$ are $R$-modules and $U$ is a submodule of $L$, the notation $UM$ refers to the image of $U \otimes_R M \to L \otimes_R M$ induced by $U \hookrightarrow L$.
\end{notation}

\begin{definition}\cite[Sec.\ 5]{HochsterJeffriesintflatness}
\label{def:IF}
    Let $R$ be a ring and $M$ an $R$-module. We say  that $M$ is \emph{intersection flat} if for any finitely generated $R$-module $L$ and any collection $\{U_i\}_{i \in \Lambda}$ of $R$-submodules of $L$, 
    \begin{equation}
    \label{eq:IF}
    \left(\bigcap_{i \in \Lambda} U_i\right)M = \bigcap_{i \in \Lambda} (U_i M).
    \end{equation}
    A ring map $R \to S$ is \emph{intersection flat} if $S$ is an intersection flat $R$-module.
\end{definition}

\begin{remark}
\label{rem:IF-OR}
Suppose $M$ is an intersection flat $R$-module. 
\begin{enumerate}
    \item Taking $L = R$ in 
    \autoref{def:IF}, we see that if $\{I_j\}_{j \in \Lambda}$ is a 
    collection of ideals of $R$, then $\left(\bigcap_{j \in \Lambda} I_j\right)M = \bigcap_{j \in \Lambda} I_jM$. Consequently, an intersection flat module is Ohm-Rush (\autoref{lem:OR-weak-intersection-flat}).

    \item $M$ is a flat $R$-module by \cite[Prop.\ 5.5]{HochsterJeffriesintflatness}. In fact, by loc. cit. flatness follows provided \autoref{eq:IF} holds whenever $\Lambda$ (the index set) is finite.

    \item In \autoref{def:IF}, it suffices to assume that $L$ is a free module of finite rank by \cite[Prop.\ 5.6]{HochsterJeffriesintflatness}.
    
    \item If $R \to S$ is an intersection flat ring map and $M$ is an intersection flat $S$-module, then $M$ is an intersection flat $R$-module. This follows by \cite[Prop.\ 5.7(a)]{HochsterJeffriesintflatness}.
\end{enumerate}
\end{remark}

\subsection{Connections between various notions}
We have introduced five different notions for modules, namely Mittag-Leffler (ML), strictly Mittag-Leffler (SML), Ohm-Rush, Ohm-Rush trace (ORT) and intersection flat. We have observed some relationships between these notions in the previous subsections (for example, see \autoref{rem:Mittag-Leffler} (a), \autoref{prop:ML-SML}, \autoref{cor:ML-implies-OR}, \autoref{rem:ORT-modules} (a)). Our goal in this subsection is to develop further connections between these notions. Many of these connections can only be made under an underlying flatness assumption.

Recall that if (P) is a property of modules, we say that an $R$-module $M$ is \emph{universally} (P) if for any $R$-algebra $R \to S$, we have that $S \otimes_R M$ has (P) as an $S$-module.

The first connection is between the notions of ML and intersection flat modules.

\begin{theorem}\label{thm:MittagORIF}
Let $R$ be a ring and $M$ a flat $R$-module.  The following are equivalent: 
\begin{enumerate}
    \item[$(1)$] $M$ is ML.
    \item[$(2)$] $M$ is universally intersection flat.
    \item[$(3)$] $M$ is intersection flat.
\end{enumerate}
If $R$ is local (not necessarily Noetherian) then $(1)-(3)$ is equivalent to
the following condition:
\begin{enumerate}
    \item[$(4)$] For any finitely generated submodule  $P$ of 
    $M$, there exists a finitely generated submodule $Q$ of $M$ containing $P$
    such that $Q$ is free and $Q \subseteq M$ is a pure extension of
    $R$-modules.
\end{enumerate}
\end{theorem}

\begin{proof}
We first prove the equivalence of $(1)-(3)$. Clearly $(2) \implies (3)$.

To prove that $(1) \implies (2)$, suppose $M$ is a ML module.  Since the ML property is ``universal'' (see \autoref{rem:Mittag-Leffler} (b)), it will suffice to show that $M$ is intersection flat.  So let $L$ be a finitely generated $R$-module and $\{U_i\}_{i \in \Lambda}$ a collection of $R$-submodules of $L$.  Consider the following 
exact sequence: \[
0 \to \bigcap_{i \in \Lambda} U_i \to L \to \prod_{i \in \Lambda} (L / U_i),
\]
where the first map is inclusion and the second map is $x \mapsto (x + U_i)_{i \in \Lambda}$.  Now tensor with $M$ and use the fact (from \autoref{rem:Mittag-Leffler} (b)) that the natural map $(\prod_{i \in \Lambda} (L / U_i)) \otimes_R M \to \prod_{i \in \Lambda} ((L / U_i) \otimes_R M)$ is injective, along with the flatness of $M$ over $R$, to get the following exact sequence 
\[
0 \to \left(\bigcap_{i \in \Lambda} U_i\right) \otimes_R M \overset{\alpha}{\to} L \otimes_R M \overset{\beta}{\to} \prod_{i \in \Lambda} ((L / U_i) \otimes_R M),
\]
where $\beta$ is given by $x \otimes m \mapsto ((x + U_i) \otimes m)_{i \in \Lambda}$.  Thus, the image $(\bigcap_i U_i) M$ of $\alpha$ is the kernel of $\beta$.  On the other hand, let $z = \sum_{j=1}^n x_j \otimes y_j \in L \otimes_R M$. Then $z\in \ker \beta$ if and only if for all $i \in \Lambda$, $\sum_{j=1}^n (x_j + U_i) \otimes y_j = 0$ in $(L / U_i) \otimes_R M$.  But the exactness of the sequences $0 \to U_i \otimes_R M \to L \otimes_R M \to (L/U_i) \otimes_R M$ show that this vanishing is equivalent to $z \in $im$(U_i \otimes_R M \to L \otimes_R M) = U_i M$ for all $i \in \Lambda$.  We have thus shown that $(\bigcap_i U_i)M = \ker(\beta) =  \bigcap_i (U_i M)$.  Since $L$ and the $U_i$ were arbitrary, it follows that $M$ is intersection flat.

The last thing to prove for the equivalence of $(1)-(3)$
is $(3) \implies (1)$.  Accordingly, let $M$ be an intersection flat $R$-module.  Now, let $G$ be a finitely generated free $R$-module and $x \in G \otimes_R M$.  Let 
$S \coloneqq \{U \colon U$ is a submodule of $G$ and $x \in UM\}$.
Then $\bigcap_{U \in S} U \in S$ as well because $x \in \bigcap_{U \in S} UM = (\bigcap_{U \in S} U)M$.  That is, for a finite
free $R$-module $G$ and $x \in G \otimes_R M$, there exists a 
\emph{unique smallest} $R$ submodule $C$ of $G$ such that $x \in C \otimes_R M$ 
(we can identify $C \otimes_R M$ with $CM$ by flatness of $M$).
By \cite[Part II, Proposition 2.1.8]{rg71} (alternate reference \cite[\href{https://stacks.math.columbia.edu/tag/059S}{Tag 059S}]{stacks-project}), 
since $M$ is flat, it follows that $M$ is ML.
If $R$ is a local ring then (1) is equivalent to (4) by \autoref{cor:ML-OR-local-difference} (1).
\end{proof}

\autoref{thm:MittagORIF} implies the following behavior of intersection flatness with respect to pure maps.

\begin{corollary}
    \label{cor:pure-descent-IF}
    Let $R \to S$ be a ring map and $M$ be an $R$-module. We have the following: 
    \begin{enumerate}
        \item[$(1)$] If $R \to S$ is pure and $M \otimes_R S$
    is an intersection flat $S$-module, then $M$ is an intersection flat $R$-module.

        \item[$(2)$] If $N$ is a pure submodule of $M$ and $M$ is an intersection flat $R$-module, then $N$ is an intersection flat $R$-module.

        \item[$(3)$] Let $\{(M_i,\varphi_{ij}) \colon i, j \in I, i \leq j\}$ be a direct system of $R$-modules indexed by a filtered poset $(I,\leq)$. If each $M_i$ is intersection flat and the transition maps $\varphi_{ij}$ are $R$-pure, then $\colim_{i \in I} M_i$ is an intersection flat $R$-module.
    \end{enumerate} 
\end{corollary}

\begin{proof}
    (1) Since $M \otimes_R S$ is $S$-intersection flat, $M \otimes_R S$ is a flat $S$-module by \autoref{rem:IF-OR}. By pure descent, $M$ is a flat $R$-module \cite[\href{https://stacks.math.columbia.edu/tag/08XD}{Tag 08XD}]{stacks-project}. In light of \autoref{thm:MittagORIF}, it suffices to show that $M$ is a ML $R$-module. But the ML property satisfies descent along a pure ring map by \autoref{thm:descentOhm-Rush} (1) and $M \otimes_R S$ is a ML $S$-module by \autoref{thm:MittagORIF} since it is intersection flat.

    (2) $N$ is flat by \autoref{lem:pure-submodule-flat} and ML by \autoref{rem:Mittag-Leffler} (c). Thus $N$ is intersection flat by \autoref{thm:MittagORIF}.

    (3) Since each $M_i$ is flat and ML, so is $\colim_{i \in I} M_i$ by \autoref{rem:Mittag-Leffler} (e).
\end{proof}

\begin{remark}
    One can give a direct proof of pure descent of the intersection flatness property along the lines of the proof of cyclically pure descent of the flat Ohm-Rush property given in \autoref{rem:cyclic-pure-descent-OR}. We omit the details.
\end{remark}

One can also use known properties about intersection flat modules to deduce consequences for ML modules. For instance, we can characterize the ML and Ohm-Rush properties in terms of ideal adic separatedness. 

\begin{corollary}
    \label{cor:ML-separation}
Let $(R, \fm)$ be a Noetherian local ring and $M$ be
a flat $R$-module. 
\begin{enumerate}
    \item[$(1)$] If $M$ is a ML $R$-module, then for all finitely generated $R$-modules $L$, $L \otimes_R M$ is $\fm$-adically separated. The converse holds if $R$ is $\fm$-adically complete.
    
    \item[$(2)$] If $M$ is an Ohm-Rush $R$-module, then for all cyclic $R$-modules $L$, $L \otimes_R M$ is $\fm$-adically separated. The converse holds if $R$ is $\fm$-adically complete.
\end{enumerate}
\end{corollary}

\begin{proof}
    $(1)$ Since $L$ is a finitely generated $R$-module, by Krull's intersection theorem we have that $L$ is $\fm$-adically separated. Since flat ML modules are intersection flat (\autoref{thm:MittagORIF}), one then has
    \[
    \bigcap_{n \in \mathbf{Z}_{> 0}} \fm^n(L \otimes_R M) = \bigcap_{n \in \mathbf{Z}_{> 0}} (\fm^nL)M = \left(\bigcap_{n \in \mathbf{Z}_{> 0}}\fm^n L\right)M = 0,
    \]
    where the second equality follows by intersection flatness. This shows that $L \otimes_R M$ is $\fm$-adically separated, that is, we have the forward implication of $(1)$.
    
    If $(R, \fm)$ is $\fm$-adically complete, then an application of Chevalley's 
    Lemma shows that $M$ must be an intersection
    flat $R$-module by the proof of \cite[Prop.\ 5.7(e)]{HochsterJeffriesintflatness} (loc. cit.
    is stated for flat $R$-algebras but the proof readily applies to
    modules). Again, by \autoref{thm:MittagORIF} $M$ is a ML $R$-module.
    This proves $(1)$.

    The proof of $(2)$ is similar. When $R$ is $\fm$-adically
    complete, the backward implication again follows by the proof of \cite[Prop.\ 5.7(e)]{HochsterJeffriesintflatness} (as before, Hochster-Jeffries state the result for flat algebras, but their proof works for modules). A minor point is that a module that is  \emph{intersection flat for ideals} in \cite{HochsterJeffriesintflatness} is in our terminology a flat Ohm-Rush module (the equivalence of the two notions
    follows by \autoref{lem:OR-weak-intersection-flat}). 
    
    Now suppose $(R, \fm)$ is an arbitary Noetherian local ring,
    $M$ is an Ohm-Rush $R$-module and $L$ is a cyclic $R$-module. We may
    assume without loss of generality that $L = R/I$, where $I$ is a
    proper ideal of $R$. Then
    \begin{align*}
    \bigcap_{n \in \mathbf{Z}_{> 0}} \fm^n(R/I \otimes_R M) 
    &\cong \bigcap_{n \in \mathbf{Z}_{> 0}} \fm^n(M/IM) 
    = \frac{\bigcap_{n \in \mathbf{Z}_{> 0}} (\fm^n + I)M}{IM}
    = \frac{\left(\bigcap_{n \in \mathbf{Z}_{>0}} \fm^n + I\right)M}{IM},
    \end{align*}
    where the last equality follows because $M$ is an Ohm-Rush 
    $R$-module (\autoref{lem:OR-weak-intersection-flat}).
    By Krull's intersection theorem, $\bigcap_{n \in \mathbf{Z}_{>0}} \fm^n + I = I$, and so, $R /I \otimes_R M$ is $\fm$-adically
    separated. 
\end{proof}

\begin{remark}
    \label{rem:improving-corollary}
    In \autoref{thm:ML-SML-ORT-intersection-flat} we will significantly improve \autoref{cor:ML-separation} and show that over a complete Noetherian local ring, the notions of being Ohm-Rush and ML are equivalent for flat modules.
\end{remark}

Gruson and Raynaud showed that a flat SML module $R$-module is equivalent to an ORT $R$-module.

\begin{citedthm}\cite[Part II, Prop.\ 2.3.4, equivalence of (i) and (iii)]{rg71}
\label{thm:SML-ORT}
Let $R$ be a ring and let $M$ be an $R$-module. The following are equivalent:
\begin{enumerate}
    \item[$(1)$] $M$ is a flat SML $R$-module.
    \item[$(2)$] $M$ is an ORT $R$-module.
\end{enumerate}
\end{citedthm}

\begin{remark}
    \label{rem:justifying-one-implication}
    In \autoref{thm:SML-ORT}, the implication $(1) \implies (2)$ follows readily from the theory we have already developed. Suppose $M$ is a flat and SML $R$-module. Let $x \in M$. We wish to show that $x \in \Tr_M(x)M$. Let $g \colon R \to M$ be the unique $R$-linear map that sends $1 \mapsto x$. Since $M$ is a SML $R$-module, $g$ admits a stabilizer that factors through $g$. Expressing $M$ as a filtered colimit of free modules of finite rank, \autoref{lem:stabilizers} (2) implies that $g$ admits a stabilizer $f \colon R \to L$ where $L$ is a free module of finite rank and such that $f$ factors through $g$ (note that $g$ factors through $f$ by \autoref{lem:stabilizers} (1)). Let $y \coloneqq f(1)$. Since $f$ factors through $g$, there exists a linear map $\varphi \colon M \to L$ such that $f = \varphi \circ g$. Thus, $y = f(1) = \varphi \circ g (1) = \varphi(x)$. Hence, $\Tr_L(y) = \Tr_L(\varphi(x)) \subseteq \Tr_M(x)$ by  \autoref{rem:ORT-modules} (h). Similarly, since $g = \phi \circ f$ for some linear $\phi \colon P \to M$, we get $\Tr_M(x) \subseteq \Tr_L(y)$. Thus, $\Tr_M(x) = \Tr_L(y)$, and since $y \in \Tr_L(y)L$ as free modules are ORT (\autoref{lem:projective-ORT}), we get $x = g(1) = \phi \circ f(1) = \phi(y) \in \Tr_L(y)M = \Tr_M(x)M$.
\end{remark}

As a consequence, we then have

\begin{proposition}
\label{prop:ORT-intersection-flat}
Let $R$ be a ring and $M$ be an ORT $R$-module.  Then $M$ is intersection flat as an $R$-module.
\end{proposition}

\begin{proof}
By \autoref{thm:SML-ORT}, $M$ is a flat SML $R$-module. Thus, by \autoref{rem:Mittag-Leffler} (a), $M$ is a flat ML $R$-module, and hence is intersection flat
by \autoref{thm:MittagORIF}.
\end{proof}

Additionally, pure maps to ORT-modules are often split.

\begin{corollary}
    \label{cor:pure-maps-to-ORT-split}
    Let $R$ be a ring and $M$ be an ORT $R$-module. Then for any finitely presented $R$-module $P$, if $\varphi \colon P \to M$ is a pure $R$-linear map, then $\varphi$ splits.
\end{corollary}

\begin{proof}
    Since $M$ is SML, the result follows by \autoref{rem:Mittag-Leffler} (g).
\end{proof}

\begin{remark}
    \label{rem:ML-SML-not-equivalent}
    Suppose $(R, \fm)$ is a Noetherian local ring that is \emph{not} $\fm$-adically complete. Then the notions of ML and SML modules do not coincide for the class of modules that are faithfully flat $R$-algebras. For an example that is relevant to the theory of $F$-singularities, consider an excellent, Henselian DVR $(R, \fm)$ of prime characteristic $p > 0$ such that $\Hom_R(F_*R, R) = 0$ (such examples exist by \cite{DattaMurayamaTate}). Consider the $R$-algebra $F_*R$. Then $F_*R$ is a faithfully flat $R$-algebra by \cite{KunzCharacterizationsOfRegularLocalRings}. Moreover, since $R$ is excellent, we will see in forthcoming work that $F_*R$ is a ML $R$-module \cite{DESTtate}. However, $F_*R$ is not a SML $R$-module by \autoref{thm:SML-ORT}. Otherwise $F_*R$ would be an ORT $R$-module, which would imply that the canonical map $R \to \Hom_R(\Hom_R(F_*R,R),R)$ is injective by \autoref{lem:ORT-cyclic-purity}. This is impossible because $\Hom_R(F_*R, R) = 0$ by the choice of $R$.
  \end{remark}

We have seen so far that the following properties are preserved under arbitrary base change: ML (\autoref{rem:Mittag-Leffler} (b)), ORT (\autoref{prop:ORT-indeterminate} (1)), intersection flat (\autoref{thm:MittagORIF}), flat and SML (\autoref{thm:SML-ORT} and \autoref{prop:ORT-indeterminate} (1)). Thus, it is natural to ask if the property of being Ohm-Rush or, more restrictively, of being flat Ohm-Rush is preserved under arbitrary base change. We will now show that this is false.

\begin{example}
    \label{eg:OR-not-preserved-base-change-polynomial}
Let $R = k[x]_{(x)}$, $S = k[\![x]\!]$, where $x$ is an indeterminate and $k$ is an arbitrary field, and let $R \to S$ be the $(x)$-adic completion map.  By \cite[Proposition 2.1]{OhmRu-content}, $S$ is an Ohm-Rush $R$-algebra.  However, if $y$ is another indeterminate, then $S[y]$ is \emph{not} an Ohm-Rush $R[y]$-module.  To see this, we imitate the method of \cite[Example 5,3]{HochsterJeffriesintflatness}.  Accordingly, let $f\in S$ be a power series that is transcendental over $R$.  For each $n \in \N$, let $f_n$ be the unique polynomial of degree $\leq n$ that agrees with $f$ modulo $x^{n+1}S$.  Set $I_n := (y-f_n, x^{n+1})R[y] \subseteq R[y]$.  Then $I_n S = (y-f,x^{n+1})S[y] = (y-f)S[y] + (x)^{n+1}S[y]$.  Since $S[y] / (y-f) \cong S$ is an integral domain, it follows from the Krull intersection theorem for Noetherian domains that $
\bigcap_{n \in \N} (I_n S[y]) = (y-f)S[y].
$
Set $R' := k[x,y]_{(x,y)}$ and $S' := k[\![x,y]\!]$.  Note that $(y-f)S' \cap R' = 0$ by the transcendence assumption, as in the argument from \cite[Example 5.3]{HochsterJeffriesintflatness}.  But then \begin{align*}
\bigcap_{n \in \N} I_n &= \left( \left(\bigcap_{n \in \N} I_n\right) S[y]\right)S' \cap R' \cap R[y] \subseteq \left( \left(\bigcap_{n \in \N} I_n\right) S[y]\right)S' \cap R' \\
&\subseteq \left( \bigcap_{n \in \N} (I_n S[y])\right)S' \cap R' = (y-f)S' \cap R' = 0.
\end{align*}
Since $\bigcap_n I_n = 0$ but $\bigcap_n (I_n S[y]) \neq 0$, it follows that $S[y]$ is not an Ohm-Rush $R[y]$-algebra.
\end{example}

The ML property coincides with the SML property for modules over a Noetherian complete local ring by \autoref{prop:ML-SML}. If we assume flatness, then all the notions introduced so far are equivalent.

\begin{theorem}
\label{thm:ML-SML-ORT-intersection-flat}
Let $(R, \fm)$ be a Noetherian local ring that is complete with respect to the
$\fm$-adic topology. Let $M$ be a flat $R$-module. Then the $R$-dual $\Hom_R(M,R)$
is an ORT $R$-module (equivalently, a flat and SML $R$-module).
Furthermore, the following are equivalent:
\begin{enumerate}
    \item[$(1)$] $M$ is intersection flat.
    
    \item[$(2)$] $M$ is ML.
    
    \item[$(3)$] $M$ is SML.
    
    \item[$(4)$] $M$ is an ORT $R$-module.

    \item[$(5)$] $M$ is an Ohm-Rush $R$-module.
    
    \item[$(6)$] If $\widehat{M}$ denotes the $\fm$-adic completion of $M$, then the canonical map $M \to \widehat{M}$ is a (cyclically) pure map of $R$-modules.
    
    \item[$(7)$] The canonical map $M \to \Hom_R(\Hom_R(M,R),R)$ is a pure map of $R$-modules.
    
    \item[$(8)$] The canonical map $M \to \Hom_R(\Hom_R(M,R),R)$ is cyclically pure as a map of $R$-modules.
    
    \item[$(9)$] For all finitely generated $R$-modules $L$,
    $L \otimes_R M$ is $\fm$-adically separated.

    \item[$(10)$] For all cyclic $R$-modules $L$, $L \otimes_R M$ is $\fm$-adically separated.
    
    \item[$(11)$] For any finitely generated submodule $P$ of $M$, there exists a finitely generated submodule $L$ of $M$ containing $P$ such that $L$ is free and is a direct summand of $M$.

    \item[$(12)$] For any cyclic submodule $P$ of $M$, there exists a finitely generated submodule $L$ of $M$ containing $P$ such that $L$ is free and is a direct summand of $M$.
    
    \item[$(13)$] $M$ is a filtered union of its finitely generated submodules that are free and direct summands of $M$.
\end{enumerate}
\end{theorem}

\begin{proof}
An ORT $R$-module is equivalently a flat and SML $R$-module by \autoref{thm:SML-ORT}. The fact that $\Hom_R(M,R)$ is an ORT $R$-module is shown
in \cite[Part II]{rg71}.  
More explicitly, in \cite[Part II, (2.4.1)]{rg71}, Raynaud and Gruson show that if $M$ is a module over \emph{any} Noetherian ring $R$ (not necessarily local) such
that the functor 
$
\Hom_R(M,\cdot) \colon \textrm{Mod}^{fg}_R \to \Ab
$
is exact on the category of finitely generated
$R$-modules, then $\Hom_R(M,R)$ must be an ORT $R$-module. For this, they use the natural
transformation of functors
$
\Hom_R(M, R) \otimes_R \cdot \Longrightarrow \Hom_R(M,\cdot)
$
and observe that this natural transformation is a natural isomorphism on $\textrm{Mod}^{fg}_R$.
Now, if $M$ is a flat module over a Noetherian
local ring $(R, \fm)$ that is $\fm$-adically complete, then
$\Hom_R(M, \cdot)$ is exact on $\textrm{Mod}^{fg}_R$ using a surprising vanishing theorem of Jensen's \cite[Thm.\ 1]{Jensenlimvanishing}: for a flat module $M$ over a Noetherian complete local ring $R$, $\Ext^i_R(M,N) = 0$ for all finitely generated $R$-modules $N$ and all
$i > 0$. This, in particular, implies $\Hom_R(M,R)$ is $R$-flat. Now let $\varphi \in \Hom_R(M,R)$ and consider the short exact sequence of finitely generated $R$-modules
$
0 \to \im(\varphi) \to R \to \coker(\varphi) \to 0.    
$
A diagram chase using
\[
\begin{tikzcd}
  0 \arrow[r] & \Hom_R(M,R) \otimes_R \im(\varphi) \arrow[d, "\cong"] \arrow[r] & \Hom_R(M,R) \arrow[d, "\id"] \arrow[r] & \Hom_R(M,R) \otimes_R \coker(\varphi) \arrow[d, "\cong"] \arrow[r] & 0 \\
  0 \arrow[r] & \Hom_R(M,\im(\varphi)) \arrow[r] & \Hom_R(M,R) \arrow[r] & \Hom_R(M,\coker(\varphi)) \ar[r] & 0
\end{tikzcd}
\]
reveals that $\varphi \in \im(\varphi) \cdot \Hom_R(M,R) = \im(\Hom_R(M,R) \otimes_R \im(\varphi) \to \Hom_R(M,R))$. If we look at the map
$
\ev @ \varphi \colon \Hom_R(\Hom_R(M,R),R) \to R,    
$
then clearly for all $x \in M$, 
$
\ev @ \varphi (\ev @ x) = \varphi(x),    
$
that is, $\im(\varphi) \subseteq \im(\ev @ \varphi) = \Tr_{\Hom_R(M,R)}(\varphi)$. Thus, 
$
\varphi \in \im(\varphi) \cdot \Hom_R(M,R) \subseteq \Tr_{\Hom_R(M,R)}(\varphi) \cdot \Hom_R(M,R),
$
that is, $\Hom_R(M,R)$ is ORT.

We now tackle the equivalence of the assertions. 

The equivalence of $(1)$ and $(2)$ follows by \autoref{thm:MittagORIF}. The equivalence of $(2)$ and $(3)$ follows by
\autoref{prop:ML-SML}. The equivalence of $(3)$ and $(4)$ follows by
\autoref{thm:SML-ORT}. 

Note that in the statement of $(6)$ the $R$-cyclic purity of $M \to \widehat{M}$ is equivalent to its $R$-purity by \autoref{lem:cyclic-purity-flat}. This is because the flatness of $M$ implies that $\widehat{M}$ is a flat $R$-module by \autoref{lem:completion-purity}.

The equivalence of $(3)$, $(6)$ and $(7)$ follows by
\cite[Part II, Prop.\ (2.4.3.1)]{rg71}. 

For the equivalence of 
$(7)$ and $(8)$, note that $(7) \implies (8)$ is clear because pure maps of modules
are cyclically pure. The implication $(8) \implies (7)$ follows by 
\autoref{lem:cyclic-purity-flat}
because $\Hom_R(\Hom_R(M,R),R)$ is a flat $R$-module since $\Hom_R(M,R)$ is $R$-flat and so taking its $R$-dual preserves flatness by the first part of this Theorem. 

For the equivalence of $(5)$ and $(6)$, note that $(5) \implies (6)$ follows by \autoref{cor:content-completion} (4) because $R = \widehat{R}$ by assumption, while we have $(6) \implies (3) \implies (4)$ and ORT modules are Ohm-Rush by \autoref{rem:ORT-modules} (a), that is, $(6) \implies (5)$.

By \autoref{cor:ML-separation}, we have the equivalence of $(2)$ and $(9)$ as well as the equivalence of $(5)$ and $(10)$. 

Thus, assertions $(1)$--$(10)$ are all equivalent.

The equivalence of $(11)$ and $(13)$ is straightforward and we omit its proof.

Since $M$ is flat, the equivalence of $(2)$ and $(11)$ follows by \autoref{cor:ML-OR-local-difference} (1) and the fact that purity of a map $L \to M$, when $L$ is finitely generated, is equivalent to its splitting since we are working over a Noetherian complete local ring (see \autoref{lem:Auslander-Warfield-lemma}).

To finish the proof, it remains to show the equivalence of $(5)$ and $(12)$. The proof is similar to that of $(2)\Longleftrightarrow(11)$. The desired equivalence follows by \autoref{cor:ML-OR-local-difference} (2) and again the fact that purity of $L \to M$ is equivalent to its splitting when $L$ is finitely generated.
\end{proof}

\begin{remark}
\label{rem:Justifying-Raynaud-Gruson}
{\*}
\begin{enumerate}[itemsep=1mm]
\item 
In \cite[Part II, Prop.\ (2.4.3.1)]{rg71}, Raynaud and Gruson prove
the equivalence of statements $(3), (6)$ and $(7)$ of \autoref{thm:ML-SML-ORT-intersection-flat} by showing $(3) \implies (7) \implies (6) \implies (3)$. Their proof of $(3) \implies (7)$ is terse, so we include a justification here for the reader's convenience. If $M$ is an ORT $R$-module, then $M \to \Hom_R(\Hom_R(M,R),R)$ is $R$-pure because $M \to \Hom_R(\Hom_R(M,R),R)$ is $R$-cyclically pure by \autoref{lem:ORT-cyclic-purity} and one has that the cyclic purity of this canonical map is equivalent to its purity by the equivalence of $(7)$ and $(8)$ in \autoref{thm:ML-SML-ORT-intersection-flat} (whose justification we have provided in the proof above). 

\item Let $(R,\fm)$ be a Noetherian local ring (not necessarily $\fm$-adically complete) and let $M$ be a flat $R$-module. In \autoref{cor:content-completion}, we saw that if $M$ is Ohm-Rush, then the canonical map $M \to \widehat{M}$ is pure as a map of $R$-modules. In \autoref{thm:ML-SML-ORT-intersection-flat} we saw that the converse holds if $R$ is complete. However, we caution the reader than purity of $M \to \widehat{M}$ in the non-complete local case does not imply that $M$ is Ohm-Rush. Indeed, if $M = \widehat{R}$, then $M \to \widehat{M}$ is an isomorphism. However, $\widehat{R}$ is rarely an Ohm-Rush $R$-module; see \cite{EpsteinShapiroOhmRushIII}.
\end{enumerate}
\end{remark}

We now have the following consequences of the previous Theorem. The first one is that flat and complete modules over a complete Noetherian local ring satisfy all the notions that have been introduced so far.

\begin{corollary}
    \label{cor:complete-over-complete-great}
    Let $(R,\fm)$ be a Noetherian local ring that is complete with respect to the $\fm$-adic topology. Suppose $M$ is a flat $\fm$-adically complete $R$-module. Then $M$ is ML, SML, Ohm-Rush, ORT and intersection flat. 
\end{corollary}

\begin{proof}
    By assumption, the canonical map $M \to \widehat{M}$ is an isomorphism of $R$-modules, and hence, it is pure. Then we get all the desired properties of $M$ by the equivalent statements $(1)-(6)$ in \autoref{thm:ML-SML-ORT-intersection-flat}.
\end{proof}

\begin{corollary}
    \label{cor:ORT-complete-local}
    Let $(R, \fm)$ be a Noetherian local ring that is complete with respect to the $\fm$-adic topology. Suppose $R \to S$ is a flat ring homomorphism and let $\widehat{S}$ denote the $\fm$-adic completion of $S$. Then we have the following:
    \begin{enumerate}
        \item[$(1)$] $R \to \widehat{S}$ is ORT. 
        \item[$(2)$] If $S$ is Noetherian and $\fm S$ is contained in the Jacobson radical of $S$, then $R \to S$ is ORT.
    \end{enumerate}
\end{corollary}

\begin{proof}
    (1) $\widehat{S}$ is $R$-flat (\autoref{lem:completion-purity}) and $\fm$-adically complete since $\fm$ is a finitely generated ideal \cite[\href{https://stacks.math.columbia.edu/tag/05GG}{Tag 05GG}]{stacks-project}. Thus, $R \to \widehat{S}$ is ORT by \autoref{cor:complete-over-complete-great}.

    (2) Since $\fm S$ is contained in the Jacobson radical of $S$, $S \to \widehat{S}$ is faithfully flat because $S$ is Noetherian. Thus, by restriction of scalars, $S \to \widehat{S}$ is $R$-pure, so $R \to S$ is ORT by \autoref{thm:ML-SML-ORT-intersection-flat}.
\end{proof}

Another consequence of \autoref{thm:ML-SML-ORT-intersection-flat} is that over a Noetherian local ring, the intersection flat modules are precisely the flat Ohm-Rush modules such that the Ohm-Rush property is preserved under arbitrary base change. That is, intersection flat modules over a Noetherian local ring are precisely flat modules that are \emph{universally Ohm-Rush}. This was once believed to always be the case \cite[Proposition 6]{Pic-contenu}, although the author of that paper no longer believes his proof \cite{Picemail}.

\begin{corollary}(cf. \cite[Proposition 6]{Pic-contenu})
\label{cor:IF-universally-OR}
Let $(R,\fm)$ be a Noetherian local ring and $M$ be an $R$-module. Then the following are equivalent:
\begin{enumerate}[label=\textnormal{(\arabic*)}]   
    \item $M$ is intersection flat.
    \item $M$ is flat and Mittag-Leffler.
    \item $M$ is flat and universally Ohm-Rush, that is, for all $R$-algebras $S$, $M \otimes_R S$ is a flat and Ohm-Rush $S$-module.
    \item 
    $M \otimes_R \widehat{R}$ is a flat and Ohm-Rush $\widehat{R}$-module, where $\widehat{R}$ is the $\fm$-adic completion of $R$.
\end{enumerate}
\end{corollary}

\begin{proof}
The equivalence of $(1)$ and $(2)$ follows by \autoref{thm:MittagORIF}. Thus, it suffices to show that $(1) \implies (3) \implies (4) \implies (1)$. Since the intersection flatness property is preserved under base change (\autoref{thm:MittagORIF}), we get $(1) \implies (3)$. Moreover, $(3) \implies (4)$ is clear. For $(4) \implies (1)$, since $M \otimes_R \widehat{R}$ is a flat and Ohm-Rush $\widehat{R}$-module, it is an intersection flat $\widehat{R}$-module by \autoref{thm:ML-SML-ORT-intersection-flat}. Then by descent of intersection flatness (\autoref{cor:pure-descent-IF}) along the faithfully flat map $R \to \widehat{R}$, $M$ is an intersection flat $R$-module.
\end{proof}

\begin{remark}
We do not know if the intersection flatness property is equivalent to the properties of being flat and universally Ohm-Rush over a Noetherian ring that is not local, and, more generally, over an arbitrary commutative ring.
\end{remark}

\section{Future work and Acknowledgements}
\label{sec:futureworkandacknowledgements}

\subsection*{Future Work}
Let $R$ be a Noetherian ring of prime characteristic $p > 0$. Kunz showed \cite{KunzCharacterizationsOfRegularLocalRings} that $R$ is regular if and only if the Frobenius map $F \colon R \to R$ is flat. Viewing the target copy of $R$ as an $R$-module by restriction of scalars along $F$ and denoting the subsequent $R$-algebra by $F_*R$, Hochster and Huneke first studied when $F_*R$ is an Ohm-Rush $R$-module in the local case \cite{HochsterHunekeFRegularityTestElementsBaseChange}. This question was later examined in detail by Sharp \cite{SharpBigTestElements} (see also \cite{KatzmanLyubeznikZhangOnDiscretenessAndRationality,BlickleMustataSmithDiscretenessAndRationalityOfFThresholds, EpsteinShapiroOhmRushI,EpsteinShapiroOhmRushII,EpsteinShapiroOhmRushIII}). We will use the techniques developed in this paper to study the Ohm-Rush, ORT and intersection flatness properties for the $R$-algebra $F_*R$ in a forthcoming paper \cite{DESTtate}. We will prove new cases of intersection flatness of $F_*R$, give a characteristization of when $F_*R$ is an intersection flat $R$-module in the local case in terms of a purity property of the relative Frobenius $F_*R \otimes_R \widehat{R} \to F_*\widehat{R}$, prove openness of pure loci statements for quotients of excellent regular rings of prime characteristic, and we will show that excellent regular rings arising in Tate's approach to rigid analytic geometry \cite{TateRigid} have intersection flat Frobenius. The last result is particularly interesting because these rigid analytic regular rings are not essentially of finite type over an excellent local ring, and so, the current techniques of prime characteristic commutative algebra do not readily apply to the homomorphic images of these rings.

\subsection*{Acknowledgements}
This project branched off from a project that the authors began with Takumi Murayama and Karl Schwede, and was facilitated by a SQuaRE at the American Institute of Mathematics (AIM). The authors thank AIM for providing a supportive and mathematically rich environment.

We are grateful to Takumi and Karl for allowing us to write this standalone paper and have greatly benefited from our conversations with them. We have also benefited from conversations with Karen Smith, Mel Hochster, Yongwei Yao, Gabriel Picavet, Jay Shapiro, Johan de Jong, Alex Perry, Remy van Dobben de Bruyn and Bhargav Bhatt.

\bibliographystyle{skalpha}
\bibliography{main,preprints}

\providecommand{\bysame}{\leavevmode\hbox to3em{\hrulefill}\thinspace}
\providecommand{\MR}{\relax\ifhmode\unskip\space\fi MR}
\providecommand{\MRhref}[2]{%
  \href{http://www.ams.org/mathscinet-getitem?mr=#1}{#2}
}
\providecommand{\href}[2]{#2}
\begin{thebibliography}{DEST25b}

\bibitem[BMS08]{BlickleMustataSmithDiscretenessAndRationalityOfFThresholds}
{\sc M.~Blickle, M.~Musta{\c{t}}{\v{a}}, and K.~E. Smith}: \emph{Discreteness
  and rationality of {$F$}-thresholds}, Michigan Math. J. \textbf{57} (2008),
  43--61, Special volume in honor of Melvin Hochster.

\bibitem[DEST25a]{DESTPhantom}
{\sc R.~Datta, N.~Epstein, K.~Schwede, and K.~Tucker}: \emph{Test ideals in
  quotients of {F}robenius {O}hm-{R}ush rings}, in preparation, 2025.

\bibitem[DEST25b]{DESTtate}
{\sc R.~Datta, N.~Epstein, K.~Schwede, and K.~Tucker}: \emph{Variants on
  {F}robenius intersection flatness and applications to {T}ate algebras}, in
  preparation, 2025.

\bibitem[DM23]{DattaMurayamaTate}
{\sc R.~Datta and T.~Murayama}: \emph{Tate algebras and {F}robenius
  non-splitting of excellent regular rings}, J. Eur. Math. Soc. (JEMS)
  \textbf{25} (2023), no.~11, 4291--4314. {\sf\scriptsize 4662293}

\bibitem[ES16]{EpsteinShapiroOhmRushI}
{\sc N.~Epstein and J.~Shapiro}: \emph{The {O}hm-{R}ush content function}, J.
  Algebra Appl. \textbf{15} (2016), no.~1, 1650009, 14.

\bibitem[ES19]{EpsteinShapiroOhmRushII}
{\sc N.~Epstein and J.~Shapiro}: \emph{The {O}hm-{R}ush content function {II}.
  {N}oetherian rings, valuation domains, and base change}, J. Algebra Appl.
  \textbf{18} (2019), no.~5, 1950100, 23.

\bibitem[ES21]{EpsteinShapiroOhmRushIII}
{\sc N.~Epstein and J.~Shapiro}: \emph{The {O}hm-{R}ush content function {III}:
  completion, globalization, and power-content algebras}, J. Korean Math. Soc.
  \textbf{58} (2021), no.~6, 1311--1325. {\sf\scriptsize 4333195}

\bibitem[Har01]{HaraTestMultiplierIdeal}
{\sc N.~Hara}: \emph{Geometric interpretation of tight closure and test
  ideals}, Trans. Amer. Math. Soc. \textbf{353} (2001), no.~5, 1885--1906.
  {\sf\scriptsize 1813597}

\bibitem[Hoc77]{HochsterCyclicPurity}
{\sc M.~Hochster}: \emph{Cyclic purity versus purity in excellent {N}oetherian
  rings}, Trans. Amer. Math. Soc. \textbf{231} (1977), no.~2, 463--488.

\bibitem[HH90]{HochsterHunekeTCandBrianconSkoda}
{\sc M.~Hochster and C.~Huneke}: \emph{Tight closure, invariant theory, and the
  {B}rian\c{c}on-{S}koda theorem}, J. Amer. Math. Soc. \textbf{3} (1990),
  no.~1, 31--116. {\sf\scriptsize 1017784}

\bibitem[HH94]{HochsterHunekeFRegularityTestElementsBaseChange}
{\sc M.~Hochster and C.~Huneke}: \emph{{$F$}-regularity, test elements, and
  smooth base change}, Trans. Amer. Math. Soc. \textbf{346} (1994), no.~1,
  1--62.

\bibitem[HJ21]{HochsterJeffriesintflatness}
{\sc M.~Hochster and J.~Jeffries}: \emph{Extensions of primes, flatness, and
  intersection flatness}, Commutative {A}lgebra -- 150 {Y}ears with {R}oger and
  {S}ylvia {W}iegand (Providence, RI), Contemp. Math., vol. 773, Amer. Math.
  Soc., 2021, pp.~63--81.

\bibitem[Jen70]{Jensenlimvanishing}
{\sc C.~U. Jensen}: \emph{On the vanishing of {$\varprojlim^{(i)}$}}, J.
  Algebra \textbf{15} (1970), 151--166.

\bibitem[Kat08]{KatzmanParameterTestIdealOfCMRings}
{\sc M.~Katzman}: \emph{Parameter-test-ideals of {C}ohen-{M}acaulay rings},
  Compos. Math. \textbf{144} (2008), no.~4, 933--948.

\bibitem[KLZ09]{KatzmanLyubeznikZhangOnDiscretenessAndRationality}
{\sc M.~Katzman, G.~Lyubeznik, and W.~Zhang}: \emph{On the discreteness and
  rationality of {$F$}-jumping coefficients}, J. Algebra \textbf{322} (2009),
  no.~9, 3238--3247.

\bibitem[Kun69]{KunzCharacterizationsOfRegularLocalRings}
{\sc E.~Kunz}: \emph{Characterizations of regular local rings for
  characteristic {$p$}}, Amer. J. Math. \textbf{91} (1969), 772--784.

\bibitem[Laz69]{Lazard}
{\sc D.~Lazard}: \emph{Autour de la platitude}, Bull. Soc. Math. France
  \textbf{97} (1969), 81--128.

\bibitem[Mat89]{MatsumuraCommutativeRingTheory}
{\sc H.~Matsumura}: \emph{Commutative ring theory}, second ed., Cambridge
  Studies in Advanced Mathematics, vol.~8, Cambridge University Press,
  Cambridge, 1989, Translated from the Japanese by M. Reid.

\bibitem[Nag62]{NagataLocalRings}
{\sc M.~Nagata}: \emph{Local rings}, Interscience Tracts in Pure and Applied
  Mathematics, No. 13, Interscience Publishers (a division of John Wiley \&
  Sons, Inc.), New York-London, 1962. {\sf\scriptsize 0155856}

\bibitem[OR72]{OhmRu-content}
{\sc J.~Ohm and D.~E. Rush}: \emph{Content modules and algebras}, Math. Scand.
  \textbf{31} (1972), 49--68.

\bibitem[Pic85]{Pic-contenu}
{\sc G.~Picavet}: \emph{Propri\'et\'es et applications de la notion de
  contenu}, Comm. Algebra \textbf{13} (1985), no.~10, 2231--2265.

\bibitem[Pic21]{Picemail}
{\sc G.~Picavet}, private communication, 2021.

\bibitem[RG71]{rg71}
{\sc M.~Raynaud and L.~Gruson}: \emph{Crit\`eres de platitude et de
  projectivit\'e. {T}echniques de ``platification'' d'un module}, Invent. Math.
  \textbf{13} (1971), 1--89.

\bibitem[Sha10]{SharpTestElementsforFpure}
{\sc R.~Y. Sharp}: \emph{An excellent {$F$}-pure ring of prime characteristic
  has a big tight closure test element}, Trans. Amer. Math. Soc. \textbf{362}
  (2010), no.~10, 5455--5481. {\sf\scriptsize 2657687}

\bibitem[Sha12]{SharpBigTestElements}
{\sc R.~Y. Sharp}: \emph{Big tight closure test elements for some non-reduced
  excellent rings}, J. Algebra \textbf{349} (2012), 284--316.

\bibitem[Smi00]{SmithTestMultiplierIdeal}
{\sc K.~E. Smith}: \emph{The multiplier ideal is a universal test ideal}, Comm.
  Algebra \textbf{28} (2000), no.~12, 5915--5929, Special issue in honor of
  Robin Hartshorne.

\bibitem[{Sta}]{stacks-project}
{\sc T.~{Stacks Project Authors}}: \emph{Stacks project},
  \url{http://stacks.math.columbia.edu}.

\bibitem[Tat71]{TateRigid}
{\sc J.~Tate}: \emph{Rigid analytic spaces}, Invent. Math. \textbf{12} (1971),
  257--289. {\sf\scriptsize 306196}

\bibitem[War69]{Warfieldpure}
{\sc R.~B. Warfield, Jr.}: \emph{Purity and algebraic compactness for modules},
  Pacific J. Math. \textbf{28} (1969), 699--719.

\end{thebibliography}

\end{document}